\documentclass[reqno]{amsart}
\usepackage{amssymb,amsthm,amsmath,graphicx,subfig,pinlabel}
\pdfoutput=1
\usepackage{xr-hyper}
\usepackage[colorlinks]{hyperref}
\usepackage[all]{hypcap}
\usepackage{todonotes}
\DeclareMathOperator{\crit}{crit}
\DeclareMathOperator{\diff}{Diff}

\DeclareMathOperator{\Ker}{Ker}
\DeclareMathOperator{\Mod}{Mod}
\DeclareMathOperator{\id}{id}
\newcommand{\co}{\colon\thinspace}

\newcommand{\R}{\mathbb{R}}
\newcommand{\Z}{\mathbb{Z}}
\newcommand{\N}{\mathbb{N}}
\newcommand{\CP}{\mathbb{CP}{}^{2}}
\newcommand{\CPb}{\overline{\mathbb{CP}}{}^{2}}
\newcommand{\MYhref}[3][blue]{\href{#2}{\color{#1}{#3}}}

\theoremstyle{definition}
\newtheorem{thm}{Theorem}[section]
\newtheorem{lemma}[thm]{Lemma}
\newtheorem{sublemma}[thm]{Sublemma}
\newtheorem{ex}[thm]{Example}
\newtheorem{rmk}[thm]{Remark}
\newtheorem{prop}[thm]{Proposition}
\newtheorem{cor}[thm]{Corollary}
\newtheorem{df}[thm]{Definition}

\title{Uniqueness of surface diagrams of smooth 4-manifolds}
\author[Jonathan D. Williams]{Jonathan D. Williams} 
\address{Department of Mathematical Sciences, Binghamton University}
\email{\MYhref{mailto:jdw.math@gmail.com}{jdw.math@gmail.com}}

\begin{document}
\begin{abstract}Previously the author presented a new way to specify any smooth, closed oriented 4-manifold by an orientable surface decorated with simple closed curves. These curves are cyclically indexed, and each curve has a unique transverse intersection with the next. These are called surface diagrams, and each comes from a certain type of map from the 4-manifold to the 2-sphere. Given such a map, the corresponding surface diagram is defined up to diffeomorphism and isotopy of individual curves. The aim of this paper is to give a uniqueness theorem stating that, for maps within any fixed homotopy class, surface diagrams are unique up to four moves: stabilization, handleslide, multislide, and shift.\end{abstract}
\maketitle
\tableofcontents
\section{Introduction}\label{introduction}
This decidedly old-school paper is concerned with a certain class of maps from an arbitrary smooth closed orientable 4-manifold $M$ to the 2-sphere, introduced as \emph{simplified purely wrinkled fibrations} in \cite{W1} and more briefly called \emph{simple wrinkled fibrations} in \cite{BH}. Under mild conditions, such a map defines a \emph{surface diagram} $(\Sigma,\Gamma)$ of $M$, in which $\Sigma$ is a closed, orientable surface decorated with a collection $\Gamma$ of simple closed curves. A surface diagram specifies $M$ up to diffeomorphism. Section \ref{moves} describes four moves on surface diagrams for a fixed 4-manifold and explains how these moves come from certain model homotopies of the fibration map. Section \ref{proof} proves that, for a fixed homotopy class of maps $M\to S^2$, any pair of simplified purely wrinkled fibrations are the endpoints of some sequence of these model homotopies. Though the moves of \cite{W1} are sufficient to relate even maps in different homotopy classes, it is an open question to determine the diagrammatic relation between surface diagrams coming from different homotopy classes of maps $M\to S^2$ (see \cite{KMT} for discussion of the collection of homotopy classes for a fixed 4-manifold). Here is the main result of the paper.
\begin{thm}\label{T}Suppose $\alpha_0,\alpha_1\co M\to S^2$ are two homotopic simplified purely wrinkled fibrations with surface diagrams $(\Sigma^i,\Gamma^i)$, $i=0,1$. Then there is a finite sequence of surface diagrams, beginning with $(\Sigma^0,\Gamma^0)$ and ending with $(\Sigma^1,\Gamma^1)$, obtained by performing stabilizations, handleslides, shifts, multislides, and their inverses.\end{thm}
The remainder of this section gives background on the origins of this work, the fundamentals of surface diagrams and the maps that induce them, and an overview of the proof of Theorem \ref{T}.
 \subsection{Background}
The subject of simplified purely wrinkled fibrations arose from the study of broken Lefschetz fibrations on smooth 4-manifolds, which were first introduced in \cite{ADK} to generalize the correspondence between Lefschetz fibrations and symplectic 4-manifolds up to blowup (see \cite{GS} for further details on symplectic structures and Lefschetz fibrations). Roughly, the result of \cite{ADK} was that if $(M,\omega)$ is a \emph{near-symplectic} 4-manifold, then a suitable blowup of $(M,\omega)$ has a smooth map to the 2-sphere, called a \emph{broken Lefschetz fibration}, such that $\omega$ restricts to a volume form on any fiber that does not intersect its vanishing locus. The vanishing locus of a near-symplectic form is always a smooth 1-submanifold of $M$. Discussed below, broken Lefschetz fibrations are a generalization of Lefschetz fibrations in which the critical locus is allowed to contain a 1-submanifold of critical points with a particular local model (see Equation \ref{eq:fold1}); when the fibration corresponds to a near-symplectic form $\omega$, this critical 1-submanifold coincides with its vanishing locus; see for example \cite{B1,L1}. Though mild, this generalization greatly increases the collection of 4-manifolds that admit such a fibration structure: it is known that every smooth, orientable 4-manifold admits a broken Lefschetz fibration (though it may not correspond to any near-symplectic form); see for example \cite{AK,B2,L1}.

The study of broken Lefschetz fibrations is partly motivated by the effort to understand the Seiberg-Witten invariants of a smooth 4-manifold $M$ geometrically. When $b^{2+}$ is positive, the Seiberg-Witten invariants at their most basic level define a map from $H^2(M;\Z)$ to the integers, defined as an algebraic count of solutions to a nonlinear elliptic pair of partial differential equations on $M$ \cite{M}. In a 1996 paper, Taubes showed that, for symplectic 4-manifolds, solutions to the Seiberg-Witten equations correspond to pseudoholomorphic curves which contribute to a special Gromov invariant he defined, called $Gr$. Pseudoholomorphic curves are submanifolds of two real dimensions, possibly with singularites, that are singled out by the chosen symplectic structure and other auxiliary data (these are choices which are later shown to not affect the values of $Gr$) \cite{T1}.

Related efforts have revolved around equipping a given manifold with structures that resemble surface bundles, namely Morse functions in three dimensions and (broken) Lefschetz fibrations in four dimensions, then considering associated spaces of pseudoholomorphic curves. In 2003 Donaldson and Smith defined a \emph{standard surface count} for symplectic 4-manifolds, counting pseudoholomorphic curves which are sections of a fiber bundle associated to a Lefschetz fibration \cite{DS}. Soon after, Usher showed that the standard surface count is equivalent to $Gr$, and thus the Seiberg-Witten invariant \cite{U}. Though this equivalence is only known to hold for a suitably chosen Lefschetz fibration whose existence is equivalent to the existence of a symplectic form, it gave a promising inroad to generalization: broken Lefschetz fibrations offer a way to continue their approach into nonsymplectic territory, further guided by an existence result of \cite{T2} stating that if a near-symplectic 4-manifold has nonvanishing Seiberg-Witten invariant, then there is a pseudoholomorphic curve with boundary given by the vanishing locus of $\omega$. In 2007, Perutz defined a generalization of the standard surface count for near-symplectic broken Lefschetz fibrations, called the \emph{Lagrangian matching invariant}, that fits in nicely with natural expectations in numerous ways. For example, there are formal similarities with ideas in \cite{T2}, vanishing theorems and calculations in special cases that coincide with those for the Seiberg-Witten invariant. However, it remains to show that it is a smooth invariant of the underlying 4-manifold (not just the isomorphism class of its chosen fibration structure), and it is generally difficult to compute. A large amount of the motivation for studying surface diagrams is the possibility that they will prove useful in addressing such issues; see for example \cite{W3}.

\subsection{Critical points}\label{criticalpoints}
Let $M$ be a smooth closed 4-manifold. Any generic smooth map $M\to S^2$ resembles a surface bundle, except one allows the presence of a one-dimensional critical locus that is a union of \emph{fold} and \emph{cusp} points. The fold locus is an embedded smooth 1-submanifold of $M$ with the following local model at each point: \begin{equation}\label{eq:fold1}(x_1,x_2,x_3,x_4)\mapsto(x_1,x_2^2+x_3^2\pm x_4^2).\end{equation} When the sign above is negative, it is known variously as an \emph{indefinite fold}, \emph{round singularity}, or \emph{broken singularity} depending on context. Figure \ref{fold} is a picture of the target space, schematically depicting the fibration structure. In that figure, like the one to its right, what appears is the target disk of a map $D^4\to D^2$, with bold arcs representing the image of the critical locus. A surface is pictured in the region of regular values that have that surface as their preimage, and tracing point preimages above a horizontal arc from left to right gives the foliation of $\R^3$ by hyperboloids, first one-sheeted, then two-sheeted, with a double cone above the fold point. The circle (drawn on the cylinder to the left) that shrinks to the cone point is called the \emph{round vanishing cycle} for that arc in the context of broken Lefschetz fibrations. When dealing exclusively with purely wrinkled fibrations as this paper does, there are no Lefschetz vanishing cycles, so the term \emph{vanishing cycle} will suffice.

Each cusp point is a common endpoint of two open arcs of fold points, with the local model \begin{equation}\label{eq:cusp}(x_1,x_2,x_3,x_4) \mapsto (x_1,x_2^3-3x_1x_2+x_3^2\pm x_4^2).\end{equation} When the sign above is negative, it is called an \emph{indefinite cusp} and is adjacent to two indefinite fold arcs as in Figure \hyperlink{1b}{1b}: here two fold arcs meet at a cusp point, for which the two vanishing cycles must transversely intersect at a unique point in the fiber.

\begin{figure}\capstart\hypertarget{1b}%
	\centering
	\subfloat[Fold.]{\label{fold}\includegraphics{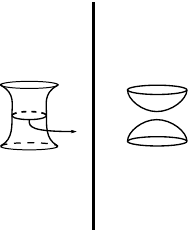}} \qquad \qquad \qquad \capstart
	\subfloat[Cusp.]{\label{cusp}\includegraphics{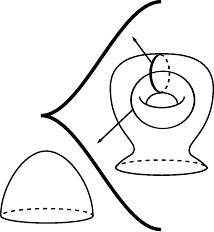}}
	\caption{\label{critpts}Critical points of purely wrinkled fibrations.}
\end{figure}

\subsection{Deformations of purely wrinkled fibrations}\label{defmovessection}
This paper is focused on one-parameter families of maps $M\to S^2$ that fail to be purely wrinkled fibrations at isolated points that have local models called \emph{birth, merge} and \emph{flip}. Using the same visual language of Figure \ref{critpts}, their effects on a purely wrinkled fibration appear in Figure \ref{defmoves}. Here follow the definitions from \cite[Section 2.4]{W1}.
\begin{figure}\capstart%
	\centering
	\subfloat[Birth.]{\label{movesbirth}\includegraphics{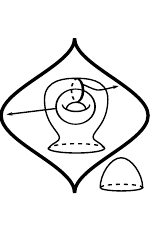}}\qquad \qquad
	\subfloat[Merge.]{\label{movesmerge}\includegraphics{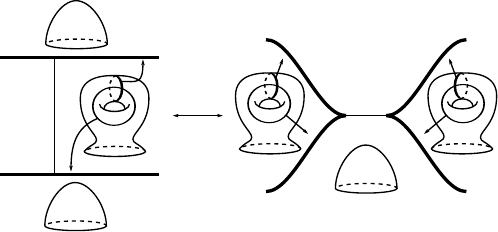}}\\ \capstart
	\subfloat[Flip.]{\label{movesflip}\includegraphics{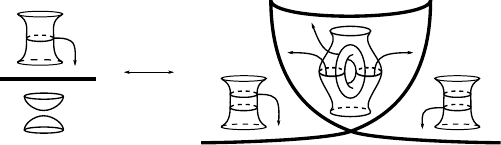}}
	\caption{\label{defmoves}Deformations of purely wrinkled fibrations.}
\end{figure}
\begin{df}The \emph{birth} move begins with a submersion and introduces a circle of critical points. It has the local model \[(t,x_1,x_2,x_3,x_4)\rightarrow(t,x_1,x_2^3+3(x_1^2-t)x_2+x_3^2-x_4^2).\]The \emph{merge} move results in the connect sum between two critical arcs, and is given by \[(t,x_1,x_2,x_3,x_4)\mapsto(t,x_1,x_2^3+3(t-x_1^2)x_2+x_3^2-x_4^2).\] The \emph{flip} move introduces a double point in the critical image, and is given by \[(t,x_1,x_2,x_3,x_4)\mapsto(t,x_1,x_2^4+x_2^2t+x_1x_2+ x_3^2-x_4^2).\]\end{df}
The conditions under which these moves and their inverses can be performed are spelled out in \cite{L1} and \cite[Section 2.4]{W1}, where these local models were introduced as moves on purely wrinkled fibrations. An additional type of homotopy, called \emph{isotopy} in \cite{L1} and \emph{2-parameter crossing} in this paper and \cite{GK1}, is one in which the critical image moves around in the base after the fashion of the Reidemeister moves. These are considered in detail in \cite{W2}.
\begin{df}For purely wrinkled fibrations $\alpha_0,\alpha_1$, a \emph{deformation} $\alpha=\alpha_t$, $t\in[0,1]$ is a homotopy realized by a sequence of births, merges, flips and 2-parameter crossings.\end{df} 

\subsection{Surface diagrams}\label{surfacediagrams}
For the purposes of this paper, a broken Lefschetz fibration is simply a smooth map from a 4-manifold $M$ to a surface $F$ (usually the sphere or the disk) whose critical locus is a collection of Lefschetz critical points and indefinite folds, and a purely wrinkled fibration is a stable map $M\to F$ whose critical locus is free of definite folds. In other words, the critical set of a purely wrinkled fibration is a union of indefinite cusps and indefinite folds. For details on stability and the critical loci of stable maps, see \cite{W2}, and perhaps more importantly \cite{L1}, which relates the results of \cite{Wa} concerning stability of families of real-valued functions to the world of purely wrinkled fibrations. By Corollary 1 of \cite{W1}, every broken Lefschetz fibration can be modified by a (possibly long) sequence of moves (which are chosen from a short list initially appearing in \cite{L1}) into a purely wrinkled fibration whose critical locus is one circle embedded by the fibration map into $S^2$; combining terminology from \cite{B1} and \cite{L1}, such a map is called a \emph{simplified purely wrinkled fibration} (in fact, the proof of that corollary implies this can be done for \emph{any} continuous $M\to S^2$, for example a constant map). Such a map $f\co M\to S^2$ has an orientable genus $g$ surface as regular fiber over one of the disks that comprise $S^2\setminus f(\crit f)$, and the regular fiber over the other disk has genus $g-1$; choosing a regular value $p$ in the higher-genus side and reference arcs from $p$ to the various indefinite fold arcs, one obtains a collection of simple closed curves in $\Sigma_g$ by recording the round vanishing cycles in a counter-clockwise direction going around the cusped critical circle. For this reason, the curves are relatively indexed by $\Z/k\Z$, where $k$ is the number of fold arcs.
\begin{df}Assuming $g\geq3$, a \emph{surface diagram} $(\Sigma_g,\Gamma)$ of $M$ is the higher-genus fiber $\Sigma_g$ of an SPWF, decorated with the relatively $\Z/k\Z$-indexed collection $\Gamma=(\gamma_1,\ldots,\gamma_k)$ of round vanishing cycles, up to orientation-preserving diffeomorphism.\end{df}
An observation by Paul Melvin is that composing with the antipodal map reverses the cyclic ordering of the vanishing cycles as measured in this way, while preserving $M$. For a fixed orientation of $M$, this reverses the preimage orientation of $\Sigma$.
\begin{rmk}
In addition to the stabilization issue in the previous remark, the requirement that $g\geq3$ (which can always be satisfied by applying the modification discussed in Section \ref{stabmove} below) is necessary for the following related reason: After using $(\Sigma_g,\Gamma)$ to form the fibration consisting of the preimage of a neighborhood of the critical image, if $g=1$ or $g=2$ there are various ways to close off the higher- or lower-genus sides with a copy of $\Sigma_{g-1}\times D^2$, according to the elements of $\pi_1(\diff(S^2))\cong\Z_2$ and $\pi_1(\diff(T^2))\cong\Z^2$, where $\diff(\Sigma)$ is the group of orientation-preserving diffeomorphisms of $\Sigma$ (see \cite{ADK,B1,H} for explicit examples). Since $\diff(\Sigma_{g})$ is simply connected for $g\geq2$, a surface diagram as defined specifies the total space of the fibration up to orientation-preserving diffeomorphism.\end{rmk}
\begin{rmk}\label{phic} The local model for the cusp requires the two round vanishing cycles to transversely intersect at a unique point in the fiber; for this reason, consecutive elements of $\Gamma$ must intersect in this way. Also, each element of $\Gamma$ must be an embedded circle in $\Sigma_g$. These two facts could be called an \emph{intersection condition} for surface diagrams, and it is natural to wonder if $(\Sigma,\Gamma)$ is a surface diagram if it satisfies the intersection condition. Unfortunately this is too much to ask: This remark discusses the conditions under which $(\Sigma,\Gamma)$ specifies a smooth 4-manifold. Just as the vanishing cycles of a Lefschetz fibration over the sphere must yield a trivial monodromy map, there is an associated monodromy map for $(\Sigma,\Gamma)$ that is trivial if and only if the pair specifies a closed 4-manifold.

Given the pair $(\Sigma,\Gamma)$ satisfying the intersection condition, construct an SPWF over the closed unit disk (like Figure \ref{spwf}) in which the vanishing cycles read $\gamma_1,\ldots,\gamma_k$ in a counter-clockwise direction. Then perform unsinking moves on all the cusps (\cite[Figure 1]{L1}) to obtain a broken Lefschetz fibration over the disk whose round vanishing cycle is $\gamma_1$ and whose Lefschetz vanishing cycles read $c_i=t_{\gamma_i}(\gamma_{i+1})$, $i\in\Z/k\Z$, so that $t_{c_i}$ is the unique positive Dehn twist sending $\gamma_i$ to $\gamma_{i+1}$ (see \cite{L1,W2} about sinking and unsinking). Let $\Sigma'$ be the surface with two marked points $(p,q)$ obtained from $\Sigma$ by replacing a tubular neighborhood of $\gamma_1$ with two disks. The mapping class $\mu=[t_{c_k}\circ\cdots\circ t_{c_1}]\in MCG(\Sigma)$ of these Dehn twists preserves the isotopy class of the unoriented circle $\gamma_1$ by construction, and so one could interpret $\mu$ as an element of $MCG(\Sigma')(p,q)$, the group of orientation-preserving diffeomorphisms of $\Sigma'$ that preserve its marked points as a set, and then forget about the marked points, taking it as an element of $MCG(\Sigma')$. Denoting the subgroup of elements that fix the unoriented circle $\gamma_1$ up to isotopy by $MCG(\Sigma)(\gamma_1)$, this amounts to a homomorphism \[\Phi_{\gamma_1}\co MCG(\Sigma)(\gamma_1)\to MCG(\Sigma')(p,q)\to MCG(\Sigma')\] that was first hinted at in \cite{ADK}, mentioned explicitly in \cite{B1} and then studied in more detail in \cite{H,Be}. It is known that the fibration of Figure \ref{spwf} can be capped off by a copy of $\Sigma'\times D^2$ to obtain an SPWF exactly when $\mu\in\ker\left(\Phi_{\gamma_1}\right)$. This offers a meaningful contrast between surface diagrams and Kirby diagrams of 4-manifolds: To check that the handlebody given by a framed and dotted link in $S^3$ closes off with 3- and 4-handles can be a difficult problem, because one must exhibit a diffeomorphism between the boundary of that handlebody and $\#_n(S^1\times S^2)$ for some $n$. The corresponding task for surface diagrams is tedious at worst and large parts of it can be done by a computer, using Alexander's method (see \cite{FM}) to check that $\mu\in\ker\left(\Phi_{\gamma_1}\right)$.
\begin{figure}\capstart
	\centering{\includegraphics{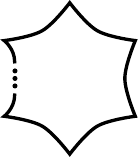}}
	\caption{A wrinkled fibration over the disk. The higher genus side of the fold is the inside of the circle.}\label{spwf}
\end{figure}\end{rmk}
Section \ref{moves} contains a list of discrete moves that may be used to modify any surface diagram. These moves preserve the diffeomorphism type of the 4-manifold specified by the surface diagram because they come from the endpoints of \emph{deformations} of the fibration map. By the main result of \cite{W1} (and also \cite{GK1} with the added value of a connected fibers result for deformations between maps with connected fibers), one may modify any homotopy $\alpha$, whose endpoints are required to be purely wrinkled fibrations, to be a deformation. These were first described in \cite{L1}, using the term \emph{deformation of purely wrinkled fibrations}. Such maps are the main focus of the paper.

To keep the distinction clear, the terms \emph{handleslide}, \emph{stabilization}, etc. will always refer to the move on surface diagrams, while the terms \emph{handleslide deformation} and \emph{stabilization deformation} will refer to the corresponding homotopies of fibration maps. Section \ref{moves} describes the so-called \emph{model deformations} that induce the moves on surface diagrams. This is important important for two reasons. First, the correspondence between model deformations and moves on surface diagrams shows that performing a move does not change the diffeomorphism type of the 4-manifold described by the diagram, because each move is shown to be the result of deforming the map that specifies the original diagram. Second, the proof of Theorem \ref{T} does not make use of the moves themselves: The argument is that there exists a way to change a deformation into a sequence of model deformations, which are linked to diagrammatic moves in Section \ref{moves}.

\begin{rmk}[Surgered surface diagrams]\label{ssdrmk}For each of the moves in Theorem \ref{T}, the highest-genus region of a base diagram momentarily becomes disconnected. Initially, reference fibers chosen over various points in the higher-genus region are decorated with identical sets of vanishing cycles up to isotopy. Depending on what goes on while there are two higher-genus components, there can be some variation in how the reference fibers over points in each component become identified when the components reunite, and one way to keep track of how these fibers are identified is to use a common reference fiber in a lower genus region. This viewpoint is used to great effect in \cite{BH,H}.

Consider in the initial SPWF a reference fiber $\Sigma_g$ over a point just to the lower-genus side of a fold arc $\gamma$, along with a short \emph{reference path} across $\gamma$ into its higher-genus side. As one traces fibers above this path, two points $p,p'$ become identified in a surgery that increases the genus of the fiber by one as in Figure \ref{fold}. Then the endpoint of this path gives a reference fiber $\Sigma_{g+1}$ to which one can add the vanishing cycles for folds bounding that higher-genus region. Pulling this picture back across $\gamma$, all the vanishing cycles descend to circles in $\Sigma_g$ except those that intersect the vanishing cycle of $\gamma$; these instead appear as \emph{vanishing arcs} whose endpoints are $p$ and $p'$. The ends of these arcs come equipped with a bijection $\beta$ between those at $p$ and those at $p'$ by the way they pair up on either side of the vanishing cycle of $\gamma$ in $\Sigma_{g+1}$. Further, $\beta$ must reverse the order of the ends going around $p$ and $p'$ in the sense that if the ends $e_1,\ldots,e_n$ are numbered clockwise around $p$, then the ends $\beta(e_1),\ldots,\beta(e_n)$ are numbered counter-clockwise going around $p'$. For this reason, the union of vanishing arcs obtained by picking a point on one vanishing arc, following it to one end $e$, continuing at $\beta(e)$ and on until returning to the chosen point will be called a vanishing cycle just like the other simple closed curves. In this way, the surface diagram can be recovered from a reference fiber in the lower-genus region, along with a chosen path into the higher-genus region.

Taking a family of such reference fibers and paths (with fixed endpoints) that travels along the critical circle of an SPWF yields a one-parameter family of such diagrams, and this paragraph describes how such a family evolves when a reference path passes a cusp (it is helpful to imagine the family of arcs in Figure \hyperlink{1b}{1b}). As the family of paths approaches a cusp, there is a vanishing arc $\nu$ whose ends correspond under $\beta$ coming from the vanishing cycle just past the cusp. This arc shrinks to a point where $p$ and $p'$ momentarily meet, so that the diagram whose path intersects the cusp itself has only one distinguished point. Passing the cusp, $p$ and $p'$ separate again. Those vanishing cycles that intersected $\nu$ form the new collection of vanishing arcs, while those ends that came from vanishing cycles disjoint from $\nu$ (and were brought together with $p$ and $p'$) remain identified, becoming arcs that pass between $p$ and $p'$, intersecting $\gamma$, which appears as a short arc running between $p$ and $p'$.\end{rmk}
 
 \subsection{Overview of the argument}\label{overview}
The first step is to show that the moves on surface diagrams correspond to particular model deformations. Section \ref{moves} describes the moves and supplies these deformations along with arguments linking them to the moves. Importantly, each model deformation (call a given one $\alpha$) has a specific pattern for three aspects of its critical locus:\begin{itemize}\item the Morse function $T$ given by projecting $\crit\alpha$ to the $t$ parameter,\item the restriction of $T$ to the cusp locus of $\alpha$, \item the stratified immersion $\crit\alpha\to[0,1]\times S^2$.\end{itemize} The map $\alpha|_{\crit\alpha}$ is a \emph{stratified} immersion because it has a nonempty critical locus consisting of the cusp and swallowtail points, while its restriction to each of these strata is an immersion. 

The strategy to prove Theorem \ref{T} stems from the observation that if the three aspects of a deformation above are shared with that of a model deformation, then it is a model deformation. The strategy is to modify $\alpha\co M_{[0,1]}\to S^2_{[0,1]}$ until there is a partition of $[0,1]$ for which $\alpha$ follows one of the four patterns on each member.

After the sections that outline the model deformations, Sections \ref{notation}--\ref{switching} collect notation and some tools that appear in the proof of the main Theorem \ref{T}. The required modification of $\alpha$ appears in Sections \ref{mod} and Section \ref{immer}. 

Section \ref{mod} addresses the embedded surface $\crit\alpha$ itself in a series of lemmas to the effect that one may assume all fibers are connected with genus at least two, $\crit\alpha$ is connected, has at most two components at any value of $t$, and that $T$ has no canceling Morse critical points. Trading precision for intuition, one could say that applying these lemmas results in a one-parameter family of maps whose critical circle remains connected, except sometimes it splits in two for a short interval in $t$. Going further, the last of these lemmas, Lemma \ref{genus}, uses tools from Section \ref{splice} to modify $\crit\alpha$ around each interval where the circle splits to distinguish each splitting as coming from a multislide or a shift deformation.
\begin{figure}\capstart
	\labellist
	\small\hair 2pt
	\pinlabel handleslide at 82 8
	\pinlabel stabilization at 155 8
	\pinlabel multislide at 245.5 8
	\pinlabel shift at 324 8
	\endlabellist
	\centering{\includegraphics[width=\linewidth]{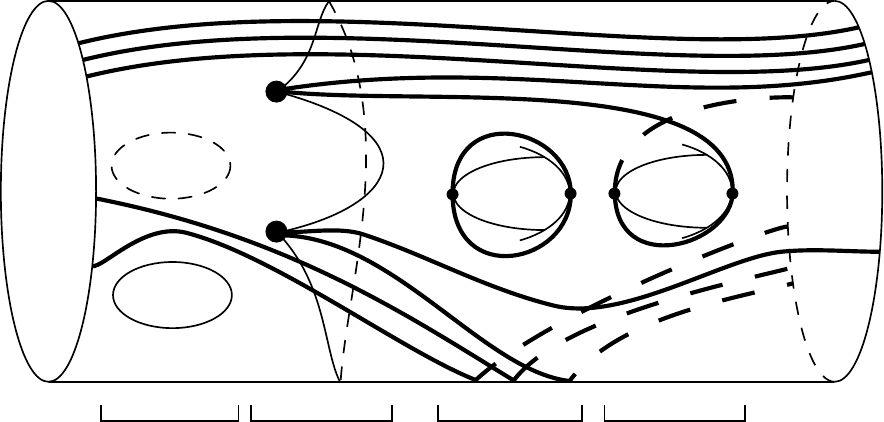}}
	\caption{The decorated critical surface of a deformation. The decorations here correspond to the four moves of Theorem \ref{T}.}\label{clean}
\end{figure}
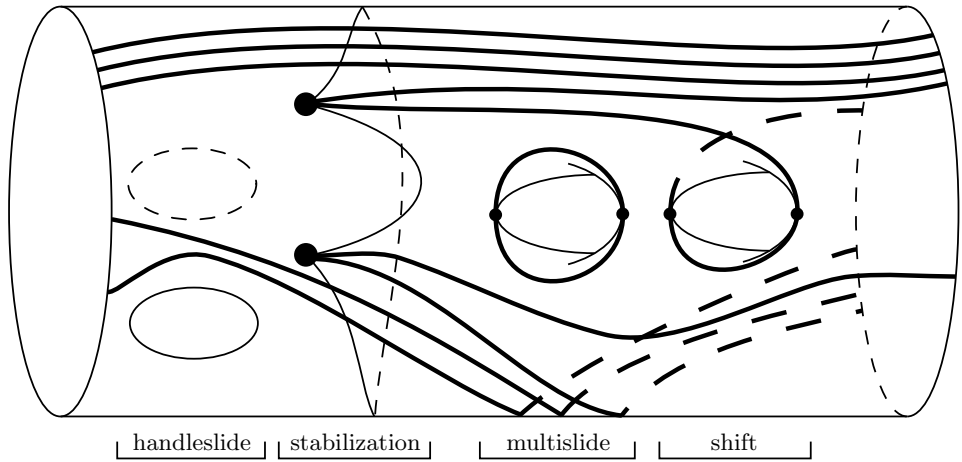
Assuming $\crit\alpha$ has undergone the above modifications, Section \ref{immer} modifies the stratified immersion $\crit\alpha\to[0,1]\times S^2$. Decorating the critical surface with so-called \emph{immersion arcs}, there are several cases one must examine in order to ensure that they precisely follow the prescriptions given by the model deformations. What remains is the required deformation. To give some idea of how these decorations appear, Figure \ref{clean} depicts $\crit\alpha$, where $\alpha$ is a sequence of the model deformations for handleslide, stabilization, multislide, and finally shift (each is underlined below the figure). The $t$ coordinate increases to the right, and the critical points of $T$ are marked out by small dots while swallowtail points (where flips occur) are large dots. The bold arcs are swept out by cusp points while the other arcs are mapped to each other in pairs by $\alpha$. All these phenomena are discussed in more detail below. One quirk of these figures is that each cylinder is cut out of $\crit\alpha$ at an angle and drawn with a slight bend to eliminate artifacts of parallax: It looks like a straight right cylinder on the paper, yet each pair of points in each slice $\alpha_t$ appears vertically aligned. This makes it easier to identify which arcs have the same image in $S^2_I$ using the \emph{simultaneity condition} defined below.

\begin{figure}\capstart
	\centering{\includegraphics[width=\linewidth]{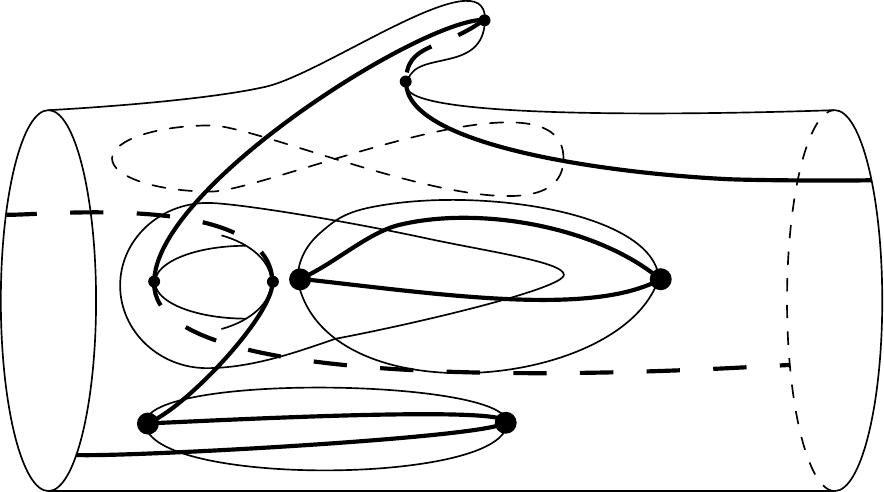}}
	\caption{A sampling of what can go wrong in a deformation.}\label{dirty}
\end{figure}
Figure \ref{dirty} exhibits some of the behavior that is ruled out for sequences of model deformations. Two highlights are the existence of a canceling pair of $T$-critical points at the top (Lemma \ref{cancelmorsecp}) and an intersection triple of immersion arcs. In the absence of cusps, such an intersection appears as a Reidemeister-III fold crossing in $S^2$, and there are no instances of such a move in the model deformations. Another issue in Figure \ref{dirty} is that none of the swallowtail points occur as part of a stabilization as in Figure \ref{clean}. The remedy for such a deformation would be to apply Lemma \ref{cancelmorsecp} for the canceling critical points (see Figure \ref{bi} for the heart of the argument), and to include the swallowtails into stabilization deformations that likely reverse later in the deformation. As for the intersection triple, Section \ref{immer} contains lemmas that allow one to include the arcs into handleslide or stabilization deformations. One of the more subtle parts of the paper is concerned with the immersion arc that encircles the pair of index one critical points of $T$: Section \ref{unlink} is dedicated to eliminating this phenomenon.
\subsubsection*{Acknowledgments}\ The author would like to thank Denis Auroux, David Gay, Kenta Hayano and Rob Kirby for their exceedingly helpful input during the preparation of this work, which was partly supported by NSF grant DMS-0635607. The author would also like to thank the reviewer(s) for their extraordinary fortitude.  

\section{Moves and model deformations}\label{moves}
\subsection{Handleslide}
\subsubsection{Handleslide move}\label{hsmove}
The first move comes from an application of Lemma 3.8 and Theorem 3.9 of \cite{H}. To perform a handleslide on $(\Sigma,\Gamma)$, the first step is to choose a pair of disjoint vanishing cycles, say $\gamma_1$ and $\gamma_n$, such that the result of surgery on $\gamma_1\cup\gamma_n$ (replacing a neighborhood of each of $\gamma_1$ and $\gamma_n$ with a pair of disks) is a connected surface of genus at least 2. Next, choose a product of three Dehn twists $\phi=t_{\gamma_1}^{-1}\circ t_{\gamma_n}^{-1}\circ t_{c}\in\diff\Sigma$, where $c$ is constructed immediately below. Then the new surface diagram is $(\Sigma,\Gamma')$, where \[\Gamma'=(\gamma_1,\phi(\gamma_2),\ldots,\phi(\gamma_{n-1}),\gamma_n,\ldots,\gamma_k).\]Since the move can only affect $\gamma_2,\ldots,\gamma_{n-1}$, these are called the \emph{nonstationary set} and the rest comprise the \emph{stationary set}. These terms will also be useful for the multislide and shift moves.

The curve $c$ comes from the following construction. First, replace neighborhoods of $\gamma_1$ and $\gamma_n$ with pairs of marked disks $(D_i,p_i),\ (D_i',p_i')$, $i\in\{1,n\}$, and let \[h\co\Sigma_{g-1}\setminus\{p_1,p_1',p_n,p_n'\}\to\Sigma_g\setminus(\gamma_1\cup\gamma_n)\] be the resulting diffeomorphism relating the surfaces on the left and right sides of Figure \ref{c} (minus their respective subsets). Then connect one of $\{p_1,p_1'\}$ to one of $\{p_n,p_n'\}$ by any embedded arc $\tilde{c}$ (an example of such an arc is on the left side of Figure \ref{c}) and let $N$ be the closure of a neighborhood of $\tilde{c}$ containing the disks that contain the endpoints of $\tilde{c}$. The boundary of $N$ appears in  the middle of Figure \ref{c}. Then $c=h(\partial N)$ is a simple closed curve in $\Sigma_g\setminus(\gamma_1\cup\gamma_n)$.
\begin{figure}\capstart%
	\centering\capstart	
	\subfloat[Constructing a curve $c$ for the handleslide move. In the right side, $c$ is the middle curve, while $\gamma_1$ and $\gamma_n$ are interchangeably the other two.]{\label{c}\includegraphics[width=\linewidth]{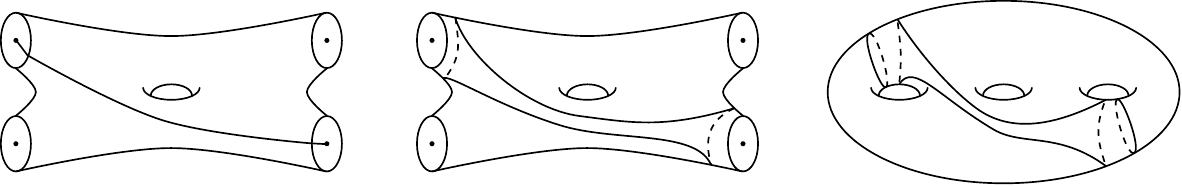}}	\\
	\subfloat[An example of how the choices above affect the curve on the left by a sequence of connect sums with $\gamma_1$ and $\gamma_n$.]{\label{hsex}\includegraphics[width=\linewidth]{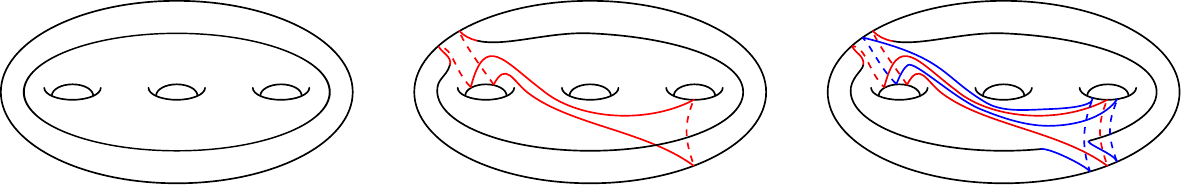}}
	\caption{}
\end{figure}

Section \ref{hsmove} ends with a short digression on why this is called a handleslide move. In a Heegaard diagram for a 3-manifold, the attaching circles can be considered vanishing cycles for critical points of an associated Morse function; see e.g. \cite{L2} for an application of this viewpoint. The deformation underlying the handleslide move is similar to that of the 3-dimensional handleslide in the sense that both can be seen as inducing push maps in surgered diagrams: In the 3-dimensional case, one chooses whether the handleslide occurs among $\alpha$ or $\beta$ circles, or in the language of this paper which circles form the nonstationary set (suppose it is the $\alpha$ circles). Then one replaces, say, $\alpha_1$ with two marked disks (that is, one moves the reference fiber just past the critical point that gave $\alpha_1$) and applies a push map to the other $\alpha$ circles. Each time one of these circles crosses a marked point, it is recorded in the Heegaard diagram as a handleslide over $\alpha_1$. Indeed, the map $\phi$ above is constructed as such a lift of a push map in the proof of \cite[Lemma 3.8]{H}. As an example, Figure \ref{hsex} shows how the nonstationary curve $\gamma_i$ undergoes a sequence of connect sums with parallel copies of $\gamma_1$ or $\gamma_n$ according to the choices made in Figure \ref{c}. Another way to check that any handleslide modifies a nonstationary curve by a sequence of connect sums with $\gamma_1$ and $\gamma_n$ is to perform the relevant Dehn twists on short arcs representing intersections of $h^{-1}(\gamma_i)$ with $\tilde{c}$ (each gives two intersections between $\gamma_i$ and $c$), or of $\gamma_i$ with $\gamma_1$ or $\gamma_n$ (each of which gives one intersection of $\gamma_i$ with $c$ and one intersection of $\gamma_i$ with $\gamma_1$ or $\gamma_n$).
 
 \subsubsection{Handleslide deformation}
The model deformation for the handleslide move is the \emph{$R_2$-move} studied in \cite{H}.
\begin{df}\label{hsdefdet}A \emph{local handleslide deformation} consists of a consecutive pair of Reidemeister-II fold crossings. In the first crossing, the higher-genus side of a fold arc approaches the higher-genus side of another, and two intersections form. In the second Reidemeister-II fold crossing, the newly-formed lune of regular values contracts, canceling the two intersections. A \emph{handleslide deformation} is a local handleslide deformation beginning with an SPWF as in Figure \ref{hsdef}, throughout which all fibers are connected and have genus at least 2.\end{df} 

A deformation such as the one in Figure \ref{hsdef} has disconnected fibers in the central lune if $\Sigma\setminus(\gamma_1\cup\gamma_n)$ is disconnected. According to \cite[Theorem 1.2]{GK1}, the initial deformation connecting $\alpha_0$ to $\alpha_1$ can be assumed to have connected fibers. Going further, Lemma \ref{genuslem} asserts there is a deformation whose fibers all have genus at least 2. The modifications in Section \ref{proof} preserve this condition.
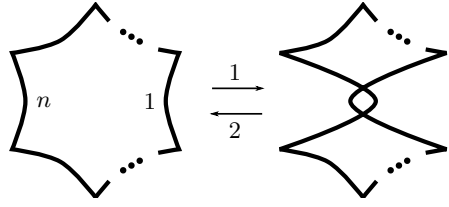
\begin{figure}\capstart
	\labellist
	\small\hair 2pt
	\pinlabel $n$ at 16 37
	\pinlabel 1 at 56 37
	\pinlabel 1 at 88 48
	\pinlabel 2 at 88 26
	\endlabellist
	\centering{\includegraphics{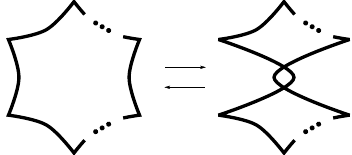}}
	\caption{Base diagrams for a handleslide deformation. Steps 1 and 2 are both Reidemeister-II fold crossings.}\label{hsdef}
\end{figure}
In \cite[Theorem 3.8]{H}, Hayano did essentially all of the work to prove that any handleslide deformation induces a handleslide move in a surface diagram. In \cite[Theorem 3.9]{H} he proved that any handleslide move is induced by some handleslide deformation. 

\begin{figure}\capstart	
\labellist
\small\hair 2pt
\pinlabel $a$ at 30 62
\pinlabel $b$ at 30 7
\pinlabel $p_1$ at 13 35
\pinlabel $p_2$ at 61 35
\pinlabel $a$ at 167 62
\pinlabel $b$ at 167 7
\pinlabel $p_1$ at 150 35
\pinlabel $p_2$ at 198 35
\pinlabel 1 at 92 45
\pinlabel 2 at 92 23
\endlabellist\centering{\includegraphics{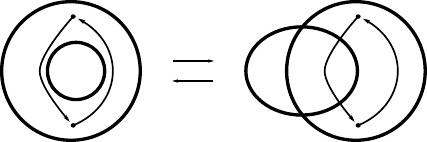}}\caption{A reproduction of \cite[Figure 5]{H}, in which the circle $\gamma$ is given as a concatenation of two paths $p_1$ and $p_2$. The leftmost lune after step 1 corresponds to the central lune in Figure \ref{hsdef}.}\label{hslem}
\end{figure}
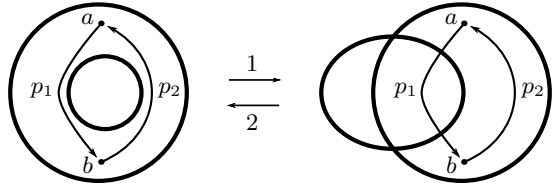

The fibration that appears in \cite[Figure 5]{H} is not an SPWF, so it remains to clarify why Hayano's argument applies to handleslide deformations. In Figure \ref{hslem}, the path $p_2$ and some chosen horizontal distribution specifies a diffeomorphism $\phi_2\co F_b\to F_a$ between the fibers above the endpoints $a$ and $b$. This diffeomorphism remains unchanged through Hayano's homotopy because the chosen horizontal distribution over $p_2$ does not change. On the other hand, Hayano shows that his homotopy changes the monodromy along $p_1\ast p_2$ precisely by composition with a map $\phi$ as in Section \ref{hsmove}. For this reason, the map $\phi_1\co F_a\to F_b$ becomes $\phi(\phi_1)$ as a result of his homotopy. Finally, in Figure \ref{hsdef}, $p_1$ gives the identification between higher-genus fibers in regions above and below the central lune in the right side of the figure, so the vanishing cycles for fold arcs contained in the upper region, as they appear in $F_a$, change by application of a map $\phi$ after a handleslide deformation.
\begin{rmk}Hayano's construction of the monodromy $\varphi_\gamma$ in \cite[Lemma 3.8]{H} as a product of three Dehn twists for a non-separating pair $(\gamma_1,\gamma_n)$ is the same as what appears in \cite[Lemma 3.7]{H} for pair that separates the fiber into surfaces of genus at least 2, so there is an immediately analogous version of handleslide for separating pairs. On the other hand, his version of the move is not necessary for the uniqueness Theorem \ref{T}.\end{rmk}

 \subsection{Stabilization}
 \subsubsection{Stabilization move}\label{stabmove}
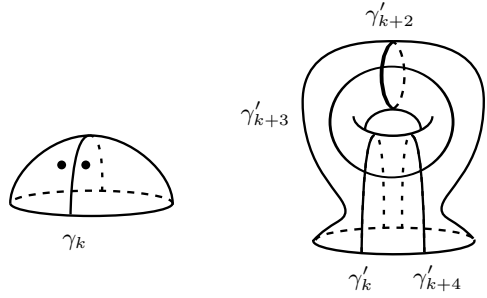
\begin{figure}\capstart
	\labellist
	\small\hair 2pt
	\pinlabel $\gamma_k$ at 18 15
	\pinlabel $\gamma_{k+3}'$ at 70 50
	\pinlabel $\gamma_{k+2}'$ at 104 77
	\pinlabel $\gamma_k'$ at 96 5
	\pinlabel $\gamma_{k+4}'$ at 117 5
	\endlabellist
	\centering{\includegraphics[width=.5\linewidth]{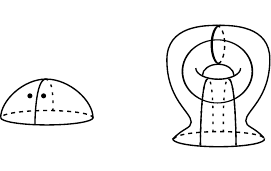}}
	\caption{Stabilizing a surface diagram near a point in $\gamma_k$. The circle $\gamma_{k+1}'$ has a more subtle description.}\label{stabimg}
\end{figure} 
This move comes from a lightly disguised generalization of a homotopy of Baykur that first appeared in Figures 5 of \cite{B2} (where it is called ``flip and slip") and 11 of \cite{L1}. Hayano in \cite{H} gave a careful description of this move; most of how it affects a surface diagram is shown in Figure \ref{stabimg}. On the left of Figure \ref{stabimg} is a small neighborhood of a point in a vanishing cycle, for which the numbering has been chosen so that it is the ``last" vanishing cycle of $\Gamma=(\gamma_1,\ldots,\gamma_k)$. According to Theorem 6.4 of \cite{H}, the stabilization deformation results in a diagram $(\Sigma_{g+1},\Gamma')$, with $\Sigma_{g+1}$ and the first $k$ entries of $\Gamma'$ obtained from $(\Sigma_g,\Gamma)$ by replacing a pair of small disks as shown at the left with a cylinder. Three of the last four entries of $\Gamma'=(\gamma_1',\ldots,\gamma_k',\gamma_{k+1}',\ldots,\gamma_{k+4}')$ are as in the figure, with $\gamma_{k+4}'$ parallel to $\gamma_k'$. The remaining circle $\gamma_{k+1}$ is given by $\mu'\circ\phi^{-1}(\gamma_{k+3}')$, where $\phi\in MCG(\Sigma_{g+1})$ is any element such that $\Phi_{\gamma_k}(\phi)=id$ and $\Phi_{\gamma_k+2}(\phi)=\mu$, and \[\mu'=t_{t_{\gamma_k'}(\gamma_1)}\circ t_{t_{\gamma_{k-1}'}(\gamma_k')}\circ\cdots\circ t_{t_{\gamma_1'}(\gamma_2')}.\] See Remark \ref{phic} for a discussion of $\Phi$ and $\mu$. By comparing the requirements on $\phi$ with the definition of the handleslide move, it is clear that once an element $\phi$ is found, that is, once an actual stabilization is performed, all the other possible stabilizations on the same vanishing cycle differ by handleslide moves. The problem of how to systematically produce such an element given $\Gamma$ is largely open, though numerous examples appear in \cite{H}.
 \begin{figure}\capstart
 	\labellist
 	\small\hair 2pt
 	\pinlabel $g$ at 47 52
 	\pinlabel $g-1$ at 75 96
 	\pinlabel $g+1$ at 255 52
 	\pinlabel $g$ at 285 52
 	\endlabellist
 	\centering{\includegraphics{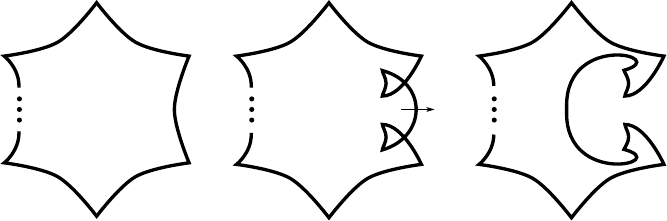}}
 	\caption{The stabilization deformation involves two flipping moves that occur on the same fold arc, eventually followed by an $R_2$ deformation that cancels the resulting intersections. The arrow indicates a fold arc that passes across the back of the sphere, completely sweeping over the lower-genus side.}\label{stabdef}
 \end{figure}
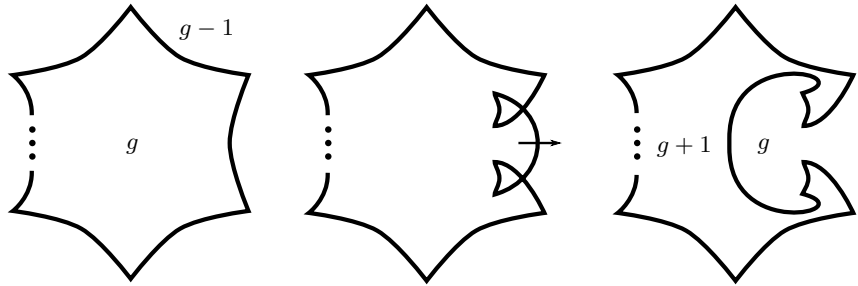
 \subsubsection{Stabilization deformation}\label{stabdefpf}
Figure \ref{stabdef} is a sequence of base diagrams for the stabilization deformation, corresponding to the move from Section \ref{stabmove}, studied in detail in \cite{H}. Here, two \emph{flipping moves} from \cite{L1} have been applied to the same fold arc, then the part of the fold arc between the two loops sweeps across the lower-genus side of the fibration to cancel the intersection points, crossing each cusp point (other than the four that appeared in the flips) exactly once. The part of the deformation following the flips is what will be called an \emph{$R_2$ deformation}, which is defined here using the language of \cite[Section 2, item (4b)]{GK1}:

\begin{df}\label{R2defn}An $R_2$ deformation is a \emph{Reidemeister-II fold crossing} (or its reverse), slightly generalized to allow cusps to cross into or out of the boundary of the central lune, with each entry or exit as in Figure \ref{tinypair}. An $R_2$ deformation is also allowed to begin or end with a cusp-fold crossing.\end{df}

As depicted in \cite[Figure 5]{H}, the first and second fold crossings of what that paper names an \emph{$R_2$-move} are both examples of what this paper calls an $R_2$ deformation.
\begin{figure}\capstart
	\centering{\includegraphics{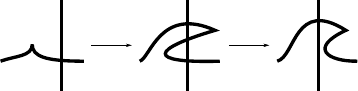}}
	\caption{Cusps might cross fold crossings during an $R_2$ deformation by a pair of 2-parameter crossings of this type, or its reverse. Here the depicted critical image is contained in a small disk that contains no other critical values.}\label{tinypair}
\end{figure}
Like in Figure \ref{clean}, depictions of the deformation in Figure \ref{tinypair} will be suppressed from pictures of the critical set. They should be understood to occur at any crossing between a cusp arc and an immersion arc.  

\begin{prop}\label{stabexists} Given any fold arc in an SPWF, there is a stabilization deformation whose initial flips occur on that fold arc.\end{prop} \begin{proof}First perform two flips on the chosen fold arc. By repeated application of \cite[Propositions 2.5(2) and 2.7]{W2}, it is possible to repeatedly apply deformations as in Figure \ref{tinypair} to move all cusps into the two loops introduced by the flips. At this point it is possible to cancel the two fold crossings by \cite[Proposition 2.5(3)]{W2}.\end{proof}

One special case of stabilization is available to surface diagrams that come from surface bundles over the sphere, that is, blank surface diagrams. Performing a birth move in this case increases the genus of the surface diagram by one and introduces two simple closed curves whose only requirement is that they intersect at a unique point. In \cite{BH}, it is part of the definition of \emph{surface diagram} that there are at least two circles.

\begin{ex}\label{s2s2ex}
Now that the term \emph{stabilization} has appeared, it makes sense to give the following example of a surface diagram. Figure \ref{s2s2} is a surface diagram for $S^2\times S^2$, with the circles numbered in the order they would appear going counter-clockwise around the image of a 10-cusped critical circle. The SPWF underlying this map is homotopic to the sphere bundle over the sphere that projects onto either of the factors of $S^2\times S^2$. Beginning with that map, perform a birth move and perform two stabilization deformations to obtain Figure \ref{s2s2}. It is curious that the other $S^2$ bundle over $S^2$ with total space diffeomorphic to $\CP\#\CPb$ also begins with a base diagram that has a sphere fiber with no vanishing cycles, and that one may perform the ``same'' three moves; however, a surface diagram specifies the total space up to diffeomorphism, so the resulting vanishing cycles must be different from those in Figure \ref{s2s2}. For this reason, stabilization (as an operation that takes $(\Sigma,\Gamma)$ as input) is not well-defined if the genus of $\Sigma$ is less than 2.
	
Hayano \cite{H} was able to sufficiently understand this phenomenon to convert the fibration of $S^4$ originally appearing in \cite{ADK} (whose fibers are genus 0 and 1, with one cusp-free fold circle between) into one with arbitrarily high genus to yield a family of surface diagrams of $S^4$. Starting with one of these diagrams, one could apply \cite[Lemma 5.1]{Be} to get a diagram of $\CP\#\CPb$.
	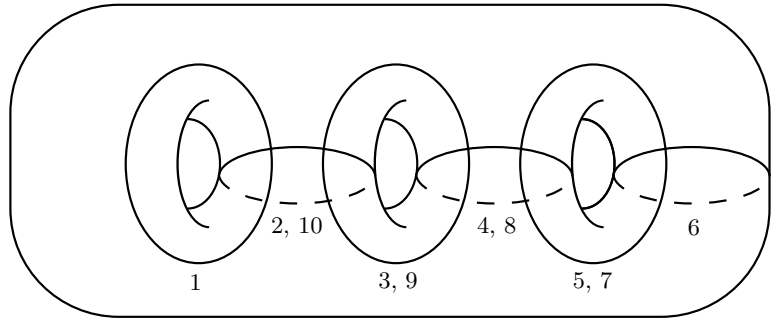
\begin{figure}\capstart
		\labellist
		\small\hair 2pt
		\pinlabel 1 at 68 13
		\pinlabel {2, 10} at 105 33
		\pinlabel {3, 9} at 142 13
		\pinlabel {4, 8} at 178 33
		\pinlabel {5, 7} at 213 13
		\pinlabel 6 at 250 33
		\endlabellist
		\centering{\includegraphics[width=0.8\linewidth]{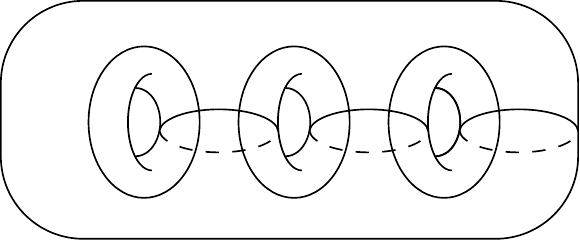}}
		\caption{A surface diagram for $S^2\times S^2$.}\label{s2s2}
	\end{figure}
\end{ex}
 
 \subsection{Multislide}\label{msmove}
 \subsubsection{Multislide move}
\begin{figure}\capstart%
	\centering\capstart	
	\subfloat[\ ]{\label{msfig1}\includegraphics{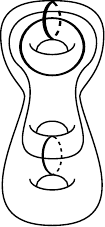}}	\qquad
	\subfloat[\ ]{\label{msfig2}\includegraphics{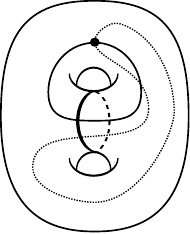}} \qquad
	\subfloat[\ ]{\label{msfig3}\includegraphics{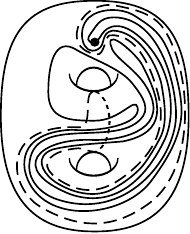}}
	\caption{An example of multislide (omitting the last step, a connect sum with $T^2$).}
\end{figure}
Independently found by Denis Auroux and Rob Kirby, then carefully explained with proof in \cite{BH}, this move comes from a deformation in which the critical locus becomes momentarily disconnected. To perform the multislide, one first finds a non-consecutive pair of vanishing cycles that intersect transversely at a unique point in $\Sigma$, say $\gamma_1$ and $\gamma_n$ (the two bold curves in Figure \ref{msfig1}). One also chooses one of the subsets $\{\gamma_2,\ldots,\gamma_{n-1}\}$ or $\{\gamma_{n+1},\ldots,\gamma_k\}$ on which to perform the move, calling it the nonstationary set (two of which appear as the other two curves in Figure \ref{msfig1}). A small tubular neighborhood of the chosen pair $\gamma_1,\gamma_n\subset\Sigma_g$ is a punctured torus $T$. The multislide proceeds as a one-parameter family of diffeomorphisms of $(\Sigma_{g-1},\gamma_2,\ldots,\gamma_{n-1})$ in which $T$ travels along an arbitrary circle $\alpha$ based at a point $p\in\partial T$, returning to where it started and dragging the nonstationary set along with it (but leaving all the other vanishing cycles $\gamma_n,\ldots,\gamma_1$ unchanged). In other words, if one were to replace $T$ with the point $p$ as in Figure \ref{msfig2}, the modification of the nonstationary set would be to apply the point-pushing map of the Birman exact sequence of mapping class groups along the circle $\alpha$ (pushing along the dotted circle in Figure \ref{msfig2} results in Figure \ref{msfig3}), then replace $T$ so that the vanishing cycles glue back together in the obvious way. For details on the point push map, see for example section 4.2 of \cite{FM}. A choice of framing also allows the nonstationary set to be affected by a power of a $\Delta$-twist, which is in the language of \cite{BH} denotes the square root of a positive or negative Dehn twist along $\partial T$; see Section 6.2 and the proof of Lemma 5.2 of that paper. To explain the terminology, it is as if a handleslide has been performed in which a collection of vanishing cycles slides over a disk spanning $\partial T$ (followed by half Dehn twists along $\partial T$). $\Delta$
 
 \subsubsection{Multislide deformation}\label{msdefpf}
The multislide deformation, corresponding to the move in Section \ref{msmove}, is depicted in Figure \ref{msdef}. To begin, choose a pair of fold arcs whose vanishing cycles happen to intersect transversely at a unique point in the fiber (they do not have to be consecutive). In the left side of Figure \ref{msdef}, a vertical arc joins the chosen pair, signifying a \emph{fold merge} that results in the right side of the figure. 
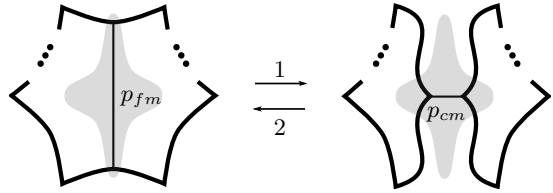
\begin{figure}\capstart
	\labellist
	\small\hair 2pt
	\pinlabel 1 at 102 45
	\pinlabel 2 at 102 23
	\pinlabel $p_{fm}$ at 50 34
	\pinlabel $p_{cm}$ at 165 28
	\endlabellist
	\centering{\includegraphics{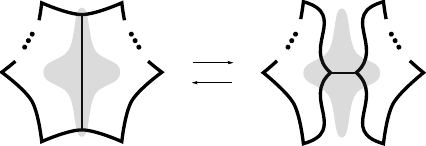}}
	\caption{The multislide deformation consists of a fold merge followed by a cusp merge. The local version is supported on a tubular neighborhood of the union of the fold merge path $p_{fm}$ and the cusp merge path $p_{cm}$, the images of which are shown, with the image of the support shaded.}\label{msdef}
\end{figure}\begin{rmk}This move is called a \emph{merge} in \cite{L1} and \cite{W1}, and an \emph{unmerge} in \cite{GK1}; the terms \emph{fold merge} for when two folds meet and \emph{cusp merge} for when two cusps meet appeared in \cite{BH}, will be used in this paper, and seems likely to enjoy future consensus.\end{rmk}The inside of each of the two circles now has $\Sigma$ as its regular fiber, decorated with vanishing cycles obtained from $\Gamma$ by deleting those whose fold arcs became contained in the other circle as a result of the fold merge (this is where the partition $\Gamma\setminus\{\gamma_1,\gamma_n\}=\{\gamma_2,\ldots,\gamma_{n-1}\}\cup\{\gamma_{n+1},\ldots,\gamma_k\}$ comes from). Before any other moves such as cusp-fold crossings, flips, or $R_2$ deformations occur, the deformation concludes with a cusp merge in which the two circles reunite along the two newly-formed cusps.
\begin{df}\label{mspairdf} A \emph{local multislide deformation} is a generalization of a multislide deformation satisfying the following two conditions: \begin{enumerate}\item[(i)]\hypertarget{mscon[i]}There is an embedding of a fibration structure of the type depicted over the shaded $D^2_{[0,1]}$ in Figure \ref{msdef} into the fibration structure over $S^2_{[0,1]}$ given by $\alpha$.\item[(ii)]\hypertarget{mscon[ii]} Let $\iota\subset\crit\alpha$ denote the collection of points where $\alpha|_{\crit\alpha}$ fails to be injective. As suggested in Figure \ref{msdef}, $\alpha(\iota)$ is bounded away from the one-parameter family of shaded disks in $S^2_{[0,1]}$ depicted in Figure \ref{msdef}.
\end{enumerate}The merge points of a local multislide deformation form a \emph{multislide pair}.\end{df}
To explain the corresponding modification of surface diagrams, consider the deformation as a map $[0,1]\times M\to[0,1]\times S^2$ that restricts to the identity map in the first factors. The critical image is a properly embedded twice-punctured torus in $[0,1]\times S^2$. If one chooses two reference points in the left side of the figure such that one reference point splits into each circle, initially their surface diagrams are isotopic. However, the identification between them is modified by the time the two circles rejoin: note that the cusps that form and then disappear trace out a circle in $[0,1]\times M$. Projecting out the homotopy parameter yields a circle in $M$, and without loss of generality one of the pair of cusps, say the right, remains stationary throughout its life, while the other traces out the circle as the homotopy progresses (the resulting surface diagram will then be the one obtained using a reference point that stays in the right circle through the deformation). To say it a different way, the horizontal arc in the middle of Figure \ref{msdef} can be taken as the image of the path to be taken by the left cusp. This so-called \emph{cusp merge path} (less precisely referred to as a \emph{joining curve} in \cite{Lev,W2}, $\alpha$ in \cite{L1}, and elsewhere unnamed) is framed, is everywhere transverse to the fiber, and has the two cusps as its endpoints. 

\begin{rmk}\label{mergepath}All cusp merge paths will be denoted $p_{cm}$ in this paper ($p_{fm}$ for fold merge paths), and note that once a cusp merge path is defined, its sub-paths at later $t$ values as the two cusps approach each other also serve as cusp merge paths, and similar for fold merge paths, so that really a cusp merge or fold merge involves a one-parameter family of merge paths $\left[p_{cm}\right]_t$, though such pedantry is only necessary in a few parts of the paper.\end{rmk}

At the moment $t_0$ when the two cusps form, the endpoints of $(p_{cm})_{t_0}$ coincide with the merge point, so $(p_{cm})_{t_0}$ is a circle in $M$. The fibers containing points on the interior of $(p_{cm})_t$ serve as $\Sigma'$, with the disk mentioned above coming from a fiberwise neighborhood of $(p_{cm})_t$. The circle $(p_{cm})_{t_0}$ projects to a circle in a fiber $\Sigma'$ along which this disk travels, inducing the above-mentioned isotopy of the vanishing cycles in the left circle, which in turn induces an element of the mapping class group of $\Sigma$ that only applies to those vanishing cycles coming from the left circle.

Choosing the left cusp instead of the right as the stationary cusp (and the left circle's reference point as the ending reference point) results in performing exactly the reverse modification to $\gamma_{n+1},\ldots,\gamma_k$, moving along the circular path in the opposite direction and applying oppositely oriented Dehn twists along the boundary of the disk. For completeness, it should be mentioned that the surgered surface diagram explanation for the multislide involves a reference path connecting the two newly-formed cusps throughout their existence; the details are left to the reader.

 \subsection{Shift}\label{shiftmove}
This move could be seen to embody the possible variations on the ordering of the vanishing cycles in a surface diagram, though it does not simply re-index the elements of $\Gamma$. It was studied carefully in \cite{BH}, and for the reader's convenience a summary follows. Because of the wording of Proposition \ref{shiftprop}, it is convenient to first outline the shift deformation. As with the multislide, one identifies a pair $\gamma_l, \gamma_k\in\{\gamma_1,\ldots,\gamma_k\}$, along with a nonstationary subset $\{\gamma_{l+1},\ldots,\gamma_{k-1}\}$.
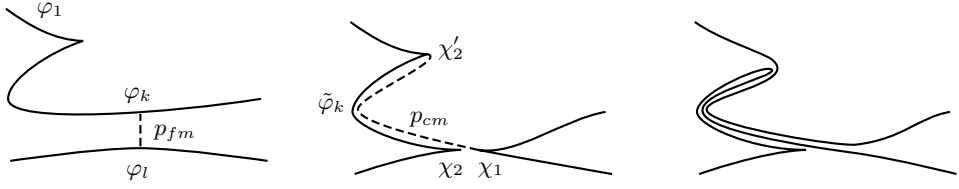
\begin{figure}\capstart
	\labellist
	\small\hair 2pt
	\pinlabel $\varphi_1$ at 17 67
	\pinlabel $\varphi_k$ at 50 35
	\pinlabel $p_{fm}$ at 64 21
	\pinlabel $\varphi_l$ at 50 6
	\pinlabel $\chi_2'$ at 168 53
	\pinlabel $\tilde\varphi_k$ at 123 32
	\pinlabel $p_{cm}$ at 160 26
	\pinlabel $\chi_2$ at 168 7
	\pinlabel $\chi_1$ at 183 7
	\endlabellist
	\centering{\includegraphics{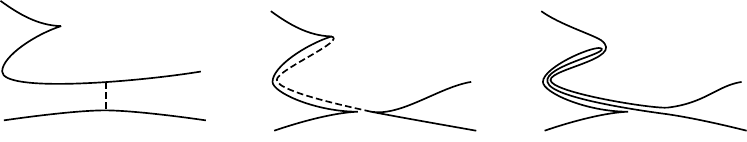}}
	\caption{The shift deformation consists of a fold merge followed by a cusp merge. The dotted arc is the image of a cusp merge path $p_{cm}$ that is closely parallel to $\tilde\varphi_k$ in $M$.}\label{shiftdef}
\end{figure}

In base diagrams, the shift deformation is similar to that of the multislide in Figure \ref{msdef}, in which a fold merge occurs between the fold arcs $\varphi_l$ and $\varphi_k$ corresponding to $\gamma_l$ and $\gamma_k$, but instead of reuniting the two circles by a cusp merge between the two newly-formed cusps, the circles reunite by a cusp merge between one of the newly-formed cusps ($\chi_1$ in Figure \ref{shiftdef}) and the cusp between $\varphi_1$ and $\varphi_k$. Such a deformation is called a \emph{generalized shift deformation} in \cite{BH}.
\begin{prop}[\cite{BH}, Proposition 6.3]\label{shiftprop}Let $w\co X\to D^2$ be a simple wrinkled fibration with surface diagram $(\Sigma;\gamma_1,\ldots,\gamma_k)$ such that for some $1<l<k$ the curves $\gamma_l$ and $\gamma_k$ intersect at one point.\begin{enumerate}\item If $w'\co X\to D^2$ is obtained from $w$ by a generalized shift deformation whose initial fold merge is applied around the two fold arcs with vanishing cycles $\gamma_l$ and $\gamma_k$, then the surface diagram of $w'$ is given by\begin{equation}\left(\Sigma,\gamma_1,\ldots,\gamma_l,\gamma_k,\chi(\gamma_{l+1}),\ldots,\chi(\gamma_{k-1})\right),\label{eq:shiftdiag}\end{equation}where $\chi\in\Mod(\Sigma)$ is a mapping class that satisfies the following properties:\begin{enumerate}\item $\chi t^{-1}_{c_k}t^{-1}_{c_l}\in\Ker\Phi_{c_k}$,\item $\chi t^{-1}_{c_1}t^{-1}_{c_l}\in\Ker\Phi_{c_1}$.\end{enumerate}\item For any $\chi\in\Mod(\Sigma)$ satisfying the conditions (a) and (b), there exists a generalized shift deformation from $w$ to a simple wrinkled fibration whose surface diagram is given by Expression \ref{eq:shiftdiag}.\end{enumerate}\end{prop} 

The above modification simplifies somewhat with the following assumption about the deformation. Suppose that, once the initial fold merge occurs, there is a path traced out by one of the resulting cusps, with the other cusp remaining stationary in $M$. Slightly after the initial fold merge, the endpoint of this path (whose interior consists of regular points at all times) travels along the fold arc $\varphi_k$ to the cusp between $\varphi_k$ and $\varphi_1$, then serves as the cusp merge path $p_{cm}\co[0,1]\to M_{1/2}$ by which the two critical circles reunite. In other words, there is a smooth embedded disk in $M$ whose boundary $c$ is naturally partitioned into three arcs as follows: two arcs are cusp merge paths between the cusps involved in the shift, and the third edge consists of fold points in $\varphi_k$. In \cite[Remark 6.5]{BH} it is explained that there exist precisely two such disks, up to isotopy relative to the boundary, yielding a possibly distinct pair of shift moves on surface diagrams, and Behrens and Hayano describe how to construct the mapping class $\chi$ in this case. For such surface diagram moves or deformations this paper and \cite{BH} use the term \emph{shift} or \emph{shift deformation}. The deformation $M_{[0,1]}\to S^2_{[0,1]}$ depicted in Figure \ref{shiftdef} is supported in an open ball containing the cusp between $\varphi_1$ and $\varphi_k$ and an open arc from $\varphi_l$, and restricts to a product map $\alpha_0\times\id_{[0,1]}$ on the rest of the fibration. The following definition singles out the merge points of such a deformation.

\begin{df}\label{shiftpairdf} A \emph{local shift deformation} is a generalization of a shift deformation satisfying the following two conditions: \begin{enumerate}\item[(i)]\hypertarget{shiftcon[i]}There is an embedding of the fibration structure over $D^2_{[0,1]}$ corresponding to the deformation in Figure \ref{shiftdef} into the fibration structure over $S^2_{[0,1]}$ given by $\alpha$.\item[(ii)]\hypertarget{shiftcon[ii]} Let $\iota\subset\crit\alpha$ denote the collection of points where $\alpha|_{\crit\alpha}$ fails to be injective. As in \cite{BH} and as suggested in Figure \ref{shiftdef}, $\alpha(\iota)$ is bounded away from the one-parameter family of disks in $S^2_{[0,1]}$ depicted in Figure \ref{shiftdef}. 
\end{enumerate}
Similar to a multislide pair, a \emph{shift pair} is the pair of merge points in $\crit\alpha$ at either end of a local generalized shift deformation. In the proof of Lemma \ref{genus} there will appear pairs that satisfy Condition (\hyperlink{shiftcon[i]}{i}) but not  (\hyperlink{shiftcon[ii]}{ii}); these will be called \emph{shift pair candidates}. In other words, a shift pair candidate is a shift pair with the possibility of fold crossings coming from fold arcs in Figure \ref{shiftdef}.\end{df}

Like a local multislide deformation, a local shift deformation differs from the shift deformation that appears in \cite{BH} only in one way: In this paper, a local shift deformation might be applied to a fibration with immersed fold arcs or disconnected critical locus, while a shift deformation is a local shift deformation whose endpoints are simplified purely wrinkled fibrations. Condition (\hyperlink{shiftcon[ii]}{ii}) is nontrivial: it proscribes, for example, embeddings such that $\alpha|_{\tilde\varphi_k}$ is not injective, or embeddings such that the image of $\chi_1$ crosses folds as it travels along $p_{cm}$.

\begin{rmk}\label{addms}The shift deformation contrasts to the multislide deformation, in which the circle in $M$ traced by one cusp (while the other remains stationary) may not bound a disk. Suppose we are given a deformation that corresponds to a shift, except that the cusp merge path $p_{cm}$ and the path $\lambda$ traced by $\chi_1$ from $m_1$ to $p_{cm}(0)$ gives a circle \[c=p_{cm}\cup\tilde\varphi_k\cup\lambda\subset M_{1/2}\] that does not bound a disk. One may decompose this into a multislide deformation followed by a shift deformation as follows. Beginning at a value $t=t_0$ immediately after the initial fold merge, perform the following deformation: Cusp merge $\chi_1$ and $\chi_2$ using a cusp merge path that causes $\chi_1$ to trace out a circle which is homotopic to $c$, and then immediately perform the reverse of that cusp merge. From there, the deformation continues as it did originally, beginning with the map $\alpha_{t_0}$. In the critical surface, it looks like the result of performing a self-connect sum along a pair of points contained in the cusp arcs emanating from the initial merge point. Now the deformation consists of a multislide followed by a shift. There is clearly an analogous modification in the local case.\end{rmk}

\begin{df}A deformation that consists of a deformation followed precisely by its reverse is called a \emph{detour}.\end{df}

As in Remark \ref{addms}, any detour beginning with the wrinkled fibration $\alpha_{t_0}$ clearly can be inserted into $\alpha$ at $t=t_0$ to get a new deformation $\alpha'$ with the same endpoints, and this is a tool that will be used repeatedly. It is easy to see that $\alpha$ and $\alpha'$ are homotopic: A detour is just a loop in the space of maps $M\to S^2$ given by an arc followed by its reverse, which is a nullhomotopic loop based at $\alpha_{t_0}$.

\section{Proof of the main theorem}\label{proof}
\noindent Previous sections connected model deformations with moves on surface diagrams. The rest of the paper is concerned with the part of the proof of Theorem \ref{T} that converts a given deformation into a sequence of model deformations.

\subsection{Some notation}\label{notation}
In what follows, the notation $M_I$ for $I\times M$ where $I=[0,1]$ will be convenient, and maps $M_I\to S^2_I$ will always be smooth and will always be the identity on the first factor (that is, they are smooth homotopies). As implied, the notation $M_t$ will denote the level set $(t,M)$, and similarly for a homotopy $\alpha$, the notation $\alpha_t$ will refer to the map $M_t\to S^2_t$. Abusing notation even further, symbols like $\crit\alpha_{\{t_0,t_1\}}$ will denote $\crit(\alpha_{t_0})\cup\crit(\alpha_{t_1})$. Finally, we routinely conflate various maps with the fibration structures they induce. The critical surface of a deformation is a smooth 2-submanifold of $M_{[0,1]}$ corresponding to the moves in \cite[Section 2.4]{W1} (also paraphrased in Section \ref{defmovessection}), where there appears proof of the statements of the rest of this paragraph. Since $\alpha$ is a deformation, the function $T=t|_{\crit\alpha}$ that projects $\crit\alpha$ to the $t$-coordinate is Morse and its critical points are those at which birth and merge moves occur. The further restriction $T|_{\overline{\chi}}$ to the closure of the cusp locus $\chi\subset\crit\alpha$ is also a Morse function on a smooth 1-manifold, whose boundary consists of the cusp points in $\alpha_0$ and $\alpha_1$. Since we are assuming $\alpha_0$ and $\alpha_1$ are simplified purely wrinkled fibrations, there is one boundary circle in each $M_i$, possibly dotted with cusps, and $\alpha_i$ is injective on each. Examining the local models for birth, merge and flip, it becomes clear that the index zero and two critical points of $T$ are precisely where births and their inverses occur, while the merges occur precisely at the index one critical points of $T$. Call these points birth and merge points, regardless of whether they correspond to births and merges or their inverses. The critical points of $T$ at birth and merge points are all cusps. The critical points of $T|_{\overline{\chi}}$ which are not birth or merge points are all swallowtail points, where flipping moves occur in the deformation. 

As in Figures \ref{clean} and \ref{dirty}, the critical manifold of a deformation can be depicted with various decorations: the immersion locus $\iota$ consisting of paired immersion arcs; the cusp locus $\chi$, which is a smoothly embedded 1-submanifold of $\crit\alpha$; and the collection of swallowtail points, which coincides with $\overline{\chi}\setminus\chi$.

\begin{df}For a double point $x$ in the critical image of $\alpha$, let $\{p,p'\}=\alpha^{-1}(x)\cap\crit\alpha$. Then the \emph{counterpart} of $p$ is $p'$. Similarly, an arc or a circle in $\iota$ can have a counterpart, defined as the arc or circle in $\iota$ with the same image under $\alpha$. For any $A\subset\iota$, the counterpart of $A$ will be denoted $A'$.\end{df}

To help gain perspective on what is being depicted, here is a short comment on vanishing cycles, a further decoration that does not find use in this paper. Away from a tubular neighborhood of $\iota$ one could indeed view $\crit\alpha$ as a kind of base diagram whose fiber above any point $p\in\crit\alpha_t$ is defined by $\alpha^{-1}\circ\widetilde{\alpha(p)}$, where $\widetilde{\alpha(p)}$ is obtained by slightly perturbing $\alpha(p)$ in $S^2_t$ to a regular value in its higher-genus side (one may then extend to $\crit\alpha\setminus\iota$ by continuity). This defines a closed, orientable surface along with a distinguished simple closed curve, which is the vanishing cycle for $\alpha(p)$ in a base diagram for $\alpha_t$. In this sense, each point on the critical surface can be marked with a regular fiber, itself marked with a simple closed curve, and each immersion arc has a vanishing cycle coming from the fold that contains its counterpart: Crossing an immersion arc has the same effect on the fiber as crossing a fold arc in a base diagram. Following fibers across a cusp arc, the  distinguished vanishing cycle changes into another that transversely intersects the previous at a unique point.

Though such a decorated depiction of $\crit\alpha$ is a picture of a smooth two-dimensional submanifold of the 5-manifold $M_I$, it is only a recording of the one-parameter family of stratified immersions $\alpha(\crit\alpha)$ and vanishing cycles, so it more properly should be considered a depiction of a one-parameter family of base diagrams. This family can be recovered from the decorated critical surface by first using the immersion locus and merge points to recover the sequence of Reidemeister-type deformations ($R_2$ and $R_3$ deformations defined in the next paragraph, and flips) and merges in the critical image: These are unambiguously specified by examining whether the genus increases or decreases as one crosses the relevant immersion arcs. After adding cusp points, one might then add regular fibers marked with vanishing cycles if that information is included in the decorations. In this way, a decorated critical surface is simply another way to record information about a deformation. It is important to note that a base diagram does not in general specify the total space up to diffeomorphism; see for example the discussion around \cite[Figure 8]{GK2}. Similarly, a decorated critical surface is a typically lossy tool for recording a deformation, though it turns out to be sufficient for proving Theorem \ref{T}.

It follows from Lemma 1 of \cite{W1} that, as $t$ increases or decreases, there are only two ways for crossings to form in the critical image of a deformation: flips and $R_2$ deformations. In the critical surface, a flipping move is encoded by a pair of immersion arcs nested around a pair of cusp arcs whose common endpoint is a swallowtail point as in Figure \ref{flip}. The $R_2$ deformation appears as two immersion arcs $\iota_1$ and $\iota_2$ that appear as in Figure \hyperlink{16b}{16b}.
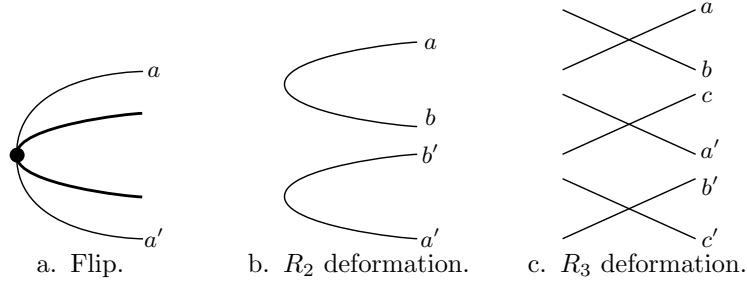
\begin{figure}\capstart\hypertarget{16b}
	\centering	
	\labellist
	\small\hair 2pt
	\pinlabel $a$ at 81 94
	\pinlabel $a'$ at 81 1
	\pinlabel $a'$ at 236 1
	\pinlabel $a$ at 236 110
	\pinlabel $b'$ at 236 48
	\pinlabel $b$ at 236 69
	\pinlabel $c'$ at 393 1
	\pinlabel $c$ at 391 80
	\pinlabel $a$ at 391 130
	\pinlabel $a'$ at 393 51
	\pinlabel $b$ at 391 96
	\pinlabel $b'$ at 393 30
	\endlabellist
	\subfloat[Flip.]{\makebox[3cm][c]{
	\includegraphics[width=2cm]{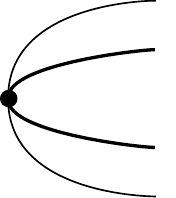}\label{flip}}}\qquad
	\subfloat[$R_2$ deformation.]{\makebox[3cm][c]{\includegraphics[width=2cm]{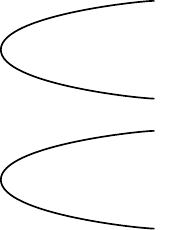}\label{$R_2$}}}\qquad
	\subfloat[$R_3$ deformation.]{\makebox[3cm][c]{\includegraphics[width=2cm]{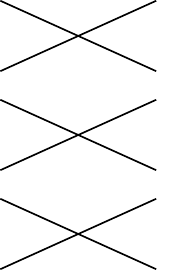}\label{$R_3$}}}
	\caption{Critical and immersion loci of $\alpha$. In such diagrams $t$ increases to the right, cusp arcs are bold, and pairs of fainter arcs are mapped to each other by $\alpha$ (components may be flipped or reordered vertically depending on $\alpha$). The dot is a swallowtail point.}\label{immerlocmods}
\end{figure} 
\hypertarget{props2}An important feature of deformations, which will be called property $s_2$, is that the maps $T|_{\iota_i}$ each have a single common critical value for each $R_2$ deformation; in other words, the arcs $\iota_i$ are tangent to the same level set $M_t$, and such tangencies always come in pairs, one pair for each $R_2$ deformation. Finally, there may be Reidemeister-III fold crossings in which cusps may lie in the initial triangle (call these $R_3$ deformations for short), whose immersion loci appear in $\crit\alpha$ as in Figure \hyperlink{16b}{16c}. \hypertarget{props3}These also come with a simultaneity condition which will be called property $s_3$: the three intersection points between immersion arcs in the critical surface must all have the same $t$-coordinate. Since $\alpha_0$ and $\alpha_1$ are injective on their critical circles, the closure of the immersion locus is a union of circles; let $c$ denote one of these (recall that swallowtail points are included in the immersion locus).

The constructions used in the proof of Theorem \ref{T} are mostly recorded as modifications of $\crit\alpha$ and its stratified immersion into $S^2_I$, and it is crucial to know that those changes actually produce deformations. Given a decorated critical surface, there are at least three conditions that must be satisfied by the result:\begin{enumerate}
\item\hypertarget{condition1} The critical points of $T$ must be precisely the birth and merge points of $\alpha$, and $\crit(T|_{\overline{\chi}})$ must be the disjoint union of $\crit T$ with all the swallowtail points.
\item\hypertarget{condition2} The immersion locus must come from an immersion of a surface into $S^2_I$ (possibly with corners coming from cusps) and follow one of the local models in Figure \ref{immerlocmods} at each tangency with any $M_t$.
\item\hypertarget{condition3} The vanishing cycles around birth, merge, swallowtail and immersion points must be valid according to:\begin{enumerate}
	\item The local models for birth, merge, and flip;
	\item The disjointness of vanishing cycles as discussed in \cite[Definition 2.1]{W2};
	\item The identifications prescribed by each pair of immersion points $(p,p')$: The vanishing cycle of an immersion point $p$ must be the vanishing cycle of the fold containing $p'$.\end{enumerate}\end{enumerate}
It is clear that the critical surface of any deformation satisfies these three conditions. In the other direction, it is a subtle question to determine whether a decorated critical surface actually comes from a deformation. In the constructions of this paper, verifying that modifications of a decorated critical surface actually yield new deformations is typically verified on an ad hoc basis by exhibiting corresponding base diagrams where necessary, using results from \cite{W2}, and using other arguments that do not appeal to the properties of any particular decorated critical surface. On the other hand, the results of Section \ref{moves} that connect the moves on surface diagrams with model deformations show that a deformation is given by a sequence of model deformations if and only if its decorated critical surface is a concatenation of the decorated critical surfaces of model deformations. The approach is to modify the decorated critical surface until the sequence of moves can be read off like in Figure \ref{clean}, though the precise implementation of each move as a modification of vanishing cycles is lost. 
\subsection{Splicing cusp arcs}\label{splice}
This section gives a few tools which will be important for gaining some control over how the cusp locus is embedded in $\crit\alpha$. The section begins with an observation about how one can push swallowtails around in a decorated critical surface.
\begin{lemma}\label{swalpushlem}By homotopy of $\alpha$, it is possible to move the swallowtail point of a flipping move backward in $t$ (or an inverse flipping move forward in $t$) across any adjacent immersion arc. The result appears in Figure \ref{swalpush}.\end{lemma}
\begin{figure}\capstart
	\centering{\includegraphics{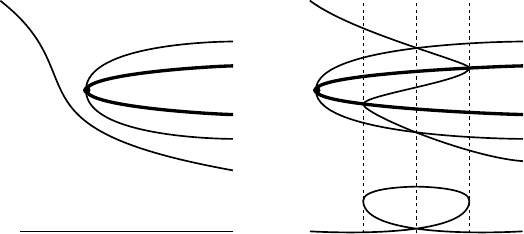}}
	\caption{Figure for Lemma \ref{swalpushlem}. A flipping move is pushed backward in $t$ so that its swallowtail point moves across an immersion arc.}\label{swalpush}
\end{figure}
\begin{proof}In such a homotopy, the initial deformation is just a flipping move occurring on a fold arc $A$ near an intersection of $A$ with another fold arc $B$. The decorated critical surface for such a deformation appears on the left side of Figure \ref{swalpush}. Following the picture on the right, the terminal deformation is also a flip on a point of $A$, but on the other side of the crossing. Then the fold arc $B$ moves across the loop formed by the flip, undergoing a pair of cusp-fold crossings and a Reidemeister-III fold crossing.

The local model for a swallowtail has a particular fibration structure
\[f_{[-\varepsilon,\varepsilon]}\co B^4_{[-\varepsilon,\varepsilon]}\to B^2_{[-\varepsilon,\varepsilon]}\]
for any $\epsilon>0$, and the required one-parameter family of deformations comes from simply choosing an appropriate extension of this fibered neighborhood by a ball of regular points of $\alpha$ (for example, take the union of $B^4_{[-\varepsilon,\varepsilon]}$ with a neighborhood of an arc that is parameterized by $t$, disjoint from $\crit\alpha$ except for its terminal point which is the swallowtail), then within this larger ball making a one-parameter family of coordinate changes that sends the swallowtail point backward in $t$. 
\end{proof}
The first tool for modifying the critical locus is a local one that introduces two swallowtail points into a cusp arc, as shown in Figure \ref{swalpair0}. As usual, the two dots are swallowtail points: the dot on the left corresponds to a flipping move, the one to the right an inverse flip. The bold arcs are cusps, and there is a fainter circle of immersion points, where vertically aligned points are mapped to each other by $\alpha$.

\begin{figure}\capstart
	\centering{\includegraphics{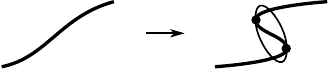}}
	\caption{A pair of swallowtail points inserted into a cusp arc.}\label{swalpair0}
\end{figure} 
Figures \ref{swalpaira} and \ref{swalpairb} give base diagrams for the right side of Figure \ref{swalpair0}. At first, there is a cusp as in Figure \ref{swalpair1}; a flip occurs, giving Figure \hyperlink{19b}{19b}; the two upper vanishing cycles in the reference fiber follow directly from the local model for flips as described in \cite{L1,W2,GK1}, while the lower two come from the fold arcs on either side of the original cusp. To obtain Figure \ref{swalpair3}, the cusp at the lower left moves into the higher-genus region, with the vanishing cycles unchanged. Next, the cusp at the top left of Figure \ref{swalpair3} moves into the lower-genus region resulting in Figure \hyperlink{20b}{20b}, at which point the loop can be shrunk away by an inverse flipping move, and the upper two vanishing cycles survive.

\begin{figure}\capstart\hypertarget{19b}%
	\centering
	\subfloat[\ ]{\label{swalpair1}\includegraphics{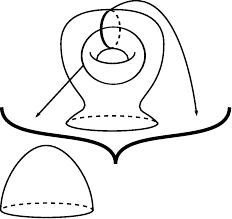}} \qquad
	\subfloat[\ ]{\label{swalpair2}\includegraphics{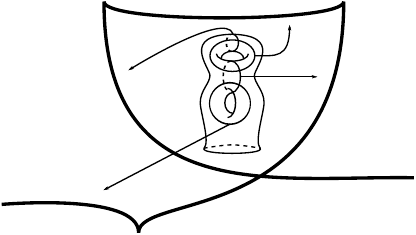}}	
	\caption{A pair of swallowtails introduced to a cusp arc, part one.}\label{swalpaira}
\end{figure}
\begin{figure}\capstart\hypertarget{20b}%
	\centering
	\subfloat[\ ]{\label{swalpair3}\includegraphics[width=0.45\linewidth]{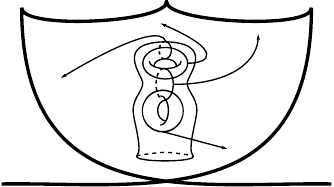}} \qquad
	\subfloat[\ ]{\label{swalpair4}\includegraphics[width=0.45\linewidth]{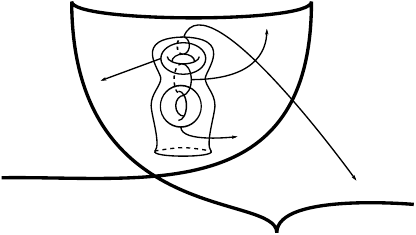}} \qquad
	\caption{A pair of swallowtails introduced to a cusp arc, part two.}\label{swalpairb}
\end{figure}

It is also possible to cancel two swallowtails that are connected by an arc of fold points as in Figure \ref{cancelswal}. In this situation, the base diagrams for the left side have an inverse flip, only for a flipping move to later occur on the same fold arc. Taking advantage of the local model for flips,\begin{equation}\label{eq:swallowtail}\left(t,x_1,x_2,x_3,x_4\right)\mapsto\left(t,x_1,x_2^4+x_2^2t+x_1x_2\pm x_3^2\pm x_4^2\right)\end{equation} in which a flipping move occurs at $t=0$, the left side of the figure has a local model obtained by replacing $t$ with $t^2+\epsilon$ for some small $\epsilon<0$; increasing $\epsilon$ to a positive value gives the right side. The gluing between the two local models of flips implicit in this construction comes from a normal framing to the fold locus along the arc of fold points connecting the two swallowtail points.
\begin{figure}\capstart
	\centering{\includegraphics{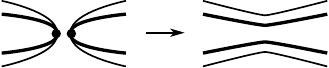}}
	\caption{Canceling a pair of swallowtails.}\label{cancelswal}
\end{figure} 

Taking further advantage of the local model, it is possible to move a flipping move backward in $t$, or an inverse flipping move forward (within the same strip of fold points), by suitably extending the embedding of its model in that direction. This strings along a trail of cusp and immersion arcs in the wake of the swallowtail point, introducing pairs of cusp arcs and immersion points into the intervening maps $\alpha_t$.

Combining these moves yields a splicing modification reminiscent of the bypass move from contact topology, as in Figure \ref{bypass}. One introduces pairs of swallowtail points to three adjacent cusp arcs, then applies the canceling move twice. Applying this move inductively, it is possible to bypass across any odd number of cusp arcs along a reference arc that intersects them all either positively or all negatively. 
\begin{figure}\capstart
	\centering{\includegraphics{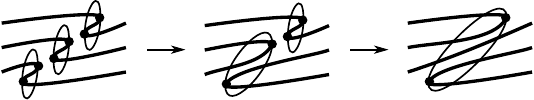}}
	\caption{Bypassing an odd number of cusp arcs.}\label{bypass}
\end{figure} A final application allows one to change the fold arc upon which any flip occurs in a deformation, as in Figure \ref{swaljump}.
\begin{figure}\capstart
	\centering{\includegraphics{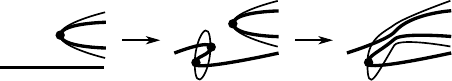}}
	\caption{Moving a swallowtail into a neighboring region of fold points.}\label{swaljump}
\end{figure}
 
\subsection{Switching between cusp and fold merges}\label{switching}
In some places it will be necessary to ensure a cusp merge occurs at a particular merge point $m\in\crit\alpha$, or that a fold merge occurs at $m$. For instance, in a multislide deformation the first merge point is a fold merge and the second is a cusp merge. To address this issue we use a trick originally due to Denis Auroux, enabling one to switch between a fold merge and cusp merge in the presence of a flipping move, as follows.

\begin{figure}\capstart\hypertarget{24b}%
	\centering
	\subfloat[\ ]{\label{atprep1}\includegraphics{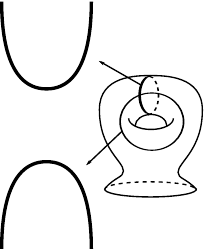}} \qquad
	\subfloat[\ ]{\label{atprep2}\includegraphics{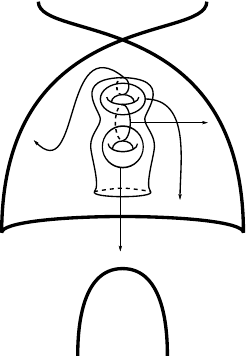}} \qquad 
	\subfloat[\ ]{\label{atprep3}\includegraphics{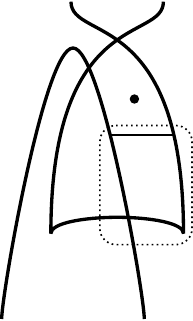}} \\
	\subfloat[\ ]{\label{atprep4}\includegraphics{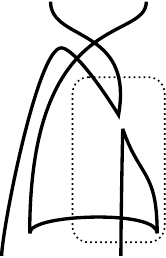}} \qquad 
	\subfloat[\ ]{\label{atprep5}\includegraphics{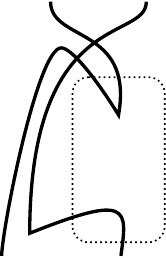}} \qquad 
	\subfloat[\ ]{\label{atprep6}\includegraphics{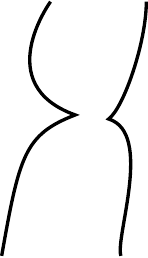}} 
	\caption{Putting a fold merge into position for reversal.}\label{atprep}
\end{figure}
Suppose a fold merge occurs at $m$, and we wish to modify $\alpha$ so that a cusp merge occurs instead. The other case, converting a fold merge into a cusp merge, is precisely the same modification with $t$ reversed. The first step is to put $m$ into a situation to which Auroux's trick applies. For this one may introduce a pair of swallowtails along its cusp arc qualitatively like in Figure \ref{swalpair0}, except the two swallowtails lie before and after the merge point with respect to $t$; the precise modification appears in Figure \ref{atprep}. The way to interpret this figure is as a replacement for the fold merge deformation, whose base diagrams are given by Figure \ref{atprep1} followed immediately by Figure \hyperlink{24b}{24f}; for fiber genus at least 2 (which will follow from Lemma \ref{genuslem} wherever it is used), the validity of the replacement follows from the validity of the intervening base diagrams and the fact that the modifications therein occur relative to the fibration above the boundary of the target disk. The modification proceeds with a flipping move followed by an $R_2$ deformation to obtain Figure \hyperlink{24b}{24c}; the vanishing cycles follow directly from the local model for the flip and the $R_2$ deformation is valid by \cite[Proposition 2.7]{W2}. The decorated reference fiber in Figure \hyperlink{24b}{24b} can be transferred to be the reference fiber above the dot in Figure \hyperlink{24b}{24c}, with arrows pointing at the same fold arcs as before. In other words, the fold arcs of Figure \hyperlink{24b}{24c} \emph{inherit} their vanishing cycles from Figure \hyperlink{24b}{24b} and we will use that term repeatedly in such contexts. Performing the indicated fold merge to obtain Figure \hyperlink{24b}{24d}, the next steps are to cancel the top two intersections with an $R_2$ deformation (which is valid by an application of \cite[Proposition 2.7]{W2} followed by \cite[Proposition 2.5(2)]{W2}) and to close off the lower loop by an inverse flip, which can be seen to be valid by drawing in the three relevant vanishing cycles inherited from those of Figure \hyperlink{24b}{24b}. The deformation above the three dotted rectangles now contains a fold merge that can be replaced by one containing a cusp merge using Auroux's trick, which follows.
\begin{figure}\capstart
	\centering{\includegraphics{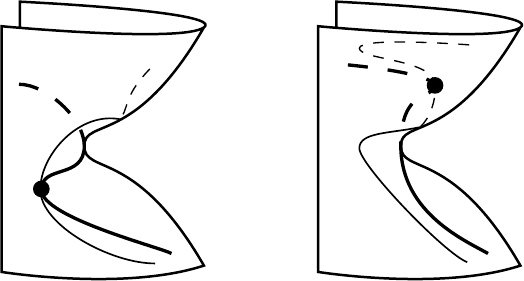}}
	\caption{Switching between cusp merge (left) and fold merge (right).}\label{circlecbd} 
\end{figure} 

The left side of Figure \ref{circlecbd} is a depiction of the saddle-shaped critical surface in $M_I$ above the dotted rectangles in Figure \ref{atprep} (for $t$ increasing left to right, it corresponds to the reversed order \hyperlink{24b}{24e}, \hyperlink{24b}{24d}, \hyperlink{24b}{24c} because it is easier to depict and validate the deformations this way). As usual, the bold arcs are cusp arcs, and in each figure $m$ lies at the saddle point of the surface, with a flipping move occurring at the dot. The fainter arcs are the immersion locus. It is is not difficult to deduce the base diagrams specified by the left side of Figure \ref{circlecbd} by considering vertical slices of the figure from left to right: the progression for the left side appears in Figure \ref{circlebd1}, which is a copy of what happened in the dotted rectangles from before. Here a fold arc experiences a flip, then one of the resulting cusps merges with another preexisting cusp. 

\begin{figure}\capstart
	\centering
	\subfloat[The reverse of this deformation appears in the sequence of three dotted rectangles in Figure \ref{atprep}.]{\label{circlebd1}\includegraphics{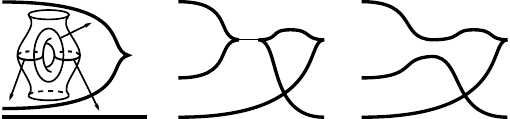}}\\
	\subfloat[The reverse of this deformation can replace the sequence of three dotted rectangles in Figure \ref{atprep}.]{\label{circlebd2}\includegraphics{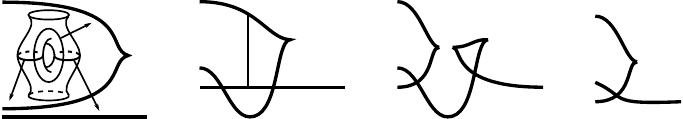}}
	\caption{Base diagrams for the left and right side of Figure \ref{circlecbd}, respectively.}\label{circlebd} 
\end{figure} 
The right side of Figure \ref{circlecbd} has base diagrams given by Figure \ref{circlebd2}, which will replace those of Figure \ref{circlebd1} in the dotted rectangles. Again, following vertical slices from left to right, one may deduce that the sequence of moves must begin with an $R_2$ deformation between a fold arc and another fold arc containing a cusp. The deformation concludes with a fold merge and an inverse flip as appears in Figure \ref{circlebd2}. For fiber genus at least 2, which will be the case according to Lemma \ref{genuslem} below, the validity of these moves and the intended substitution follow entirely from the vanishing cycles in the initial base diagrams in Figure \ref{circlebd}, which are inherited from the base diagrams of Figure \ref{atprep}. In Figure \ref{circlebd2} the initial $R_2$ deformation is valid because the vanishing cycles involved are disjoint. The following fold merge and inverse flip are straightforward to verify using the vanishing cycles inherited from the initial base diagram. To restate the prescription for turning a fold merge into a cusp merge, one first substitutes the reverse of Figure \ref{circlebd2} into the dotted rectangles of Figure \ref{atprep}, and then substitutes the resulting deformation for a neighborhood of the merge point $m$. 
 
\subsection{Modification of the critical manifold}\label{mod}
 As discussed in Section \ref{overview}, the argument begins with a given deformation $\alpha\co M_I\to S^2_I$ whose endpoints are maps that induce surface diagrams. Provided the endpoints are homotopic, a deformation between them was proven to exist in \cite[Theorem 1]{W1}, but we need something more: by a theorem of Gay and Kirby, one may assume the fibers of $\alpha$ are connected. Taking their arguments slightly further, the following lemma is the starting point for our modifications. From this point on, $\alpha$ should be taken to be the deformation that results from applying any preceding lemmas.
\begin{lemma}\label{genuslem}
Let $\alpha_0$ and $\alpha_1$ be homotopic simplified purely wrinkled fibrations, both of which have lower fiber genus at least 2. Then there is a deformation $\alpha$ from $\alpha_0$ to $\alpha_1$ such that every regular fiber is connected and has genus at least 2.\end{lemma}
\begin{proof}In the connected fibers part of the proof of Theorem 1.2 in \cite{GK1}, they explain how to use arcs in $M_I$ connecting distinct path components of fibers of $\alpha$ as a prescription for modifying the deformation in a way that results in the connected sum of the two path components at their endpoints (they use the symbol $G_s^{-1}(q)$ to denote the preimage of $p\in S^2_s$ in the deformation $G$ at time $s$). Applying this construction finitely many times to the preimages of regions of the target space that have disconnected fibers results in a deformation with connected fibers. Since their modification does not actually require the fibers to be disconnected, it can be used to inflate the fiber genus over regions in the interior of the target space. 

Here is another argument coming from an idea of David Gay, using similar ideas. Given any point $(t,x)\in S^2_I$, there is a neighborhood $U_x=(t-\epsilon,t+\epsilon)\times D^2$, sufficiently small such that $\alpha|_{\alpha^{-1}(U_x)}$ has a one-parameter family of sections disjoint from $\crit\alpha$, for some $\epsilon>0$ depending on $(t,x)$ and the chosen $D^2$. Cover $S^2_I$ with such neighborhoods and choose a finite subcover with $n$ members, such that no member is contained in another. Denote the section $s^i$ sitting over the member 
\[N^i\approx(t-\epsilon,t+\epsilon)\times D^2\] 
of this open cover. Because it is the intersection of a pair of smooth 3-balls in a 5-manifold, generically each path component of $s^i\cap s^j$ is a one-dimensional submanifold $\gamma\subset M_I$, transverse to the level sets $M_t$ at all but finitely many points. Note $\gamma$ could be a circle or could be diffeomorphic to a closed interval. Choose a subarc $\tilde\gamma\subset\gamma$ that is transverse to the level sets $M_t$ and whose endpoints are either endpoints of $\gamma$ or tangencies between 
$\gamma$ and $M_t$. Choose a one-parameter family $p$ of smooth embedded paths $p_t\subset\overline{s^i_t}\cap\alpha^{-1}_t(N^j)$ from $\tilde\gamma_t$ to $\partial\overline{s^i_t}$. This can be done by taking the family of lifts specified by $s^i$ of a one-parameter family of arcs in $\overline{N^i}\cap N^j$ connecting $\alpha_t(\tilde\gamma_t)$ to $\partial\overline{N^i}$, which is possible because of the three following facts:\begin{enumerate}\item $\alpha_t(\tilde\gamma_t)\in\overline{N^i_t}\cap N^j_t$ for all $t$ in the support of $\tilde\gamma$,\item $\partial\overline{N^i_t}\cap N^j_t\neq\emptyset$ because $N_j^t$ is not contained in $N_i^t$, and\item $\overline{N^i_t}$ is contractible.\end{enumerate} Now set 
$\tilde s^i=s^i\setminus\overline{\nu p}$, where $\nu p$ denotes a neighborhood of $p$ in $s^i\cap\alpha^{-1}(N^j)$. Now the new collection 
\[\{s^1,s^2,\ldots,s^{i-1},\tilde s^i,s^{i+1}\ldots,s^n\}\]
still projects to an open cover of $S^2_I$, but the intersection arc $\tilde\gamma$ has been removed without adding new intersections. Repeating this process finitely many times results in an open cover $\mathcal{U}$ of $S^2_I$ such that each member has a section of regular points diffeomorphic to the product of an open interval and an open disk, and these sections are pairwise disjoint in $M_I$. Now choose pairwise disjoint tubular neighborhoods $H^i$ of these sections in $M_I$. Within each of these, choose coordinates such that the deformation $\alpha$ is a trivial deformation $\R^4_I\to\R^2_I$ given by the obvious projection. The deformation $\beta$ consisting of a birth followed by its inverse in $\R^4_I$ has the same boundary fibration as the trivial deformation, so it is possible to replace those trivial deformations in the neighborhoods $H^i$ with copies of $\beta$. The result of this is to increase the fiber genus over the 3-ball bounded by each copy of $\crit\beta$. One may assume each copy of $\crit\beta$ is sufficiently close to the boundary of each member of $\mathcal{U}$ so that the fiber genus increases by at least 1 at every point away from $S^2_{[0,\varepsilon]}$, over which we may assume $\alpha_t=\alpha_0$, and away from $S^2_{[1-\varepsilon,1]}$, over which we may assume $\alpha_t=\alpha_1$, for some small $\varepsilon>0$. This procedure can be repeated once more to ensure all regular fiber components have genus at least 2.\end{proof}
\begin{rmk}\label{genusrmk}The modifications that appear from now on do not decrease fiber genus or introduce disconnected fibers in the sense that they do not increase $\max_{p\in S^2_I}\left\{\chi(\alpha^{-1}(p))\right\}$. This is achieved by not performing genus-decreasing finger moves or $R_3$ deformations in which the triangle is on the higher-genus side of all three fold arcs (see \cite[Propositions 2.5(1a) and 2.9]{W2}, respectively), and by similarly restricting modifications to the immersion locus (see Lemma \ref{r3moves}).\end{rmk}
\begin{lemma}\label{connectedcrit}Let $\alpha$ be a deformation given by Lemma \ref{genuslem}. Then there is a deformation with the same endpoints whose critical locus is connected.\end{lemma}
\begin{proof}The following argument is reminiscent of Theorem 6.1 of \cite{L1}, though variations of it have appeared from time to time, going as far back as \cite{Lev}. Denote by $A$ the path component of $\crit\alpha$ whose image under $T$ contains the minimum value of $T$, and suppose there is a path component $B\subset\crit\alpha$ distinct from $A$. Then $B$ must have an index 0 $T$-critical point because $\crit\alpha_0$ is connected, and by a homotopy of $\alpha$ one may simply push that point backward in $t$ until $T(A)\cap T(B)$ contains an open interval; choose $t_0$ in that interval. Here follows a short deformation $f$ whose endpoints are $\alpha_{t_0}$. 

If $A$ has a cusp at $t=t_0$ then there exists a curve $v\co[0,1]\to M$ from that cusp to one of the two cusps formed in the birth by which $B$ appeared, transverse to the fibers, whose intersection with $\crit\alpha$ at $t_0$ is precisely the two cusps at its endpoints. Restricting the fibration to a neighborhood of $v$, by a small perturbation near its endpoints it is possible to arrange for $v$ to specify a cusp merge as in the right side of Figure 4 of \cite{W1}. Starting with the fibration $\alpha_{t_0}$, insert the detour given by performing the cusp merge and immediately its reverse. If there is no such cusp, begin with a flip on a fold arc of $A$ to introduce two cusps in $A$, perform the merges mentioned above, then end with the inverse flip that returns the fibration to $\alpha_{t_0}$. This is the deformation $f$. It is clear that $\crit f$ is a cylinder with two $T$-critical points, one at each merge point. The endpoints of $f$ agree with $\alpha_{t_0}$, so that it is possible to insert this deformation into $\alpha$ to get a new deformation, resulting in the connect sum of $A$ and $B$. Repeating this process a finite number of times decreases the number of components of $\crit\alpha$ to one.\end{proof}

\begin{lemma}\label{cancelmorsecp}
Let $\alpha$ be the deformation resulting from Lemma \ref{connectedcrit}. Then there is a deformation with the same endpoints satisfying exactly one of the following conditions:\begin{itemize}\item If $\crit\alpha_0$ and $\crit\alpha_1$ are nonempty, then $\crit T$ consists entirely of merge points.\item If exactly one of $\crit\alpha_0$ and $\crit\alpha_1$ is nonempty, then $\crit T$ has merge points and a single birth point.\item If $\crit\alpha_0=\crit\alpha_1=\emptyset$, then $\alpha$ can be assumed to be the trivial deformation $\alpha_t=\alpha_0$, $t\in I$.\end{itemize}\end{lemma}

\begin{proof}First assume $\crit\alpha_0$ and $\crit\alpha_1$ are nonempty, so that the goal, in other words, is to remove all the birth points. By Lemma \ref{connectedcrit}, $\crit\alpha$ is connected so that $\crit\alpha$ is nonempty just before any birth move.
\begin{figure}\capstart%
	\labellist
	\small\hair 2pt
	\pinlabel $b$ at 53 92
	\pinlabel $b$ at 277 92
	\pinlabel $m$ at 265 45
	\endlabellist
	\centering
	\subfloat[A birth occurs. Cusp arcs run along the profile of the upper disk. ]{\makebox[3.6cm][c]{\includegraphics[width=3cm]{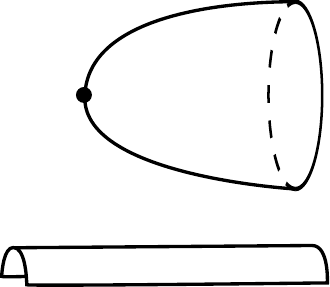}\hypertarget{27b}}\label{bi1}}\qquad
	\subfloat[A self-connect sum in $\crit\alpha$.]{\makebox[3.6cm][c]{\includegraphics[width=3cm]{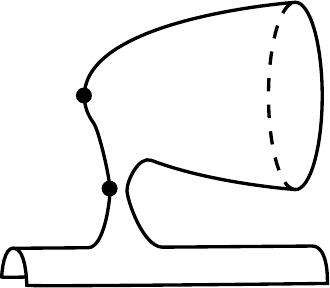}\label{bi2}}}\qquad
	\subfloat[Perturb to eliminate $b$ and $m$.]{\makebox[3.6cm][c]{\includegraphics[width=3cm]{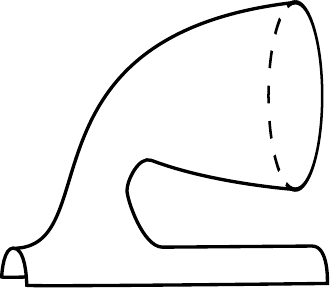}\label{bi3}}}
	\caption{Canceling critical points of $T$ in the proof of Lemma \ref{cancelmorsecp}.}\label{bi}
\end{figure}

Suppose a birth occurs, with birth point $b\in M_t$, and choose small $\epsilon>0$. As suggested by Figure \ref{bi}, and just like in the proof of Lemma \ref{connectedcrit}, perform a self-connect sum of $\crit\alpha$, where the initial cusp merge has its merge point $m\in M_{t+\epsilon}$, and occurs between one of the cusps that emanate from the birth point $b$, and a cusp point that does not lie within the newly introduced critical circle coming from $b$. Choosing $\epsilon$ sufficiently small, the deformation in a neighborhood of the cusp arc connecting $b$ and $m$ appears in base diagrams as a birth near a cusp point, immediately followed by a cusp merge between that cusp and one of the cusps in the newly introduced critical circle $C$. Now it is not hard to see that there is a homotopy of such a deformation in which $m$ moves toward the level set $M_t$, canceling $b$ as in Figure \hyperlink{27b}{27c}. This completes the argument for removing birth points from $\alpha$, resulting in a new deformation; applying the same argument with $t$ reversed removes the inverse birth points in a symmetric manner.

Two cases remain in the proof, as follows. If $b$ minimizes $T$, then $\crit\alpha_0$ is empty and one simply ignores $b$, applying the previous argument to all other birth points, with $b$ serving as the unique birth point coming from an initial stabilization of $\alpha_0$, creating a critical circle that survives through the deformation to be $\crit\alpha_1$ (one reverses the $t$ parameter if $\alpha_1$ is the side with empty critical locus). Finally, it may happen that $\crit\alpha_0=\crit\alpha_1=\emptyset$. In this case, the corresponding surface diagrams are both empty and $\alpha$ can be assumed to be the trivial deformation, $\alpha_t=\alpha_0$ for all $t\in I$, because of the uniqueness of surface bundles of genus at least three over the 2-sphere.\end{proof}

Recall the Definitions \ref{mspairdf} and \ref{shiftpairdf} for \emph{multislide pair} and \emph{shift pair}, respectively. Hereafter, the term \emph{merge pair} refers to a pair of merge points that forms either a shift pair or a multislide pair.

\begin{lemma}\label{genus}The deformation $\alpha$ resulting from Lemma \ref{cancelmorsecp} can be modified so that every index one critical point of $T\co\crit\alpha\to I$ is part of a merge pair.\end{lemma}

\begin{proof}Denote the first merge point of the deformation $\alpha$ at $t=t_1$ by $m_1\in\crit\alpha_{t_1}$. Because $\crit\alpha_0$ is connected for $t\leq t_1$, $\crit\alpha_t$ has two components for $t$ slightly larger than $t_1$. Since $\crit\alpha_1$ is connected, eventually all components much reunite at some merge point (called $m_2\in\crit\alpha_{t_2}$ below) and the first step (speaking as if there were never more than two components) is to define a pair of paths between $m_1$ and $m_2$, one in each component. More precisely, choose two smooth curves $c_{[t_1,t_2]}^i\subset\crit\alpha_{[t_1,t_2]}$, $i=1,2$, satisfying the following conditions:
\begin{itemize}
	\item $c_{t_1}^i=m_1$ and $c_{t_2}^i=m_2$ for $i=1,2$ and for some $m_2\in\crit\alpha_{t_2}$, where $t_1<t_2$.
	\item $c_t^1$ and $c_t^2$ lie in distinct path components of $\crit\alpha_t$ for $t\in(t_1,t_2)$.
\end{itemize}
The above conditions imply that the point $m_2$ is the first merge point at which the path components of $\crit\alpha_t$ containing $c_t^i$ become reunited (as above, there exists such $m_2$ because $\crit\alpha_1$ is connected). The idea of the proof is to first use the splicing deformation to introduce a cusp arc which is in some sense as isotopic as possible to the circle $c=c^1\cup c^2$, then use it in a surgery to shrink $c$ to size, which connects those path components at each $t$ except for a short interval containing $t_1$. The next step is to insert detours containing pairs of merge points near $m_1$ which allow the immersion locus to be organized such that the result is a sequence of merge pairs starting with $m_1$. The result then follows inductively, applying the same construction for the next merge point to occur after this sequence.

Recall that for $m_1,m_2$ to form a merge pair, it is necessary for a fold merge to occur at $m_1$ and a cusp merge at $m_2$. If this is not the case for some $m_i$, use Auroux's method from Section \ref{switching} to make it so, and construct $c^1,c^2$ as above.

If necessary, perturb $c^i$ to be transverse in $\crit\alpha$ to the cusp and swallowtail locus, so that $c^i$ consists of fold points except for $k^i$ transverse intersections with the cusp locus. If $k^i$ is odd for $i=1$ or $i=2$, the splicing and extending methods of Section \ref{splice} allow one to arrange that $m_1$ and $m_2$ lie within a single cusp arc. Otherwise, it is possible to arrange for the cusp arcs containing $m_1$ and $m_2$ to be adjacent in the part of $\crit\alpha_{(t_1,t_2)}$ containing the arcs $c^1$ and $c^2$ (that is, for $i=1,2$ there is an arc of fold points in $M_t$ which contains $c^i_t$ and connects the two cusp points in the critical circle $\crit\alpha_t$, for each $t\in(t_1,t_2)$). In either case, let $\gamma\subset\crit\alpha_{[t_1,t_2]}$ denote the cusp arc containing $m_2$.

After applying these modifications to the cusp locus, the next step is to choose a positive $\epsilon\ll t_2-t_1$ and a cusp merge path $p_{cm}\co[0,1]\to M_{t_1+\epsilon}$ between the cusps that reunite at $m_2$, as follows. The cusp arc $\gamma\subset M_{[t_1+\epsilon,t_2]}$ from the last paragraph has both endpoints in $M_{t_1+\epsilon}$, and $T|_\gamma$ has exactly one critical point, the merge point $m_2$. Let $\tilde p_{cm}$ denote the projection of $\gamma$ to $M_{t_1+\epsilon}$. Now, $\tilde p_{cm}\cup\gamma$ is the boundary of a smooth embedded disk $D\subset M_{[t_1+\epsilon,t_2]}$, whose intersections with the level sets $M_t$ form a one-parameter family of embedded arcs which is transverse to the fibers ($TD$ is spanned by $\frac{\partial}{\partial t}$ and a lift to $M_t$ of a nonzero tangent vector to the image of each path in $S^2_t$). Further, since $\dim D+\dim\crit\alpha<\dim M_{[0,1]}$, $D$ can be chosen disjoint from $\crit\alpha$ except along the boundary arc that coincides with $\gamma$. After possibly perturbing $D$ near the part of its boundary that coincides with $\gamma$, $D_{[t_1+\epsilon,t_2-\epsilon]}$ becomes a one-parameter family of cusp merge paths (each with an unspecified framing in the sense of \cite[Section 3.2.2]{BH}); set $p_{cm}$ to be a parameterization of the arc $D_{t_1+\epsilon}$. There is a framing of the arc $D_{t_2-\epsilon}$ that makes it a cusp merge path for $m_2$, and transporting this framing back across $D$ to lie in $M_{t_1+\epsilon}$ also turns $p_{cm}$ into a framed cusp merge path between the endpoints of $\gamma$. This is the unique framing such that, if one were to insert the detour given by performing the cusp merge according to $p_{cm}$ (producing a merge point $m_1'$ in, say, $M_{t_1+2\epsilon}$) and then immediately its reverse (a fold merge producing the merge point $m_1''\in M_{t_1+4\epsilon}$), in a neighborhood of $D\cap M_{[t_1+3\epsilon,t_2]}$ the deformation would be given by the concatenation of local models for the fold merge at $m_1''$ followed precisely by its reverse at $m_2$. The fibration structure near $D\cap M_{[t_1+3\epsilon,t_2]}$ could then be replaced by one whose base diagrams all consist of a pair of parallel fold arcs, leaving the merge points $m_1$ and $m_1'$ and no others. For those who are not satisfied with the claim that the fibration structure near $D\cap M_{[t_1+3\epsilon,t_2]}$ is itself a detour, and so can be replaced as claimed, the replacement can be explicitly achieved using a homotopy of $\alpha$ supported on a fibered neighborhood of $D$, parameterized by $s\in[0,1]$, modeled on
\begin{equation}\label{eq:H_st}H_{s,t}\co(s,t,x_1,x_2,x_3,x_4)\mapsto (t,x_1,x_2^3+3((1-s)(1-2t^2)-s-x_1^2)x_2+x_3^2-x_4^2).\end{equation}
This model comes from beginning with the fold merge deformation
\[FM_t\co(t,x_1,x_2,x_3,x_4)\mapsto(t,x_2^3+3(t-x_1^2)x_2+x_3^2-x_4^2),\ t\in[-1,1]\]
\cite[Equation 8]{W1}, doubling it along its terminal fibration $FM_1$ by replacing the $t$ in the coordinate
\[x_2^3+3(t-x_1^2)x_2+x_3^2-x_4^2\]
with $1-2t^2$ to obtain
\[x_2^3+3((1-2t^2)-x_1^2)x_2+x_3^2-x_4^2,\]
then interpolating the resulting deformation from $H_{0,t}=FM_{1-2t^2}$ toward the deformation $H_{1,t}=FM_{-1}$ by scaling the timelike parameter $1-2t^2$ by $1-s$. It is important to note that the presence of immersion arcs intersecting $c^1$ or $c^2$ does not affect the validity of this construction, because this is a modification of the fibration structure of $M_I$ induced by $\alpha$ only near $D$.

Choosing $\epsilon$ small enough, there are no birth, merge or flipping moves in the interval $(t_1,t_1+2\epsilon)$ between the two merge points $m_1,m_1'$ because they are supported in arbitrarily small balls in $M_I$ that can be assumed disjoint from $M_{t_1}$. However, Condition (\hyperlink{mscon[ii]}{ii}) for multislides or (\hyperlink{shiftcon[ii]}{ii}) for shifts is not automatically satisfied: using for instance the labels from Figure \ref{shiftdef}, there may be cusp-fold crossings involving $\chi_1$ as it travels along its cusp merge path $p_{cm}$ (which could even have self-intersections in its image in $S^2$), or fold crossings involving $\tilde\varphi_k$ that exist throughout the interval $[t_1,t_1+2\epsilon]$. Both of these are forbidden in Definitions \ref{mspairdf} and \ref{shiftpairdf}. One way to address this is to convert $(m_1,m_1')$ into a sequence of shift pair candidates (recall the end of Definition \ref{shiftpairdf}), then give an algorithm for achieving Condition (\hyperlink{shiftcon[ii]}{ii}) for these candidates. Once this is done, Condition (\hyperlink{shiftcon[i]}{i}) is easily verified.

Suppose $k$ was odd for both $c^1$ and $c^2$, so that $m_1$ and $m_1'$ are contained in a circle of cusp points like at the top left of Figure \ref{shiftcand}. To turn the pair into two pairs of merge points, each of which is a shift pair candidate, insert a detour given by a flip at $t=t_1-\epsilon$, followed by its reverse, so that there appears a small loop just to the side of either of the cusps that forms at $m_1$, then immediately perform its reverse. This creates a pair of cusp arcs in $\crit\alpha$ that meet in a pair of swallowtails. Use Lemma \ref{swalpushlem} to extend the local model for the inverse flip forward in $t$ past $m_1'$, resulting in the left side of Figure \ref{shiftcand}. Now there is a cusp arc $\gamma'$ (whose endpoints are the newly introduced swallowtails) adjacent in the fold locus to the circle containing $m_1$ and $m_1'$. Using an arbitrary cusp merge path, perform a self-connect sum in $\crit\alpha$ between $\gamma'$ and the cusp circle containing $m_1$ and $m_1'$ using the same kind of detour as for Figure \hyperlink{27b}{27b} (such a cusp merge path exists by \cite[Proposition 2.7]{W2}); the result is a pair of shift pair candidates.
\begin{figure}\capstart%
	\labellist
	\small\hair 2pt
	\pinlabel $\gamma'$ at 0 26
	\pinlabel $\searrow$ at 7 18
	\endlabellist
	\centering{\includegraphics{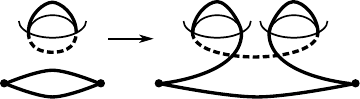}}
	\caption{Turning a multislide pair candidate into two shift pair candidates. The dots are swallowtails and the bold arcs are cusps.}\label{shiftcand} 
\end{figure} 

Now suppose it is not the case that $m_1$ and $m_1'$ are connected by a cusp arc, so that the cusp merge path $p_{cm}$ connects cusps $\chi_1',\chi_2'$ that are adjacent to the cusp points $\chi_1$ and $\chi_2$ (respectively) emanating from $m_1$ (left side of Figure \ref{shiftcand2}). Modify this as above by performing a self connect sum in $\crit\alpha$ according to a cusp merge path between, say, $\chi_1$ and $\chi_2'$. This also results in a pair of shift pair candidates, and this concludes the list of cases for how the cusp arcs containing $m_1$ and $m_1'$ lie within the fold locus for $t\in[t_1,t_1+2\epsilon]$.
\begin{figure}\capstart
	\centering{\includegraphics{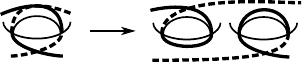}}
	\caption{Taking a connect sum between points in the two solid cusp arcs produces a pair of shift pair candidates.}\label{shiftcand2} 
\end{figure}

We may now assume without loss of generality that the pair of merge points under consideration, $m_1$ and $m_1'$, form a shift pair candidate, and now is a good time to state the notation, using when possible the labels from Figure \ref{shiftdef}. There is a fold merge at $m_1\in M_{t_1}$ between fold arcs $\varphi_k$ and $\varphi_l$ forming two cusp points $\chi_1$ and $\chi_2$, and a cusp merge at $m_1'\in M_{t_1+2\epsilon}$ according to an arbitrary cusp merge path $p_{cm}\co[0,1]\to M_{t_1+\epsilon}$ which travels from $\chi_1$ to $\chi_2'$, where $\chi_2$ and $\chi_2'$ bound $\tilde\varphi_k$. Finally,  let $p_{fm}\co[0,1]\to M_{t_1-\epsilon}$ be the fold merge path for $m_1$, oriented so that $p_{fm}(1)\in\varphi_k$. The rest of the proof applies to any such pair of merge points, so it suffices to consider only this pair.

By definition, the shift pair candidate generally fails to be a shift pair because the image of the fold arc $\tilde\varphi_k$ (and possibly the cusp merge path $p_{cm}$) may cross itself or other fold arcs. The list below summarizes an algorithm to address these issues. The algorithm produces many shift pair candidates, successively more similar to a sequence of shift pairs. At each step, there are fold arcs, cusps, and merge paths that play analogous roles to those in previous steps. The following argument refers to all of these objects using the labels from Figure \ref{shiftdef} to streamline the notation, but is careful to keep track of necessary distinctions. The result is that $m_1$ is the first merge point in a sequence of local shift deformations interspersed with many 2-parameter crossings. 

\begin{enumerate}
\item[(\hyperlink{genusitem1f}{1})]\hypertarget{genusitem1} By a homotopy of $\alpha$, shorten $p_{cm}$ so that the cusp $\chi_2'$ between $\varphi_1$ and $\varphi_k$ lies close to $\chi_1$ (that is, pointing into the same region of regular values) as soon as it appears. This causes the new $\tilde\varphi_k$ to possibly have fold crossings and causes the new $p_{cm}$ to be embedded and disjoint from the critical image of $\alpha$.
\item[(\hyperlink{genusitem2f}{2})]\hypertarget{genusitem2}Break the deformation into a sequence of shift pair candidates that do not have self-crossings in their respective $\tilde\varphi_k$ fold arcs, but may have fold crossings between their respective $\tilde\varphi_k$ and other fold arcs. Do this in a particular way to allow item (4) below. 
\item[(\hyperlink{genusitem3f}{3})]\hypertarget{genusitem3}For a shift pair candidate resulting from item (2), modify $\alpha$ by a homotopy so that $\tilde\varphi_k$ is free of all fold crossings, converting it to a shift pair.
\item[(\hyperlink{genusitem4f}{4})]\hypertarget{genusitem4}Immediately after the sequence of shift pairs from item (3), perform a sequence of 2-parameter crossings to return $\alpha$ to the endpoint of the original shift pair candidate deformation.\end{enumerate}

\hypertarget{genusitem1f}Item (\hyperlink{genusitem1}{1}) is an extension of the modification involving $D$ and $H_{s,t}$ above. Alternatively, but less rigorously, one may view it as pushing $\chi_2'$ along most of the reverse of $p_{cm}$ before $t_0$, so that $\chi_1$ is pointing into the same region of regular values as $\chi_2'$ as soon as it appears. Recall that the construction of the shift pair candidate resulted in a cusp merge in which $\chi_1$ traveled along a cusp merge path $p_{cm}$ toward the stationary cusp $\chi_2'$ (recall the labels from Figure \ref{shiftdef}). Insert a detour supported near $p_{cm}$ in which $\chi_2'$ travels along the reverse of $p_{cm}$ (as specified by the framing of $p_{cm}$ as a joining curve), then returns to where it started. Now delay the return of $\chi_2'$ so that $\chi_1$ trails it closely during its own traversal of $p_{cm}$. The part of $p_{cm}$ that remains between the two cusps as they move together along $p_{cm}$ is a one-parameter family of joining curves, so it traces out a disk in $M_{[0,1]}$ that serves an analogous role to $D$ above, to which the deformation $H_{s,t}$ may be applied. The fold merge at $m_1$ now produces cusps $\chi_1,\chi_2$ that point into the same region of regular points as $\chi_2'$. The cusp-fold crossings originally undergone by $\chi_1$ are now cusp-fold crossings undergone by $\chi_2'$ at $t$-values less than that of $m_1$.

\hypertarget{genusitem2f}Item (\hyperlink{genusitem2}{2}) involves inserting a flip near each self-intersection of $\tilde\varphi_k$, and using its two cusps to break a candidate into three candidates, thereby breaking an immersed $\tilde\varphi_k$ into a collection of fold arcs, each serving as an embedded $\tilde\varphi_k$ for a member of a sequence of shift candidates. This addresses self-crossings in $\tilde\varphi_k$; the remaining crossings are treated in item (\hyperlink{genusitem3}{3}). Because of item (\hyperlink{genusitem1}{1}), assume without loss of generality that $\alpha(p_{cm})$ is a short embedded arc of regular values in $S^2_{t_0+\epsilon}$ connecting the cusp points $\alpha(\chi_1),\alpha(\chi_2')$. To begin, modify the fibration near each self-crossing of $\tilde\varphi_k$ according to two cases that correspond to the two types of self-crossings available to a fold arc. 
These cases appear in Figure \hyperlink{30b}{30}: Each self-crossing in $\tilde\varphi_k$ receives an additional loop coming from a flip at $t=t_1-\epsilon$ at the location suggested by the figure, which later goes away by an inverse flip at $t=t_1+3\epsilon$ (as usual, one adds these flips by inserting a flip immediately followed by its inverse and extending the local model for each move forward and backward in $t$ by Lemma \ref{swalpushlem}). In either case, instead of traveling along $p_{cm}$, the cusp $\chi_1$ will first undergo cusp-fold crossings that exist by \cite[Proposition 2.7]{W2} to lie just to the higher genus side of cusp 1, following a path whose image is the dotted line in the figure, and which is transverse to the fibers and disjoint from $\crit\alpha$. Perform the detour consisting of a cusp merge between $\chi_1$ and cusp 1, followed by its inverse, so that the fold that is parallel to the dotted line in Figure \hyperlink{30b}{30} now serves as the fold arc $\tilde\varphi_k$ for the first of three candidates to be produced by the modification. 
Similarly, send $\chi_1$ along a path that follows the fold arc between cusps 1 and 2, and perform the same kind of detour with $\chi_1$ and cusp 2. This happens all along the critical arc until $\chi_1$ cusp merges with the cusp at the end of $\varphi_1$: $\chi_1$ follows along $\tilde\varphi_k$, performing cusp-then-fold-merge moves at each cusp it encounters. This adds two candidates to the deformation for each self-intersection of the original $\tilde\varphi_k$. If any flips are inserted in this item, then the fold arcs trailing behind $\chi_1$ will be different from that indicated by the short embedded path $p_{cm}$ (and this is the only difference between the ending maps, because the merges were detours). Item (\hyperlink{genusitem4}{4}) addresses this. Before that, however, item (\hyperlink{genusitem3}{3}) turns all of these candidates into shift pairs.
\begin{figure}\hypertarget{30b}\capstart%
	\labellist
	\small\hair 2pt
	\pinlabel $\varphi_1$ at 5 105
	\pinlabel $\chi_2'$ at 50 93
	\pinlabel $\tilde\varphi_k$ at 5 75
	\pinlabel $\varphi_l$ at 5 8
	\pinlabel 2 at 93 62
	\pinlabel 1 at 93 45
	\pinlabel $\chi_2$ at 102 18
	\pinlabel $\chi_1$ at 132 18
	\pinlabel $\varphi_1$ at 205 105
	\pinlabel $\chi_2'$ at 250 93
	\pinlabel $\tilde\varphi_k$ at 205 75
	\pinlabel $\varphi_l$ at 205 8
	\pinlabel $\chi_2$ at 304 18
	\pinlabel $\chi_1$ at 332 18
	\pinlabel 2 at 271 27
	\pinlabel 1 at 255 24
	\endlabellist
	\centering
	\subfloat[\ ]{\makebox[5.5cm][c]{\includegraphics{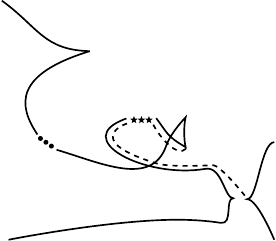}\label{badloopa}}} \hspace{1.5cm} \capstart
	\subfloat[\  ]{\makebox[5.5cm][c]{\includegraphics{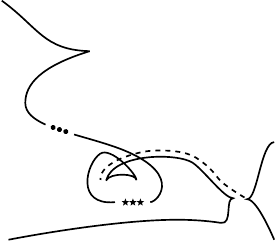}\label{badloopb}}}
	\caption*{Figures for item (\protect\hyperlink{genusitem2}{2}) in the proof of Lemma \ref{genus} giving the placement of a flip. The text of item (\protect\hyperlink{genusitem4}{4}) refers to the star ellipsis in each.}
	\subfloat[\ ]{\makebox[5.5cm][c]{\includegraphics{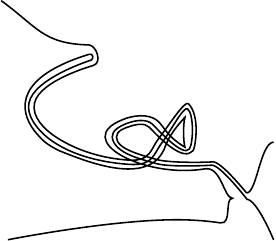}\label{notsoevila}}} \hspace{1.5cm} \capstart
	\subfloat[\  ]{\makebox[5.5cm][c]{\includegraphics{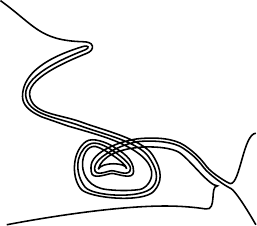}\label{notsoevilb}}}
	\caption{\label{genuspics}The fibrations which result from the added merge pairs indicated in the two figures above, assuming $\tilde\varphi_k$ had just one self intersection. Each is not the original fibration $\alpha_{t_1+2\epsilon}$; the discrepancy is addressed by item (\protect\hyperlink{genusitem4}{4}).}
\end{figure}

\hypertarget{genusitem3f} Item (\hyperlink{genusitem3}{3}) addresses the fold crossings in $\tilde\varphi_k$, which by this stage are not self-crossings. Assuming items \hyperlink{genusitem1}{1} and \hyperlink{genusitem2}{2} have been applied, the given shift pair candidate has a cusp merge path whose image closely follows $\varphi_k$ as in the dotted path of Figure \ref{shiftdef}, with the possible addition of other critical arcs crossing $\varphi_k$. In particular, item (\hyperlink{genusitem3}{3}) applies to the first and third candidates of any triplet produced in item (\hyperlink{genusitem2}{2}), while the second candidate needs no modification and is already a shift pair. During a $t$ interval before the fold merge, the idea is to push any crossings in $\varphi_k$ so that they do not that end up in $\tilde{\varphi}_k$ once the fold merge occurs, then perform the local shift deformation, then return those crossings to where they started. The merges and the movements of crossings have disjoint supports called $U_{me}$ and $U_{cr}$, respectively.

The first step is to define $U_{me}$ and $U_{cr}$. The set $U_{me}$ will be the support of the shift candidate (\emph{me} stands for \emph{merge}), while $U_{cr}$ will be the support of a homotopy described below that pushes the fold crossings away from the part of $\varphi_k$ that will become $\tilde\varphi_k$ (\emph{cr} stands for \emph{crossings}). The second step is to prove that such a homotopy exists, and that one may arrange for $U_{me}$ and $U_{cr}$ to be disjoint, so that the proposed deformation, a local generalized shift deformation occurring between some 2-parameter crossings, has the same endpoints as (and thus can be substituted for) the candidate we started with. 

To set notation, pick $\epsilon>0$ and call the merge points of the candidate $m_1\in M_{t_1}$, $m_1'\in M_{t_1+2\epsilon}$, using the same labels as in Figure \ref{shiftdef}. To precisely define $U_{me}$, it helps to refer to Figure \ref{pushed}. The initial fold merge is supported on a tubular neighborhood of the fold merge path $p_{fm}\co[0,1]\to M_{t_1-\epsilon}$ between $\varphi_k$ and $\varphi_l$, and the subsequent cusp merge is supported on a tubular neighborhood of the cusp merge path $p_{cm}\co[0,1]\to M_{t_1+\epsilon}$ from the resulting cusp on the right, $\chi_1$, to the cusp between $\varphi_1$ and $\tilde\varphi_k$ at $p_{cm}(1)$. Now, as in Remark \ref{mergepath}, $p_{fm}$ specifies a family of merge paths $\left[p_{fm}(s)\right]_{[t_1-\epsilon,t_1)}$. Thus, $\left[p_{fm}(1)\right]_{t_1-\epsilon}$ is connected to $\left[p_{cm}(0)\right]_{t_1+\epsilon}$ by the union of $\left[p_{fm}(1)\right]_{[t_1-\epsilon,t_1]}$ and a cusp arc $\gamma\subset M_{[t_1,t_1+\epsilon]}$. Denote the projection $\pi_t\co M_I\to M_t$ and the concatenation of the four paths by $P$: 
\[P=\left[p_{fm}([0,1])\right]_{t_1-\epsilon}\ast\left[p_{fm}(1)\right]_{[t_1-\epsilon,t_1]}\ast\gamma\ast\left[p_{cm}([0,1])\right]_{t_1+\epsilon}\subset M_{[t_1-\epsilon,t_1+\epsilon]}.\]
Denote a tubular neighborhood of this path in $M_{t_1}$ by $\nu\pi_{t_1}(P)$, and define \[U_{me}=\left\{(t,m)\in M_{[t_1-\epsilon,t_1+2\epsilon]}:\pi_{t_1}(m)\in\nu\pi_{t_1}(P)\right\}.\]This is essentially the product of $[t_1-\epsilon,t_1+\epsilon]$ with a neighborhood of $\pi_{t_1}(\gamma)$, slightly extended to contain part of $\varphi_l$.

\begin{figure}\capstart
	\labellist
	\small\hair 2pt
	\pinlabel $\varphi_1$ at 3 120
	\pinlabel $\varphi_k$ at 163 50
	\pinlabel $\varphi_l$ at 163 5
	\endlabellist
	\centering{\includegraphics{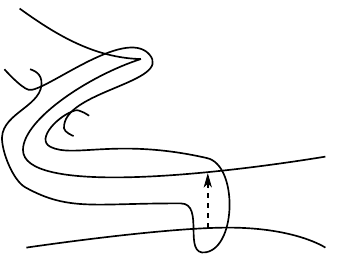}}
	\caption{Applying item (\protect\hyperlink{genusitem3}{3}) to push two fold crossings off the relevant part of $\varphi_k$. The dotted line is $\alpha(p_{fm})$.}\label{pushed} 
\end{figure}

To define $U_{cr}$, it is necessary to describe the deformation of which it is the support. Begin by traveling downward along $\varphi_k$ away from the cusp $\chi_2'$ in Figure \ref{pushed}. Consider the first fold crossing one encounters between $\varphi_k$ and some fold arc $A$. If $\chi_2'$ is not on the lower-genus side of $A$, move on to the next crossing. Otherwise, push the crossing between $A$ and $\varphi_k$ upward along $\varphi_k$ past $\chi_2'$. More precisely, there is a deformation as in Figure \ref{tinypair} in which $A$ is the vertical fold arc and the fold arcs $\varphi_1,\varphi_k$ are respectively to the left and right of the cusp, in effect pushing the crossing off what will be $\tilde\varphi_k$ to lie in $\varphi_1$. The initial cusp-fold crossing exists by \cite[Proposition 2.7]{W2} and the subsequent $R_2$ deformation exists by \cite[Proposition 2.5(3)]{W2}. The intersection point has now been moved to lie in $\varphi_1$, as shown in Figure \ref{pushed}. This procedure can be repeated for all crossings oriented like $A$: Use  \cite[Proposition 2.5(1a)]{W2} to perform finger moves across any other arcs that might intersect $\varphi_k$, and then use \cite[Proposition 2.9]{W2} to perform $R_3$ deformations, to move each crossing up over the cusp. Now move the remaining crossings (for which $\chi_2'$ is on their higher-genus sides) downward away from $\chi_2'$, just past $\alpha(p_{fm}(1))$, then perform finger moves of the offending fold arcs downward along the image of $\alpha(p_{fm})$ to cross $\alpha(p_{fm}(0))$. An illustration of what would result from pushing the one of each of the two types of intersections out of the way appears in Figure \ref{pushed}. The support of the finger move pushing a fold arc down to intersect $\varphi_l$ is a neighborhood of a pair of disks which can be chosen disjoint from a neighborhood of $p_{fm}$, and the disjointness condition of \cite[Definition 2.1]{W2} and \cite[Theorem 2.4]{W2} implies that the support of the rest of this deformation is disjoint from the part of $\varphi_k$ whose image runs parallel to $\alpha(p_{cm})$. For this reason, it is possible to arrange for Condition (\hyperlink{shiftcon[ii]}{ii}) in the definition of local shift deformation by first pushing off the intersections, then performing the pair of merge moves, and then reversing the deformation that pushed off the intersections. Once this is achieved, Condition (\hyperlink{shiftcon[i]}{i}) is verified by simply examining the base diagrams in the image of $U_{me}$ just after pushing off the intersections. In this way, the initial shift candidate is broken into a sequence of local shift deformations.

\hypertarget{genusitem4f}The last step, item (\hyperlink{genusitem4}{4}), takes as input the fibration that results from the sequence of local shift deformations coming from Item (\hyperlink{genusitem3}{3}) and sends the pair of fold arcs that trailed behind $\chi_1$ (which may have many loops as in Figures \hyperlink{30b}{30c} and \hyperlink{30b}{30d}) back to coincide with the fold arcs that trailed behind $\chi_1$ at the end of item (\hyperlink{genusitem1}{1}) (which follow a crossing-free path from where the initial fold merge occurred to where the final cusp merge occurred). Once this item is applied, the fibration structure just after the final cusp merge agrees with the fibration structure at the analogous $t$ value before the algorithm began (that is, immediately after the shift candidate). The algorithm is defined to begin with the same fibration as the beginning of the given shift candidate, so the result can be substituted into $\alpha$ to satisfy the requirements of Lemma \ref{genus}.

Consider $\tilde\varphi_k$, before the flips are introduced by item \hyperlink{genusitem2}{2}, as a path oriented from $\chi_2$ to $\chi_2'$. The two fold points that map to the first self-intersection of $\tilde\varphi_k$ bound a fold arc that maps to an oriented based loop; call it $\ell\co[0,1]\to S^2$. Within that loop, the first self-intersection to occur in $\ell|_{(0,1)}$ singles out a smaller loop within $\ell$, and so on, until one reaches an innermost loop. Item \hyperlink{genusitem3}{3} puts a flip loop nearby, for which the picture is just like Figures \ref{badloopa} and \hyperlink{30b}{30b}, though the fold arc is not interrupted by a star ellipsis.

Note that the flips suggested by Figure \hyperlink{30b}{30} double the number of crossings in $\tilde\varphi_k$, so that the new path taken by the image of $\chi_1$ in $S^2$ along the original $\tilde\varphi_k$,  plus the flip loops from item (\hyperlink{genusitem2}{2}), is isotopic to the original short embedded path $p_{cm}$ from $\chi_1$ to $\chi_2'$ resulting from item (\hyperlink{genusitem1}{1}) via a sequence of only Reidemeister-II and Reidemeister-III moves, a claim which this paragraph serves to establish. If an innermost loop is as in Figure \hyperlink{30b}{30b}, Figure \ref{badloopbslip} illustrates a way to remove its self-intersections using a trick that is reminiscent of the stabilization deformation of Figure \ref{stabdef} (the third entry in Figure \ref{badloopbslip} is the result of pushing a proper subinterval of the short arc between the two intersections around the back of $S^2$). The 2-parameter crossings indicated by the figure all exist by \cite[Proposition 2.11]{W2}: The strip between the two parallel fold arcs trailing behind $\chi_1$ is on the higher-genus side of each, with vanishing cycles on its two sides that correspond to canceling Morse critical points, so if one fold arc is able to initially undergo a 2-parameter crossing, then the next is able to follow behind in a deformation supported in the same ball of regular points in $M$ that supported the movement of the first: That ball is a neighborhood of a one-parameter family of fold merge paths from one side of the strip to the other. For this reason, it suffices to establish the movements for the leading fold arc. For the first arrow in Figure \ref{badloopbslip}, use \cite[Propositions 2.5(1a), 2.9, and 2.5(3)]{W2}, respectively. For the second arrow, use \cite[Proposition 2.11]{W2}. For the third, use \cite[Proposition 2.5(3)]{W2} twice. If an innermost loop is as in Figure \ref{badloopa}, the same argument applies, without the need to sweep the pair of fold arcs around the back of $S^2$ (see Figure \hyperlink{30b}{30c}): just use \cite[Proposition 2.5(3)]{W2} to contract the bigon.

\begin{figure}\capstart%
	\centering{\includegraphics{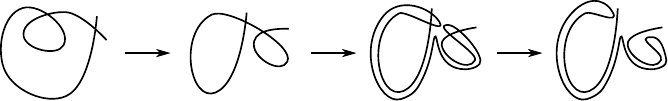}}
	\caption{A schematic for resolving intersections coming from Figure \protect\hyperlink{30b}{30b}. Keep in mind each arc represents two parallel fold arcs, oriented so that the strip between them lies on the higher-genus side of each. This movement occurs in the interval between one candidate pair and the next.}\label{badloopbslip}
\end{figure}

Once these movements are applied to the innermost loop, they apply to the next innermost loop, and so on inductively until the two parallel fold arcs are free of crossings. Further move the pair of arcs so that their image coincides with the image of the pair of arcs as they appeared at the end of item \hyperlink{genusitem1}{1}: a short, embedded path between the points previously occupied by $\chi_1$ and $\chi_2'$. Now the deformation consists of a sequence of local generalized shift deformations instead of the single shift candidate given in item \hyperlink{genusitem1}{1}, followed by a sequence of 2-parameter crossings. Because the cusp and fold merges that occurred between the first and last merge points were constructed as detours, the result is a fibration that agrees with the ending of the candidate yielded by Item \hyperlink{genusitem1}{1}, except that even though the pair of fold arcs that trailed behind $\chi_1$ have the same image as the original pair resulting from item \hyperlink{genusitem1}{1}, they may not be isotopic to the original pair in $M$ (relative to their endpoints). This is easily rectified by adding one more local multislide deformation whose initial fold merge is between these fold arcs.
\end{proof}
 
\subsection{Modification of the stratified immersion of the critical manifold}\label{immer}
At this point, $\alpha$ has some of the characteristics of the required deformation: the Morse function $T\co\crit\alpha\to I$ and the part of the cusp locus containing its critical points are as required; however, the base diagrams can still appear disorganized because of self-intersections in the critical image. The object of this section is to modify the stratified immersion $\crit\alpha\rightarrow S^2_I$ until the remaining double points all come from the model deformations for the moves in Theorem \ref{T}.

\subsubsection{Immersion loci of deformations}\label{cbd}
It will be necessary to collect some facts before beginning the main argument of this section. To summarize the results of Section \ref{mod}, $\crit\alpha_t$ is now a single immersed circle that, as $t$ progresses, can momentarily split in two at merge pairs. The only restrictions on $\iota$ at this point in the paper come from Definitions \ref{mspairdf} and \ref{shiftpairdf}: For a multislide pair, the two cusps involved in the deformation sweep out a circle in $\crit\alpha$ that is required to be disjoint from $\iota$; for shift pairs, the same circle in the critical twice-punctured torus is disjoint from $\iota$: The cusp $\chi_1$ and the fold arc $\tilde{\varphi}_k$ from Figure \ref{shiftdef} remain disjoint from $\iota$ for those $t$-values at which $\chi_1$ exists. One final way to describe it: Neglecting the cusp locus, one could specify $(\crit\alpha,\iota)$ as coming from a cylinder $S^1_{[0,1]}$ decorated with circles coming from $\iota$ and pairs of points $(p_i,q_i)\in S^1_{t_i}$, in the complement of these circles, at which the cylinder undergoes a self connect sum, with each $t_i$ corresponding to a single merge pair.

To avoid clutter in the exposition, it will be convenient to include swallowtail points in the immersion locus throughout the rest of the paper.

\begin{lemma}\label{circlepairs} For deformations whose critical locus is connected at each value of $t$, an immersion circle $c$ is one of a pair of circles which are mapped to each other by $\alpha$ if and only if $c$ is free of swallowtail points. In this situation, $c$ and any subarc $s\subset c$ such that $\partial s$ is a self-crossing are nullhomotopic in $\crit\alpha$.\end{lemma}
\begin{proof}\ 

$(\Rightarrow)$: if $c$ contained a swallowtail point, one could choose a point on $c$ and mark its counterpart on the other circle. One could then follow the one-parameter family of double points along the immersion arcs until one of them enters a neighborhood of the swallowtail point as in Figure \ref{flip}, and here we come across a one-parameter family of triple points (or higher), the third arc coming from the other side of the immersion arc passing through the swallowtail point, which cannot be simplified by a small perturbation. This is a contradiction because in a deformation the only triple points are those coming from Reidemeister type three moves, which result in isolated triple points. Thus, $c$ is free of swallowtail points.

$(\Leftarrow)$: By continuity and the local models for fold crossings in \cite[Section 2]{GK1}, if $c$ is not one of a pair of circles with common $\alpha$-image then it is a single circle mapped by $\alpha$ in a two-to-one fashion outside of a finite subset of points where a small loop in the critical image forms or collapses, which necessarily occurs at a swallowtail point. 
%Arbitrarily choosing an orientation of $c$, $\frac{\partial}{\partial t}$ projects to oppositely oriented tangent vectors at a chosen pair of points of $c$ with a common image. A swallowtail may be located by tracing the pair (as a one-parameter family of double points) along $c$ in the direction specified by the two tangent vectors, so that the two points must eventually coincide at a swallowtail point.

For the last statement, let $p$ denote the $T$-minimal point of $s$ or $c$. For a contradiction, restrict attention to a minimal subarc $\tilde{s}$ containing $p$ and violating the lemma. Choose $p_1\in\tilde{s}$ with $T(p_1)=t_1$ such that $x=\alpha(p_1)$ is one corner of the bigon formed in the critical image immediately after the Reidemeister-II fold crossing or cusp-fold crossing corresponding to $p$. For $p_1\in\crit\alpha_{t_1}$, denote by $\pi_{t_1}$ the projection of $\tilde{s}\subset\crit\alpha=[0,1]\times S^1$ to $\crit\alpha_{t_1}$. This is a degree-1 map $S^1\to S^1$. In particular, the image of $\pi_{t_1}$ contains a short fold arc $\gamma$ containing $p_1$ and the image of $\alpha\circ\pi_{t_1}$ contains a short arc in the critical image transverse to $\alpha(\gamma)$ at $x$. But this implies $\tilde{s}$ contains the counterpart of some $p_2\in\tilde{s}$, contradicting the fact that $\tilde{s}$ is contained in a pair of circles with common $\alpha$ image.
\end{proof}

\begin{df}Lemma \ref{circlepairs} suggests the terminology \emph{immersion pair} for a pair of circles in the critical locus which have the same image, and \emph{immersion single} for a circle of immersion points containing at least one swallowtail point.\end{df}

It is time to further classify immersion circles and the critical points in their complement.

\begin{df}Each point in $p\in\crit\alpha\setminus\iota$ has an associated genus as follows. Choose a regular fiber $F_p$ just to the higher genus side of $\alpha(p)$. Define \emph{the genus at $p$} as the genus of $F_p$.\end{df}
\begin{df}\label{gdgidef}Fix a short path in $\crit\alpha$, parameterized by $t$, that transversely crosses an immersion circle at the immersion circle's minimal $t$-value. Where the path crosses the immersion circle, the genus of the regular fiber increases or decreases by 1.\begin{itemize}\item If the genus decreases, call the circle \emph{genus-decreasing}. If it increases, call it \emph{genus-increasing}.\item Immersion pairs composed of genus-increasing (resp. decreasing) circles are called \emph{genus-increasing pairs} (resp. \emph{genus-decreasing pairs}). \item If the genus increases at one member of an immersion pair but decreases at the other, call the pair \emph{mixed}.\item Disregarding the global aspects of immersion circles, there are notions of genus-increasing, genus-decreasing, and mixed $R_2$ deformations in which two fold crossings form with increasing $t$. Similarly, there are mixed and genus-increasing cusp-fold crossings. This explains the modifiers \emph{genus-increasing, genus-decreasing}, or \emph{mixed} for an $R_2$ deformation that forms a pair of fold crossings.\end{itemize}\end{df}

\begin{rmk}\label{tracingorientations}Some examples of this terminology: The Reidemeister-II fold crossing that initiates a handleslide deformation is genus-decreasing, and the disk bounded by the loop in the bottom right of Figure \ref{swalpush} lies on the higher-genus side of the immersion arc that forms its boundary. In other words, the loop appears by a genus-increasing $R_2$ deformation.

On a related note, it will be useful to \emph{trace orientations} near any crossing triple in $\iota$ as follows. Suppose in Figure \hyperlink{16b}{16c} that $a$ is oriented so that the genus is decreasing as one crosses $a$ in the increasing $t$ direction. This actually comes from the orientation of the strip of fold points passing through the page along $a$, so that the genus decreases as one travels along $b$ across $a$ with increasing $t$. The orientation of this strip also causes the genus to decrease traveling along $b'$, so that $c'$ gets an orientation: The genus decreases traveling across $c'$ with increasing $t$.\end{rmk}

\subsubsection{Unlinking $\iota$}\label{unlink}
 Here begins the process to simplify $\iota$. The first step is to reduce to considering a single immersed critical cylinder, rather than a pair of cylinders or a higher-genus surface, by \emph{unlinking} $\iota$ from the merge pairs, to be isotoped off the small intervals of $t$ they bound.
\begin{lemma}[Ambient isotopy lemma]\label{isotopecrit}For $s\in[0,1]$ choose an ambient isotopy $i^{S^2}_s$ of $S^2_I$ that does not introduce or eliminate tangencies between the image of the fold or cusp locus and the level sets $S^2_t$. Let $C\subset S^2_I$ be the stratified critical image of $\alpha$. Then there is a homotopy $\alpha_s$ of the deformation $\alpha$ such that $\alpha_s(\crit\alpha_s)=i^{S^2}_s(C)$.\end{lemma}
\begin{proof}Give $i^{S^2}_s$ the coordinates $i^{S^2}_s(t,x)=(t_s,x_s)$. The map $i_s^{S^2}\circ\alpha$ is a map whose critical image is $i^{S^2}_s(C)$. Define the ambient isotopy $i_s^M$ of $M_I$ that changes the $t$-coordinate of a point the same way $i_s$ changes the $t$-coordinate of its $\alpha$-image: $i^M_s(t,m)=(t_s,m)$. For each $s$, the tangency condition ensures all tangencies between the critical image of $i_s^{S^2}\circ\alpha$ and level sets $S^2_t$ occur at merge points (recall births have been eliminated by now), and that all critical points have the local models required by deformations, so it only remains to make sure the maps $\alpha_s$ are the identity on the $t$ coordinate. For this reason, $\alpha_s=i^{S^2}_s\circ\alpha\circ(i^M_s)^{-1}$ is the required deformation. \end{proof}

Here follows a particular application of Lemma \ref{isotopecrit}. It appears explicitly because of its fundamental nature, and to illustrate the seemingly nontrivial modifications to sequences of base diagrams such an isotopy can produce.

\begin{figure}\hypertarget{33b}\capstart%
  	\centering
  	\subfloat[A basic simplifying move in $\iota$. The arcs' counterparts have the same appearance.]{%
  	\begin{minipage}[c][3.4cm]{0.4\textwidth}\centering%
  	\label{simplecancelfig}\includegraphics{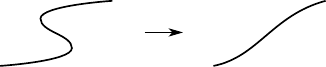}
  	\end{minipage}}\qquad
  	\subfloat[The critical image in $S^2_I$ before simplifying.]{%
  	\begin{minipage}[c][3.4cm]{0.4\textwidth}\centering%
  	\label{tinypairins2}\includegraphics{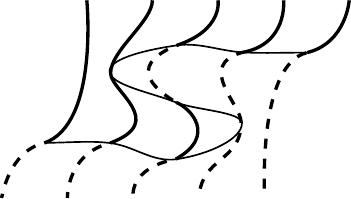}
  	\end{minipage}}
  	\caption{Depictions of a certain behavior of $\iota$ for Corollary \ref{simplecancelcor}. In both images, the page is part of $\crit\alpha$ or its image.}\label{simplecancelcorfigs}
  \end{figure}
\begin{cor}\label{simplecancelcor}Pairs of critical points of $T|_\iota$ as in the left side of Figure \ref{simplecancelfig} (in which no other immersion arcs or critical points are allowed to be present in the pictured region of $\crit\alpha$) can be eliminated by a homotopy of $\alpha$ to obtain the right side of the figure. The same modification is available to the vertical and horizontal reflections of the figure.\end{cor}
\begin{proof}Up to reflections of Figure \ref{simplecancelfig}, the deformation corresponding to Figure \ref{simplecancelfig} has base diagrams given by Figure \ref{tinypair}, with the cusp replaced by a fold point. The required homotopy of $\alpha$ is realized by an ambient isotopy in $S^2_I$ that gives a one-parameter family of deformations (by Lemma \ref{isotopecrit}) in which the $t$-value of the first $R_2$ deformation (that is, the $t$-value at which the images of the fold arcs become tangent) approaches that of the second $R_2$ deformation; see Figure \hyperlink{33b}{33b} for a depiction of the critical image of the deformation. In that figure, one of the fold arcs sweeps out the plane given by the page, while the other fold arc sweeps out a surface with stripes given by five $T$-level sets in $\crit\alpha$. The stripes become dotted lines where they dive below the page. The remaining S-shaped curve in Figure \hyperlink{33b}{33b} is where they intersect.\end{proof}

\begin{figure}\capstart
	\centering
	\subfloat[\ ]{\label{tinypairs1} \includegraphics{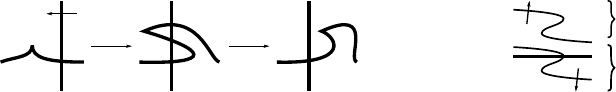}}\\
	\subfloat[\ ]{\label{tinypairs2} \includegraphics{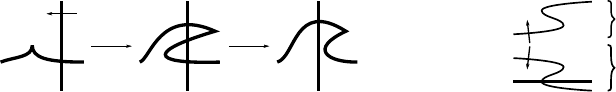}}
	\caption{Base diagrams for the two versions of Figure \ref{tinypair}, with corresponding decorated critical sets to the right. As is common with these pictures, each bracketed piece can be flipped vertically (also flipping its arrow), and each pair of bracketed pieces can switch places, and still represent the same deformation. The arrows give the correspondence between orientations in the pictures, and may be all marked ``1" or all marked ``2" as in \cite{GK1}.}\label{tinypairs}
\end{figure}
Here is a similar application that finds immediate use; its proof is essentially the same as that of Corollary \ref{simplecancelcor}. Figure \ref{tinypairs} depicts the two ways in which a cusp arc may cross an immersion arc. In that figure, either all arrows point into the lower-genus side of each fold or immersion arc, or they all point into the higher-genus side of each fold or immersion arc. Note that Figure \ref{tinypairs1} consists of mixed $R_2$ deformations and Figure \ref{tinypairs2} consists of either genus-increasing $R_2$ deformations or genus-decreasing $R_2$ deformations (according to the interpretation of the arrows). 
\begin{cor}\label{tinypairscor}By Proposition \ref{isotopecrit}, one may convert between Figures \ref{tinypairs1}  and \ref{tinypairs2} by a homotopy of $\alpha$.\end{cor}

Before adding to the list of moves on deformations, note there is generally not a Reidemeister-II move between a given cusp arc and immersion arc. For example, the single coming from a stabilization deformation has a particular form laid out in Section \ref{stabdefpf}, and canceling the fold crossings coming from two flips using some other $R_2$ deformation does not necessarily yield a stabilization move, as detailed in the proof of Lemma \ref{pseudostabilizations}: It could yield a stabilization occurring within a sequence of handleslides. Nevertheless, immersion arcs enjoy some freedom of movement.

\begin{lemma}\label{reidcusp}Immersion arcs can move around cusp arcs in at least three ways:\begin{enumerate}

\item[(\hyperlink{reidcusp(1b)pf}{1a})]\hypertarget{reidcusp(1a)}One may perform a finger move between an immersion arc and a parallel cusp arc on its lower-genus side, introducing two intersection points between them.
\item[(\hyperlink{reidcusp(1b)pf}{1b})]\hypertarget{reidcusp(1b)}Suppose there is a lune $L$ consisting of fold points in $\crit\alpha$ such that:\begin{itemize}
\item One side of $L$ is a cusp arc $C$.
\item The other side of $L$ is an immersion arc $A$.
\item Recall that $\chi$ denotes the cusp locus of $\alpha$. Then $L\cap(\iota\cup\chi)=A\cup C$.\end{itemize}
Then, possibly after an application of Corollary \ref{tinypairscor}, there is a homotopy of $\alpha$ whose endpoints appear in Figure \ref{reidcusp(1)fig}.
\item[(\hyperlink{reidcusp(2)pf}{2})]\hypertarget{reidcusp(2)} Suppose there is a triangle $\Delta$ in $\crit\alpha$ such that two sides are immersion arcs $A$ and $B$, the third side is a cusp arc $C$, and $\Delta\cap(\iota\cup\chi)=A\cup B\cup C$. Then there is a homotopy of $\alpha$ realizing a Reidemeister-III type move as indicated by $\Delta$.
\end{enumerate}\end{lemma}

\begin{figure}\capstart
	\labellist
	\small\hair 2pt
	\pinlabel $A$ at -3 26
	\pinlabel $C$ at -3 8
	\pinlabel $L$ at 28 17
	\endlabellist
	\centering
	\subfloat[Moving an immersion arc $A$ past a cusp arc $C$ in Lemma \ref{reidcusp}(\protect\hyperlink{reidcusp(1b)}{1b}). The lower side of each immersion arc is its lower-genus side. A similar pinching movement occurs with the counterpart of $A$, producing a new immersion pair.]{%
		\begin{minipage}[c][2.1cm]{0.75\textwidth}\centering%
		\label{reidcusp(1)fig} \includegraphics{reidcusp(1)fig}
		\end{minipage}}\\
	\labellist
	\small\hair 2pt
	\pinlabel $A$ at 7 23
	\pinlabel $B$ at 36 23
	\pinlabel $C$ at 21.5 4
	\endlabellist
	\subfloat[An example deformation corresponding to $\Delta$ in Lemma \ref{reidcusp}(\protect\hyperlink{reidcusp(2)}{2}).]{\label{reidcusp(2)fig} \includegraphics{reidcusp(2)fig}}
	\caption{Figures for Lemma \ref{reidcusp}.}\label{reidcuspfig}
\end{figure}
\begin{proof}\hypertarget{reidcusp(1a)pf} For item (\hyperlink{reidcusp(1a)}{1a}), the modification is simply a detour in which the cusp moves into the higher-genus side of a nearby crossing as in Figure \ref{tinypairs}, then returns.

\hypertarget{reidcusp(1b)pf} For item (\hyperlink{reidcusp(1b)}{1b}), the base diagrams of Figure \ref{tinypairs1}, followed by the reverse, is the sequence of base diagrams corresponding to the left side of Figure \ref{reidcusp(1)fig}. The proposed homotopy would cancel the pair of cusp-fold crossings that occur, leaving behind an immersion pair. It is not hard to verify the claim: The pair of cusp-fold crossings is a homotopy supported in a 4-ball neighborhood of the cusp, and the proposed homotopy of $\alpha$ consists of retracting that neighborhood until its image in $S^2$ is disjoint from the fold arc it crosses.

\hypertarget{reidcusp(2)pf} For item (\hyperlink{reidcusp(2)}{2}), the deformation corresponding to the triangle $\Delta\subset\crit\alpha$ begins with a triangle in the critical image, consisting of arcs $A$, $B$ and $C$, with a cusp lying on side $C$ (see Figure \ref{reidcusp(2)fig} for one such configuration). The cusp crosses side $A$ out of the triangle, then there is a Reidemeister-III fold crossing, and then the cusp crosses fold arc $B$ back into the triangle. The proposed modification would result in the cusp crossing side $B$ first, then side $A$. The midpoint of the proposed homotopy would have the cusp crossing $A$ and $B$ simultaneously at their own point of intersection. The disjointness of the vanishing set of $C$ from those of $A$ and $B$ that is required for the existence of these two deformations (and similarly for the intervening deformations) follows from the existence of all the intersections in the initial deformation (in verifying this, it helps to reverse $t$ if necessary so that the cusp in the initial base diagram points into the triangle like in Figure \ref{reidcusp(2)fig}).\end{proof}

\begin{lemma}\label{unlinklemma}For a deformation $\alpha$ resulting from Lemma \ref{genus} (that is, all critical points of $T$ are contained in merge pairs), there is a modification which causes $\alpha_t$ to be injective on its critical locus at those values of $t$ for which $\crit\alpha_t$ has two components.\end{lemma}
Another way to state the lemma is to say one can move $\iota$ into the critical cylinders between one merge pair and the next in the decorated critical surface of $\alpha$. Suppose at the $t$-values $t_0<t_1$ lie the first merge pair that has the pair of cylinders $\crit\alpha_{(t_0,t_1)}$ immersed. The main issue is that $\iota$ may have \emph{linked immersion circles}: a circle is linked if it is not freely homotopic to one that lies outside of $M_{[t_0,t_1]}$ by a homotopy that does not decrease the $t$-value of any point in the circle. In other words, treating an immersion circle as simply a circle in a suface, one cannot homotope it forward in $t$ past the merge pair. The following sublemma arranges $\iota$ so that its circles are embedded approaching $t_0$ as in Figure \ref{nest}, resulting in a collection of embedded, disjoint arcs which appear to be linked: call these \emph{linked arcs}. Lemma \ref{unlinklemma} is essentially proved by converting the collection of linked arcs to be a collection of genus-decreasing linked arcs, then momentarily canceling them over an interval containing $[t_0,t_1]$, thereby including the merge pair into a higher-genus deformation with embedded critical image.

\begin{figure}\capstart
	\centering{\includegraphics{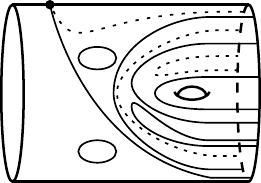}}
	\caption{An application of Sublemma \ref{nestsublemma} results in embedded $\iota_{[t',t_1]}\subset\crit\alpha$.}\label{nest}
\end{figure} 
The following sublemma assumes the existence of a certain $t$ value $t'$. If the fold merge at $t_0$ turns out to be the first fold merge in $\alpha$, set $t'=0$. Otherwise, $t'$ may as well lie just after the previous merge pair, between which $\crit\alpha$ is embedded by hypothesis.
\begin{sublemma}\label{nestsublemma}Suppose that at the $t$-values $t_0<t_1$ lie the first merge pair such that $\iota_{[t_1,t_1]}\neq\emptyset$. Let $0\leq t'<t_0$ be a $t$-value such that $M_{[t',t_0)}$ contains no merge points and $\iota_{t'}=\emptyset$. Then there is a homotopy of $\alpha$ that causes $\iota_{[t',t_1]}$ to be embedded in $\crit\alpha$.\end{sublemma}
Another way to state the conclusion of Lemma \ref{nestsublemma} is to say it is possible to eliminate all $R_3$ deformations from $\alpha_{[t',t_1]}$ by homotopy.
\begin{proof}This is merely an application of Lemma \ref{isotopecrit} in which all $R_3$ deformations in $[t',t_1]$ are pushed forward of $t_1$, using the reverse of the homotopy that appeared in  Corollary \ref{simplecancelcor} where necessary to arrange for intersecting immersion arcs to travel along the same cylinder in $\crit\alpha_{[t_0,t_1]}$. The fact that this can be achieved by ambient isotopy as required by that lemma follows from the restrictions on $\iota$ coming from the definitions of shift and multislide pairs. Consider the description in the last sentence of the first paragraph of Section \ref{cbd} (``One final way..."): The homotopy is one in which $\iota$ slides along the cylinder, perhaps with crossings getting hung up on points $p_i,q_i$, at which instances of Corollary \ref{simplecancelcor} occur.\end{proof}

\begin{proof}[Proof of Lemma \ref{unlinklemma}]Keep in mind that the terms \emph{pair} and \emph{single} in this proof are relative to $\iota_{\{t',t_1\}}$. What appears to be an immersion pair in $\iota_{[t',t_1]}$ may turn out to be connected in $\iota_{[0,1]}$, for example. The proof proceeds in the following steps, each step starting with the immersion locus near a merge pair as given by Sublemma \ref{nestsublemma} and the previous steps, and finishing with a progressively refined immersion locus near the merge pair that still satisfies the conclusion of Sublemma \ref{nestsublemma}:
\begin{enumerate}\label{unlinkproofsteps}
\item[(\hyperlink{unlinklemmastep1}{1})]\hypertarget{unlinklemma1}Convert all mixed immersion pairs into genus-increasing pairs, adding a genus-increasing single each time.
\item[(\hyperlink{unlinklemmastep2}{2})]\hypertarget{unlinklemma2}If the first immersion arc is genus-increasing (even if it is not linked), convert it to be part of a genus-decreasing pair.
\item[(\hyperlink{unlinklemmastep3}{3})]\hypertarget{unlinklemma3} Convert all subsequent immersion singles into genus-decreasing pairs.
\item[(\hyperlink{unlinklemmastep4}{4})]\hypertarget{unlinklemma4}Convert each genus-increasing immersion pair into a pair of singles and apply step 3 to each single produced.
%convert to two singles with disk argument and apply previous step
\item[(\hyperlink{unlinklemmastep5}{5})]\hypertarget{unlinklemma5}Now $\iota_{[t',t_1]}$ is a collection of embedded genus-decreasing pairs. Replace $\alpha$ with another deformation with the same endpoints in which these pairs cancel momentarily while the merge pair occurs.
\end{enumerate}
The last step may sound strange: It amounts to inserting a detour supported away from the support of the merge pair.

\hypertarget{unlinklemmastep1} Step \hyperlink{unlinklemma1}{1}: Convert mixed pairs. To eliminate mixed immersion pairs, consider Figure \ref{elimmixed}, which gives an alternative deformation to $\alpha$ when a mixed $R_2$ deformation precedes a merge pair. Instead of performing the mixed $R_2$ deformation, perform a flip in the location indicated by the figure, then a genus-increasing $R_2$ deformation and a genus-decreasing $R_2$ deformation as shown in the next two base diagrams, producing a triangle in the critical image that can be contracted by inverse flip.
\begin{figure}\capstart
	\centering{\includegraphics{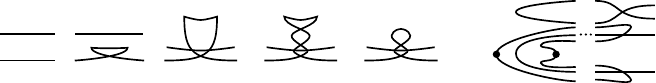}}
	\caption{Converting a mixed $R_2$ deformation for Step \protect\hyperlink{unlinklemma1}{1} of the proof of Lemma \ref{unlinklemma}. At far left, the genus decreases moving downward across each fold. The two dotted arcs are part of a genus-increasing single as suggested by the figure.}\label{elimmixed}
\end{figure}The difference between the result of this deformation and that of the mixed $R_2$ deformation is there is now a central lune, bisected by the upper fold arc. Because there were originally no $R_3$ deformations in $[t',t_1]$, there are no movements of fold arcs across this lune to introduce complications, so this deformation can be substituted for the mixed $R_2$, with the bisected lune contracting immediately after the merge pair to resume $\alpha$ as it was before the modification. The decorated critical surface appears at the right, and consists of a genus-increasing single and a genus-increasing pair. The merge pair and other $R_2$ deformations occur in the dotted omitted part, supported over regions of $S^2$ distinct from that of the base diagrams at the left. This concludes Step \hyperlink{unlinklemma1}{1}.\\

\hypertarget{unlinklemmastep2}Step \hyperlink{unlinklemma2}{2}: Convert any initial genus-increasing pair or single. This is achieved by inserting an appropriate detour. In Figure \ref{gipbd}, cusps have been sprinkled about in order to give the reader some perspective on how such a deformation would appear; indeed, these modifications are valid regardless of the placement of cusps, and what matters is the placement of the swallowtails in relation to the immersion locus. Figures \ref{gipbd1} and \hyperlink{38b}{38b} depict the initial genus-increasing $R_2$ deformation. Immediately, as depicted in Figure \hyperlink{38b}{38c}, two flips occur, followed by a pair of $R_2$ deformations to obtain Figure \hyperlink{38b}{38d}, valid by \cite[Proposition 2.5(3)]{W2}. The deformation then reverses itself back to Figure \ref{gipbd1} and $\alpha$ resumes as it did before. Deducing the correspondence between Figures \ref{gipcbd1}-\hyperlink{39b}{40b} and \ref{gipbd} is a moderately straightforward exercise.
\begin{figure}\hypertarget{38b}\capstart%
	\centering
	\subfloat[\ ]{\label{gipbd1}\includegraphics{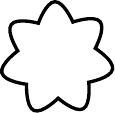}} \qquad
	\subfloat[\ ]{\label{gipbd2}\includegraphics{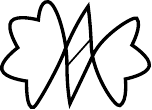}} \qquad
	\subfloat[\ ]{\label{gipbd3}\includegraphics{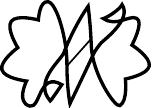}} \qquad
	\subfloat[\ ]{\label{gipbd4}\includegraphics{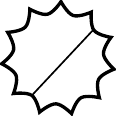}}
	\caption{Base diagrams for the detour of Figure \protect\hyperlink{39b}{40b}. The deformation proceeds a-b-c-d-c-b.}\label{gipbd}
\end{figure}
\begin{figure}\hypertarget{39b}\capstart	
	\labellist
	\small\hair 2pt
	\pinlabel $\longrightarrow$ at 125 43
	\endlabellist
	\subfloat[\ ]{\label{giscbd1}\includegraphics{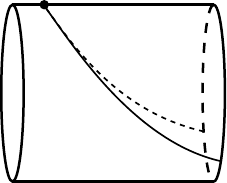}}\hspace{1cm}
	\subfloat[\ ]{\label{giscbd2}\includegraphics{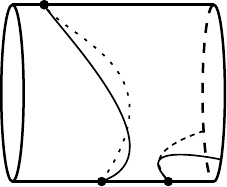}}
	\caption{Converting a single.\label{giscbd}}
	\labellist
	\small\hair 2pt
	\pinlabel $\longrightarrow$ at 125 43
	\endlabellist
	\centering
	\subfloat[\ ]{\label{gipcbd1}\includegraphics{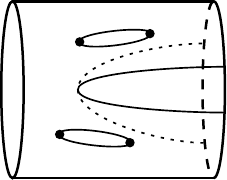}}\hspace{1cm}
	\subfloat[\ ]{\label{gipcbd2}\includegraphics{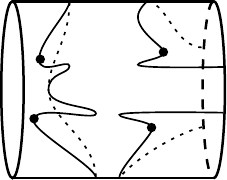}} 
	\caption{converting a pair.}
	\caption*{Decorated critical surfaces that depict the conversion of genus-increasing immersion arcs in Step \protect\hyperlink{unlinklemma2}{2} of the proof of Lemma \ref{unlinklemma}. In each case, $t'$ moves to the ``halfway point" where $\crit\alpha$ is embedded.}\label{gicbd}
\end{figure}
For a genus-increasing single, after a homotopy of $\alpha$ the single appears by a flip as in Figure \ref{giscbd1}, and the similar detour of performing a second flip and canceling the two intersections using \cite[Proposition 2.5(3)]{W2} results in the decorated critical surface in Figure \hyperlink{39b}{39b}. In either case, $t'$ is moved forward past the newly introduced generalized stabilization deformation to lie just before the genus-decreasing $R_2$ deformation. Finally, push the inverse flips forward past $t_1$. This concludes Step \hyperlink{unlinklemma2}{2}.\\

\hypertarget{unlinklemmastep3}Step \hyperlink{unlinklemma3}{3}: Convert the rest of the singles into genus-decreasing pairs. Any genus-decreasing single contains a swallowtail that corresponds to an inverse flipping move, because this is the only way for an immersion arc and its counterpart to lie in the same path component of $\iota$. Push this swallowtail forward in $t$ (along a path that is disjoint from $\iota$) past $t_1$ to cause these immersion arcs to lie within genus-decreasing pairs in $[t',t_1]$. Using Proposition \ref{isotopecrit} and Corollary \ref{simplecancelcor}, a genus-increasing single can be assumed to appear by a flipping move whose trailing immersion arcs are free of critical points of $T|\iota_{[t',t_1]}$. Apply the modification of Figure \ref{swalpush} as necessary to cause the single to be the first immersion arc to appear in $[t',t_1]$, then apply Sublemma \ref{nestsublemma} to push the resulting $R_3$ deformations past $t_1$ and apply Step \hyperlink{unlinklemma2}{2} to the single. This concludes Step \hyperlink{unlinklemma3}{3}.\\

\hypertarget{unlinklemmastep4}Step \hyperlink{unlinklemma4}{4}: Convert the genus-increasing pairs into genus-decreasing pairs. We begin with a deformation consisting of a sequence of genus-decreasing $R_2$ deformations, followed by a genus-increasing $R_2$ deformation. Label the two points in the critical image that approach each other in the genus-increasing $R_2$ deformation $p$ and $q$. Because there are no $R_3$ or mixed $R_2$ deformations for $t\in[t',t'')$, if two of these lunes intersect, then one is contained in the other (points in a newly-formed lune $L$ contained in a lune $L'$ could only move to lie outside of $L'$ by an $R_3$ deformation, a mixed $R_2$ deformation or a genus-increasing $R_2$ deformation). This nesting condition is preserved if one contracts an innermost lune using \cite[Proposition 2.5(3)]{W2}, so all lunes can be contracted to yield an SPWF $\alpha_{t''-\epsilon}$ for small $\epsilon>0$, and the following cases use such cancellations in detours.

If $p$ lies in one of the bigons formed by the genus-decreasing deformations, then the orientations of the fold arcs requires $q$ to lie in the same bigon. 

If $p$ and $q$ lie in the same side, then the genus-increasing $R_2$ deformation forms a loop in the critical image as in the left side or the right side of Figure \hyperlink{38b}{38b}: it is embedded and has exactly one crossing $x$ which is a corner of the bigon just formed. Place a flip as in Figure \hyperlink{38b}{38c} and cancel that crossing as in Figure \hyperlink{38b}{38d}, then reverse those moves, thereby inserting a detour (if there had been crossings in that side between $p$ and $q$, then they will be pairs of corners of bigons and it will be necessary to cancel them before canceling $x$, then reverse those cancellations after the detour that cancels $x$). This causes the immersion arcs corresponding to the genus-increasing $R_2$ deformation to lie within two genus-increasing immersion singles, to which Step \hyperlink{unlinklemma3}{3} applies. 

If $p$ and $q$ lie within distinct sides of their bigon, Then one of the corners of that bigon can be canceled against one of the corners of the bigon created by the genus-increasing $R_2$ deformation using \cite[Proposition 2.5(3)]{W2}. That cancellation followed by its reverse is a detour whose insertion results in the connect sum between the immersion arcs corresponding to those corners (this is an instance of Lemma \ref{r3moves}(\hyperlink{r3moves(1)}{1}) below). At this point, the immersion arcs corresponding to the genus-increasing $R_2$ deformation now lie within a genus-decreasing pair. 

If $p$ does not lie in one of the bigons formed by the genus-decreasing $R_2$ deformations, then those bigons can all be contracted after the genus-increasing $R_2$ deformation as in the first paragraph of this step. Contract enough of these to use the trick from the third paragraph (starting with the left side or the right side of Figure \hyperlink{38b}{38b}), then reverse this cancellation, thereby inserting a detour. At the halfway point of this detour, use the trick from the first case in this paragraph.\\

\hypertarget{unlinklemmastep5}Step \hyperlink{unlinklemma5}{5}: Clear the immersion arcs from $\alpha_{[t_0,t_1]}$. The immersion locus $\iota_{[t',t_1]}$ is still embedded and now consists of genus-decreasing pairs that either bound pairs of disks contained in $[t',t_0-\epsilon]$ for some small $\epsilon>0$, or that are genus-decreasing pairs that intersect $\crit\alpha_{t_1}$. The former can be disregarded because their support is already disjoint from $M_{[t_0,t_1]}$ and is disjoint from the support of the following modifications to $\alpha$. For the latter, recall from Step \hyperlink{unlinklemmastep4}{4} the formation of a collection of lunes as the finger moves occur. 
The second half of Step \hyperlink{unlinklemma5}{5} is to temporarily contract these lunes, thereby including the impending local shift or multislide deformation into a higher-genus fibration with embedded critical image. Beginning at a value $t^{(4)}<t_0$ that comes after the last lune is formed, consider a detour $\delta$ given by contracting all of the lunes by repeated application of \cite[Proposition 2.5(3)]{W2}, then the reverse. 
As discussed in \cite[Remark 2.2]{W2}, the support of any 2-parameter crossing deformation that contracts one of these lunes is a ball in $M$, crossed with an interval in $t$. This ball in $M$ is a neighborhood of the critical arc $\varphi_0$ that maps to a neighborhood of the lune in \cite[Figure 2]{W1} even when $\varphi_0$ has cusps. If a critical value is contained in the boundary of more than one of these lunes, then these lunes are nested and the support of the cancellation homotopy for the outermost lune contains the support of that of the lunes it contains. Thus the support of the detour is a finite disjoint union of balls $S_\delta=\cup_iB_i\subset M$, crossed with an interval in $t$. 
As for the merge pair, let $S_{mp}=(\nu p_{fm}\cup\nu p_{cm})\times(a,b)$ denote the support of the associated generalized shift or multislide deformation (\emph{mp} stands for \emph{merge pair}), where $\nu p_{fm}$ is a tubular neighborhood of the fold merge path, $\nu p_{cm}$ is a neighborhood of the cusp merge path. 

Since any side $\ell$ of any bigon contracted by $\delta$ lies on the lower-genus side of the critical arc crossing $\ell$ at its endpoints, if $\ell$ contains an endpoint $z\in p_{fm}$, it is possible to push one of the crossings along $\ell$ past $z$, with the support of the homotopy automatically disjoint from $p_{fm}$. For this reason, it is possible to choose for each cancellation in the detour $\delta$ to push around a critical arc that does not contain any endpoint of $p_{fm}$. With this choice, it is possible to arrange for $S_\delta\cap(\nu p_{fm}\cup\nu p_{cm})=\emptyset$. In other words, it is possible to contract the lunes in such a way that the critical arcs whose images pass over the region where the merge pair occurs are bounded away from the curves $\nu p_{fm}$ and $\nu p_{cm}$. Restating it fiberwise, it is possible for the pairs of points corresponding to arcs moved by $\delta$ to avoid the two points corresponding to $p_{fm}$ and $p_{cm}$. 
With this understood, choose $c$ slightly larger than $b$ and replace $\alpha|_{\left(S_\delta\right)\times(t^{(4)},c)}$ with the detour $\delta$, where the first half of the detour occurs within the interval $\left(t^{(4)},a\right)$, then $\delta$ continues as a trivial deformation in which nothing occurs until $t=b$, at which point the reverse commences, re-forming the lunes. This concludes Step \hyperlink{unlinklemma5}{5}.

After applying steps \hyperlink{unlinklemma1}{1}-\hyperlink{unlinklemma5}{5} to all subsequent merge pairs, the resulting deformation satisfies the conclusion of Lemma \ref{unlinklemma}.\end{proof}

\subsubsection{Organizing the unlinked immersion locus}
This section lists the last few modifications to $\crit\alpha$. According to Lemma \ref{unlinklemma}, each twice-punctured critical torus associated to a merge pair is embedded by $\alpha$ into $S^2_I$, while the typical intervening critical cylinder between shift pairs is immersed. For this reason, $\alpha_t$ is as required in the conclusion of Theorem \ref{T} near those values of $t$ for which $\crit\alpha_t$ is disconnected, and it suffices to restrict attention to the intervening critical cylinders, where the goal is to first arrange for $\iota$ to be an embedded 1-submanifold of $\crit\alpha$. The following Sublemma is used in the proof of Lemma \ref{isolate}.

\begin{sublemma}\label{generator}Suppose $\crit\alpha_{[t_0,t_1]}\approx S^1_{[t_0,t_1]}$ and  $c\subset\iota_{[t_0,t_1]}$ is a circle that is not nullhomotopic in $\crit\alpha$. Then $[c]$ generates $H_1\left(S^1_{[t_0,t_1]}\right)\cong\Z$.\end{sublemma}
\begin{proof}Using Proposition \ref{isotopecrit}, arrange for the local minimal values of $T|_c$ to be contained in $(t_0-\epsilon,t_0)$ and for the local maximal values of $T|_c$ to be contained in $(t_1,t_1+\epsilon)$ for some small $\epsilon>0$. Then $c_{[t_0,t_1]}$ is an immersed braid in $\crit\alpha_{[t_0,t_1]}$ with its strands naturally partitioned into pairs like in Figure \ref{immerlocmods}. Note that two strands in the same partition are not allowed to cross, because the only crossings allowed for $\iota$ come from $R_3$ deformations (the required pattern of intersections appears in Figure \hyperlink{16b}{16c}). By Lemma \ref{circlepairs}, there is at least one swallowtail in $c$, say in $\crit\alpha_{(t_0-\epsilon,t_0)}$ after possibly reversing $t$. Follow the pair of paths traced by a pair of points $A,B$ forward along the two arcs of $c$ that emanate from the swallowtail until they enter $\crit\alpha_{(t_1,t_1+\epsilon)}$. If $A$ and $B$ do not approach distinct local maxima of $T|_c$ (where according to Figure \hyperlink{16b}{16b} an $R_2$ deformation occurs), then they reunite at another swallowtail, so that $c$ would represent an element of $\{-1,0,1\}\subset H_1(\crit\alpha_{[t_0,t_1]})$, satisfying the proposition. For this reason, suppose $A$ and $B$ approach distinct local maxima of $T|_c$. As $A$ and $B$ cross the maxima, they enter a pair of arcs which form another member of the partition according to Figure \hyperlink{16b}{16b}. Follow these until $A$ and $B$ return to $\crit\alpha_{(t_0-\epsilon,t_0)}$, and so on, until all of $c$ is traced out and the two points reunite, necessarily at a second swallowtail (note this, along with Lemma \ref{circlepairs}, gives another proof that every immersion single has exactly two swallowtails). Recording the map \[c\co S^1\hookrightarrow\crit\alpha_{[t_0-\epsilon,t_1+\epsilon]}\approx[t_0-\epsilon,t_1+\epsilon]\times S^1\to S^1\] in polar coordinates $(1,\theta(s))|_{s\in[0,1]}$, assume (using Proposition \ref{isotopecrit} if necessary) that the endpoints of the strands traced by $A$ map to even multiples of $\pi$ and those traced by $B$ map to odd multiples of $\pi$ (except for the two swallowtail points, which map to, say, even multiples of $\pi$). Parameterize the members of partition $i$ as $A_i(u),B_i(u)$, $u\in[0,1]$, where $T(A_i(u))=T(B_i(u))$ and the orientation of $B_i$ coming from $u$ is the opposite of its orientation coming from the polar coordinates, while the two orientations of $A_i(u)$ agree. The observation is that if strands in the same partition never cross, then the two contributions to the winding number \[\frac{\theta(1)-\theta(0)}{2}=\sum\limits_i\frac{\theta(A_i(1))-\theta(A_i(0))}{2}-\frac{\theta(B_i(1))-\theta(B_i(0))}{2}\] of each partition not containing a swallowtail cancel exactly, so that the winding number, in other words the element $[c]\in H_1\left(S^1_{[t_0,t_1]}\right)$, lies in $\{-1,0,1\}$. Finally, $c$ was assumed not to be nullhomotopic, so $[c]=\pm1$.\end{proof}

The model deformations from Section \ref{moves} do not involve $R_3$ deformations, and the following lemma is the main tool for eliminating them from $\alpha$. In light of Lemma \ref{reidcusp}, which should be understood to be used as needed, the cusp locus does not show up in the arguments until Lemma \ref{evertlemma}.

\begin{lemma}\label{r3moves}
Let $\alpha$ be a deformation such that there is a lune $L$ or triangle $\Delta$ in $\crit\alpha$ bounded by arcs in the immersion locus $\iota$.\begin{enumerate}
\item[(\hyperlink{r3moves(1)pf}{1})]\hypertarget{r3moves(1)}(Pinching lemma) As in Figure \ref{pinchfig}, suppose $a,b$ are points in immersion arcs such that:\begin{itemize}\item $T(a)=T(b)$\item $\alpha(a)\neq\alpha(b)$\item There is a smooth arc $f\co[0,1]\to\crit\alpha_{T(a)}$ which connects $a$ to $b$ and is otherwise disjoint from $\iota$,\item The image of $f$ lies on the lower-genus side of both immersion arcs.\item There is another arc $f'$ connecting the counterparts $a',b'$ of $a$ and $b$ with the analogous properties.\end{itemize}Then there is a homotopy of $\alpha$ realizing the modification that appears in that figure.
\item[(\hyperlink{r3moves(23)pf}{2})]\hypertarget{r3moves(2)}(Triangle move lemma) Assume $\partial\Delta$ is free of swallowtail points and that $c$ is a boundary arc of $\Delta$ with $\Delta$ on its lower-genus side. Then $\alpha$ is homotopic to another deformation with the same endpoints in which $c$ has moved across $\Delta$ in a Reidemeister-III type move (which this paper calls a \emph{triangle move}). This new deformation has exactly two more $R_3$ deformations; see Figure \ref{r3moves_2_fig}.
\item[(\hyperlink{r3moves(23)pf}{3})]\hypertarget{r3moves(3)} (Bigon move lemma) Assume $\partial L$ is free of swallowtail points and that $c$ is a boundary arc of $L$ with $L$ on its lower-genus side. Then the $R_3$ deformations that give the two corners of $L$ can be eliminated by a Reidemeister-II type move in $\iota$ (which this paper calls a \emph{bigon move}) without introducing other Reidemeister-III fold crossings. See Figure \ref{r3moves_3_fig}.
\item[(\hyperlink{r3moves(4)pf}{4})]\hypertarget{r3moves(4)} (Finger move lemma) Suppose an immersion arc $c$ lies parallel and adjacent to another immersion arc $d$ on its lower-genus side, and that $d$ is not the counterpart of $c$. Then it is possible to perform a finger move in $\iota$, pushing a small part of $c$ across $d$ at the expense of introducing a mixed or genus-increasing immersion pair and one pair of $R_3$ deformations. See Figure \hyperlink{46b}{46b}. \end{enumerate}\end{lemma}

It is straightforward to verify that all movements of critical points in Lemma \ref{r3moves} occur toward their lower-genus side, thereby increasing the genus of affected fibers. For this reason, any application of Lemma \ref{r3moves} preserves the conditions of Remark \ref{genusrmk}), that all fibers are connected and have genus at least 2. For brevity, applications of Lemma \ref{r3moves} will sometimes be called \emph{pinching moves, triangle moves, bigon moves and finger moves.}

\begin{figure}\capstart
	\centering	
	\labellist
	\small\hair 2pt
	\pinlabel $b$ at 11 26
	\pinlabel $f$ at 19 16
	\pinlabel $a$ at 11 5
	\endlabellist
	\includegraphics{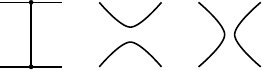}
	\caption{Modification of $\iota$ in Lemma \ref{r3moves}(\protect\hyperlink{r3moves(1)}{1}). Two horizontal immersion arcs undergo connect sum according to the reference path $f$. The same figure with labels $a',b',f'$ depicts what happens to the counterparts of these immersion arcs.\label{pinchfig}}
\end{figure}

\begin{figure}\capstart
	\labellist
	\small\hair 2pt
	\pinlabel $F$ at 83 132
	\pinlabel $\Delta$ at 71 132
	\pinlabel $c$ at 76 80
	\pinlabel $c$ at 235 80
	\pinlabel $\Delta$ at 27 25
	\pinlabel $F$ at 64 25
	\endlabellist
	\centering{\includegraphics{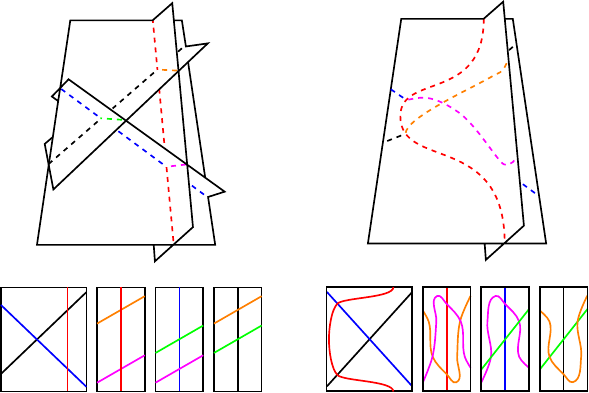}}
	\caption{Figure for Lemma \ref{r3moves}(\protect\hyperlink{r3moves(2)}{2}). The strip $F$ swept out by the fold passing through $c$ transverse to $\Delta$ bulges to the left in the right side of the figure. The two diagonal strips are omitted from the right side to avoid clutter, and decorated critical surfaces appear at bottom. Unlike other depictions of $\iota$, $t$ ranges vertically here for aesthetic reasons.}\label{r3moves_2_fig}
\end{figure}

\begin{figure}\capstart
	\labellist
	\small\hair 2pt
	\pinlabel $c$ at 2 199
	\pinlabel $L$ at 79 144
	\pinlabel $F$ at 98 144
	\pinlabel $c$ at 56 96
	\pinlabel $c$ at 166 96
	\pinlabel $c$ at 242 200
	\pinlabel $L$ at 33 34
	\pinlabel $F$ at 78 34
	\endlabellist
	\centering{\includegraphics{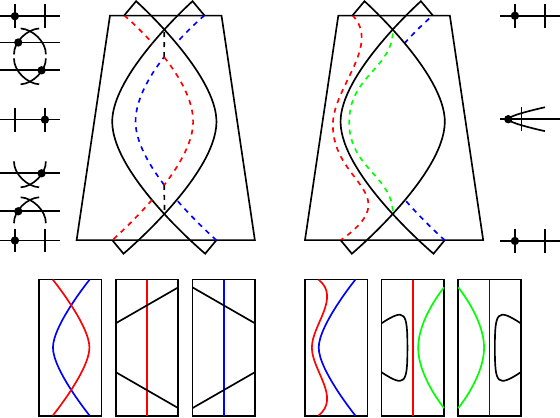}}
	\caption{Figure for Lemma \ref{r3moves}(\protect\hyperlink{r3moves(3)}{3}). The strip $F$ bulges to the left in the right side of the figure. Cross-sections (which are base diagrams when $t$ ranges vertically) are included on either side to help understand the picture, with $c$ marked by a dot. As depicted in the decorated critical surfaces below, the short immersion arcs passing transverse to $L$ at its corners undergo connect sum. Unlike other depictions of $\iota$, $t$ ranges vertically here for aesthetic reasons.}\label{r3moves_3_fig}
\end{figure}

\begin{proof}\hypertarget{r3moves(1)pf}Item (\hyperlink{r3moves(1)}{1}) is a simple application of \cite[Proposition 2.5(3)]{W2}. The curves $\alpha\circ f$ and $\alpha\circ f'$ comprise the relevant bigon in $S^2_{T(a)}$, and the modification is to insert a detour given by that fold crossing, followed by its reverse.

\hypertarget{r3moves(23)pf} Assertions (\hyperlink{r3moves(2)}{2}) and (\hyperlink{r3moves(3)}{3}) follow much like \cite[Propositions 2.5(1a) and 2.7]{W2}. For now, suppose $t$ parameterizes $c$ in the left side of Figures \ref{r3moves_2_fig} and \ref{r3moves_3_fig}, so that an open 2-disk $D$ of critical values containing $\Delta$ or $L$ is foliated by the horizontal arcs $D_t$. Denote by $F\subset S^2_I$ the thin strip swept out by the fold passing through $c$ transverse to $\Delta$ or $L$, as indicated in the figures. The proposed movement can be interpreted as a one-parameter family of finger moves on the level sets $F_t$, pushing the point $c_t$ along a smooth arc parallel to $D_t$ ($D$ appears as a possibly cusped arc passing vertically through the lunes in \cite[Figures 3 and 5]{W2}, and near $D$ the modifications just appear as moving the cusps in $D_t$ into the higher genus side of $F_t$). As discussed in the proofs of the propositions mentioned above, at each $t$ the finger move exists exactly when there is an arc $a_t$ in the base connecting the two pictured fold arcs (in \cite[Figures 3 and 5]{W2}) whose lift to $Z$ is disjoint from $Y$. For this reason, the moves suggested by Figures \ref{r3moves_2_fig} and \ref{r3moves_3_fig} exist when there is a one-parameter family of such arcs: a surface in $S\subset S^2_I$ parallel to $D$ whose lift to the one-parameter family of vanishing sets $Z=\bigcup_tZ_t$ is disjoint from the family $Y=\bigcup_tY_t$. For varying $t$, $Z\cap Y$ appears as a family of vertical arcs emanating from $F_t$ which are necessarily bounded away from intersections of $F_t$ with the other folds transverse to $D$, so $S$ can be constructed by taking advantage of the one-dimensional space of choices for each $a_t$, allowing those arcs to wander along $F$ for varying $t$ to stay disjoint from $\alpha(Z\cap Y)$.

The previous paragraph shows that the proposed moves exist as moves on deformations when $\Delta$ and $L$ appear as pictured, with $t$ varying vertically, or more generally as homotopy moves that can be applied to triangles and lunes in the space of stable maps from a 5-manifold to a 3-manifold. It remains to check that these moves can be performed for general triangles and lunes, staying within the class of deformations. Because they preserve $\crit\alpha$ as a stratified submanifold, it is enough to check they can be performed while keeping the image of the fold locus transverse to the level sets $S^2_t$. Because $\partial L$ and $\partial\Delta$ are free of swallowtails, their corners must come from two or respectively three $R_3$ deformations occurring at pairwise distinct $t$-values. For this reason, $L$ and $\Delta$ may deviate from the left sides of Figures \ref{r3moves_2_fig} and \ref{r3moves_3_fig} by the addition of critical points of $T|_{\partial L}$ or respectively $T|_{\partial\Delta}$. Such critical points may be pushed out of the boundary by appropriately applying Lemma \ref{isotopecrit}: One may retract $L$ or $\Delta$, pushing $c$ across any critical points in the other sides. The only part of this retraction which is not an ambient isotopy of the critical image is when $c$ crosses a cusp arc, but this corresponds to moving a cusp into the higher genus side of a fold, which is always allowed. This and Corollary \ref{simplecancelcor} eliminate all of the critical points of $T|_{\partial L}$ or $T|_{\partial\Delta}$ except possibly one critical point in $c$.

\hypertarget{r3moves(4)pf}Assertion (\hyperlink{r3moves(4)}{4}) follows quickly from considering Figure \ref{iotafingerbd}, which gives base diagrams for the first half of a detour whose initial finger move exists by \cite[Proposition 2.5(1a)]{W2}, and whose initial $R_3$ deformation exists by \cite[Proposition 2.9]{W2}.
\begin{figure}\hypertarget{44b}\capstart%
	\labellist
	\small\hair 2pt
	\pinlabel $c$ at 18 63
	\pinlabel $d$ at 18 11
	\pinlabel $c$ at 225 49
	\pinlabel $d$ at 225 24
	\endlabellist
	\centering
	\subfloat[Base diagrams for the first half of the detour.]{%
\begin{minipage}[c][6cm]{0.5\textwidth}\centering%
	\label{iotafingerbd}\includegraphics{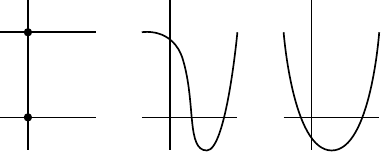}
\end{minipage}}\qquad 
	\subfloat[A depiction of $\iota$.]{%
\begin{minipage}[c][6cm]{0.42\textwidth}\centering%
	\label{iotafingercbd}\includegraphics{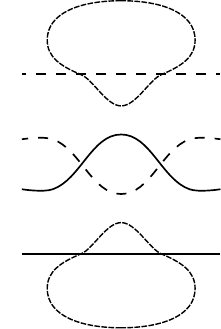} \qquad
\end{minipage}}
	\caption{Performing a finger move in $\iota$ for Lemma \ref{r3moves}(\protect\hyperlink{r3moves(4)}{4}).}\label{iotafinger}
\end{figure}
\end{proof}

It might be necessary to continue a finger move (Figure \ref{iotafinger}) forward or backward in $t$, which would require the finger's counterpart to also move forward or backward in $t$  (this counterpart would be one of the two arcs inside of the circles). When that happens, it is an instance of Lemma \ref{isotopecrit} in which any immersion arcs obstructing such movement end up moving along with the counterpart, preceding it in its path forward or backward in $t$. 

Having established the above tools, the idea is to arrange for the non-isolated part of $\iota$ to consist entirely of immersion pairs, then to eliminate crossings using a particular combination of the moves above. With that in mind, here is some terminology.

\begin{df}	A crossing in $\iota$ is \emph{primary} if it is a corner of a lune $L$ or triangle $\Delta$ as in Figures \ref{r3moves_2_fig} and \ref{r3moves_3_fig}. Given an arbitrary crossing in $\iota$, say the upper crossing in Figure \hyperlink{16b}{16c}, the other two crossings and the immersion arcs $c,c'$ are called \emph{secondary}.\end{df}

The secondary arcs in Figure \ref{r3moves_2_fig} are orange, fuchsia and green (for non-color renderings, they are the arcs in the narrow rectangles at the bottom of the figure that are not vertical). They are black in the bottom of Figure \ref{r3moves_3_fig} (the four arcs that undergo connect sum).

\begin{df}For short, call an immersion circle \emph{isolated} if it is embedded in $\crit\alpha$ and it is disjoint from all other immersion circles. A \emph{pseudostabilization} is a deformation $\alpha_{[t_1,t_2]}$ that resembles a flip-and-slip move. More precisely, it is any deformation satisfying the following criteria:\begin{itemize}\item The map $\alpha_{t_i}$ is an SPWF for $i=1,2$.\item There are no birth or merge points.\item The immersion locus of $\alpha$ consists of an isolated single that is not nullhomotopic in $\crit\alpha$.\end{itemize}A stabilization deformation is a pseudostabilization, the term \emph{pseudostabilization deformation} is redundant, and there is no \emph{pseudostabilization move} on surface diagrams defined in this paper (see Lemma \ref{pseudostabilizations}).\end{df}

\begin{lemma}\label{isolate}The deformation $\alpha$ resulting from Lemma \ref{unlinklemma} can be modified by homotopy so that each circle in $\iota$ is an isolated single that corresponds to a pseudostabilization or is one of an isolated pair.\end{lemma}
\begin{proof}
The immersion locus gives the $t$-interval $[0,1]$ a partition whose elements are of two types. For $a_i<b_i<a_{i+1}\ldots$, these are the subintervals $[a_i,b_i]$ over which $\iota_t$ is nonempty for all $t\in[a_i,b_i]$, and the complementary subintervals $(b_i,a_{i+1})$ over which $\iota_t=\emptyset$. By Lemma \ref{unlinklemma}, $\crit\alpha_{[a_i,b_i]}$ is a cylinder for all $i$. It is sufficient to prove the lemma by modifying $\alpha$ near $t$-values in $[a_1,b_1]$, for then the argument applies to the other intervals $[a_i,b_i]$, $i>1$. The proof proceeds by progressively modifying the immersion locus in that interval until it consists of isolated circles. The modifications tend to create a lot of isolated singles near the endpoints of the interval, and the endpoints of the interval are repeatedly redefined to omit these extraneous circles.

Lemmas \ref{isotopecrit}, \ref{reidcusp} and \ref{r3moves} together imply that an immersion arc can be isotoped toward its lower-genus side regardless of the presence of other cusp and immersion arcs, possibly at the expense of adding new immersion pairs (Lemmas \ref{reidcusp}(\hyperlink{reidcusp(1b)}{1b}) and \ref{r3moves}(\hyperlink{r3moves(4)}{4})) or new crossings between immersion circles (Lemma \ref{r3moves}(\hyperlink{r3moves(2)}{2}, \hyperlink{r3moves(4)}{4})). The idea of the proof is to first convert all of the immersion singles in that interval into pairs (using a technique that adds many non-isolated pairs), then, using the lemmas in the previous sentence, to isolate the pairs in a way that does not produce an unending cycle of crossing creation. It is a technique that tends to leave its worksite littered with many isolated pairs.

Using Lemma \ref{swalpushlem} and the trick from Figure \ref{swaljump}, push all reverse flipping moves forward in $t$ to lie in  $\crit\alpha_{b_1}$ and push all flipping moves backward in $t$ to lie in $\crit\alpha_{a_1}$. If necessary, slightly perturb $\alpha$ near the swallowtail points so they do not coincide with the Reidemeister-II fold crossings that occur at those $t$ values. Insert the detour from Figures \ref{giscbd1}-\hyperlink{39b}{39b} to one of the swallowtails at $t=a_1$. This converts the swallowtail deformation, which here corresponds to a flipping move, to a pseudostabilization followed by a genus-decreasing $R_2$ deformation and an inverse flipping move, whose swallowtail may be placed to have a $t$-value slightly greater than $a_1$. Use Lemma \ref{swalpushlem} to push another swallowtail at $a_1$ backward in $t$ to lie at a $t$-value just after the pseudostabilization, but before the $R_2$ deformation just produced. Then insert the detour from Figures \ref{giscbd1}-\hyperlink{39b}{39b} to it. One by one, convert the flipping moves at $t=a_1$ into pseudostabilizations and inverse flipping moves using this procedure, and do the same, with $t$ reversed, to the swallowtails at $t=b_1$. Since the sequence of pseudostabilizations and the sequence of inverse pseudostabilizations just constructed each has isolated immersion locus, restrict attention to the intervening immersion locus, calling the $t$-interval it inhabits $[a,b]$. By construction, $\iota_t$ is empty for $t$ slightly smaller than $a$ and slightly larger than $b$. Since the flipping move is only deformation that changes the parity of the number of crossings in the critical image, the number of flips in $[a,b]$ is congruent mod 2 to the number of inverse flips (note the flips are clustered near $t=b$ and the inverse flips near $t=a$). For some integer $n>0$, suppose there are $2n$ more swallowtails near $t=a$ than at $t=b$. Let $S_n$ denote a detour consisting of $n$ stabilizations followed by the reverse of those $n$ stabilizations, and insert $S_n$ starting at a value of $t$ slightly greater than $b$. Once this insertion has occurred, near the halfway point of the inserted copy of $S_n$ there occurs the final Reidemeister-II fold crossing of the sequence of stabilizations. Set $b$ to be the $t$-coordinate of this fold crossing and restrict attention to this new interval. If (instead of near $t=a$) there are $2n$ more swallowtails near $b$, similarly insert $S_n$ just before $t=a$ and redefine $a$ in the symmetric way, with $a$ now defined to be the $t$-coordinate of the first Reidemeister-II fold crossing in the sequence of inverse stabilizations of $S_n$. Now there is a collection of inverse flipping moves near $t=a$ and a collection of the same number of flipping moves near $t=b$. Choose a bijection $\beta$ between these two sets of swallowtails, and choose a swallowtail $\sigma$ near $t=a$. Using Lemma \ref{swalpushlem} and the trick from Figure \ref{swaljump}, push $\sigma$ forward in $t$ and possibly across cusp arcs to lie adjacent to $\beta(\sigma)$ as in Figure \ref{cancelswal}. Cancel the two swallowtails as in that figure. Repeat for all the other swallowtails near $t=a$. The immersion locus under consideration has now been modified so that all the crossings occur in a subinterval of $[a,b]$ during which $\iota$ consists entirely of immersion pairs, each of which are nullhomotopic in $\crit\alpha$ by Lemma \ref{circlepairs}. With this understood, assume $\iota_{[a,b]}$ is precisely that collection of immersion pairs.

Pick distinct points $p^q,p^{q'}\in\crit\alpha_a$. An immersion pair $(c,c')$ is \emph{visible} if there are disjoint paths $f^q$, $f^{q'}$ in $\crit\alpha\setminus\iota$ from $p^q$ to $q\in c$ and from $p^{q'}$ to $q'\in c'$, respectively, with positive $t$-derivative. Since $\iota_a=\emptyset$, this definition does not depend on $p^q,p^{q'}$. Viewing $\alpha$ as merely a homotopy of oriented maps $S^1\to S^2$, the new homotopy in which $q$ and $q'$ have been pushed along $f^q$ and $f^{q'}$ to lie in $\crit\alpha_a$ yields the fact that a visible pair is never a mixed pair ($q,q'$ would correspond to a Reidemeister-II move applied to a piecewise-smoothly embedded circle in $S^2$, which cannot be mixed).

\emph{Case 1:} \hypertarget{isolatecase1} Suppose there is a visible genus-increasing pair $(c,c')$. Using Lemma \ref{reidcusp}(\hyperlink{reidcusp(1a)}{1a}) if necessary and Lemma \ref{isotopecrit}, push $q$ and $q'$ along $f^q$ and $f^{q'}$ to lie in $\crit\alpha_a$. Insert the detour from Figures \ref{gipcbd1}-\hyperlink{39b}{40b} and regard the isolated pseudostabilization in Figure \hyperlink{39b}{40b} as part of $\alpha_{[0,a)}$, so that $(c,c')$ is now a single and the first immersion circle to appear in $\iota_{[a,b]}$; call that single $d$ and let $I_d\subset[a,b]$ denote the collection of $t$-coordinates of points in $d$. 
\begin{figure}\hypertarget{45b}\capstart%
	\labellist
	\small\hair 2pt
	\pinlabel $d$ at 7 7
	\pinlabel $a$ at 9 92
	\pinlabel $a'$ at 21 52
	\pinlabel $^\ast$ at 182 41
	\pinlabel $\nearrow$ at 8 35
	\pinlabel $\leftarrow$ at 12 52
	\pinlabel $\rightarrow$ at 14 96
	\pinlabel $\nwarrow$ at 258 78
	\pinlabel $\nearrow$ at 257 27
	\endlabellist
	\centering\capstart	
	\subfloat[\ ]{\label{mixed1}\includegraphics{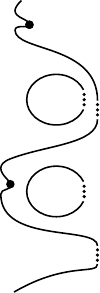}}\hspace{1cm}
	\subfloat[\ ]{\label{mixed2}\includegraphics{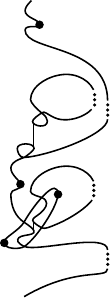}}\hspace{1cm}
	\subfloat[\ ]{\label{mixed3}\includegraphics{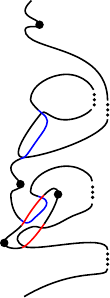}}\hspace{1cm}
	\subfloat[\ ]{\label{mixed4}\includegraphics{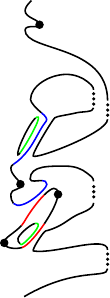}}
	\caption{Eliminating divided circles in Case \protect\hyperlink{isolatecase1a}{1a} of the proof of Lemma \ref{isolate}. The top and bottom of each figure is identified to form a small part of $\crit\alpha$. Arrows indicate the direction of decreasing genus. The bigon lemma is applied to the starred lune.}\label{mixed}
\end{figure}

\emph{Case 1a:}\ \hypertarget{isolatecase1a}Suppose $d$ is isolated. Then, as in Figure \ref{mixed1}, there may be \emph{divided} immersion pairs $(a,a')$, which means $a$ is on the higher-genus side of $d$ and $a'$ is on the lower-genus side of $d$ (assume $(a,a')$ is the first such pair to appear).  Generally a divided pair is not isolated, but the modifications in this case occur in an arbitrarily small neighborhood of the $t$-value at which the pair $(a,a')$ appears, so the present figure will suffice. The orientations of $a$ and $a'$ come from considering the immersion just before their appearance: This stage of the deformation originally came from a genus-increasing $R_2$ deformation applied to an SPWF, and an intersection between one side of the resulting bigon and one of the outer loops necessarily produces the depicted orientations. This step includes $(a,a')$ into $d$ without creating further divided or non-isolated circles. Figure \ref{mixed1} came from the previous paragraph. Figure \hyperlink{45b}{45b} comes from two applications of Lemma \ref{swalpushlem}. Figure \hyperlink{45b}{47c} comes from an application of the Pinching Lemma \ref{r3moves}(\hyperlink{r3moves(1)}{1}) followed by Corollary \ref{simplecancelcor} to eliminate canceling $T$-critical points. Figure \hyperlink{45b}{45d} comes from an application of the Bigon Lemma \ref{r3moves}(\hyperlink{r3moves(3)}{3}).

\emph{Case 1b:}\ \hypertarget{isolatecase1b} Suppose $d$ is not isolated. Figure \ref{sewing} gives a way to eliminate the first crossing in $d$. First use Lemma \ref{swalpushlem} to obtain Figure \hyperlink{46b}{46b}, then Lemma \ref{isotopecrit} for Figure \hyperlink{46b}{46c}, then apply the Bigon Lemma \ref{r3moves}(\hyperlink{r3moves(3)}{3}) to the tiny bigon in the center of that figure to obtain Figure \hyperlink{46b}{46d}. The placement of the topmost crossing is mostly irrelevant: If it was more properly depicted to have been on the lower-genus side of $d$, the same steps produce Figure \hyperlink{46b}{46e} instead. If that crossing was a self-crossing, then this modification increases the number of immersion pairs by 1. In that case, Figures \hyperlink{46b}{46d} and \hyperlink{46b}{46e} depict one of those pairs as separated by $d$ so that, once $d$ is isolated, and if that crossing was a self-crossing in a circle other than $d$, the modification of Figure \ref{mixed} is available to absorb that pair into $d$ without adding crossings. Finally, if any crossing in Figure \ref{sewing1} is a self-crossing in $d$, then all three crossings are self-crossings because $d$ is a single, which forces the secondary arcs of that crossing to also be in $d$. In that case, tracing orientations as in Remark \ref{tracingorientations} from the orientation of the swallowtail-containing arc in the lower two crossings in Figure \ref{sewing1} shows that the orientations of the arcs in the upper crossing are genus-decreasing from left to right, so that if the arc pair $(2,2')$ in Figure \hyperlink{46b}{46d} is in $d$ then the arc pair $(1,1')$ that just split off from $d$ is contained in an isolated pair because the crossing we just resolved was the first crossing in $d$. In Figure \hyperlink{46b}{46e}, if $(3,3')$ is in $d$ then $(4,4')$ (which similarly cannot be mixed) qualifies for \hyperlink{isolatecase1}{1} as soon as the algorithm is finished with what is left of $d$. If it turned out that $(2,2')$ or $(3,3')$ were not in $d$, then $(1,1')$ or $(4,4')$ are in $d$ respectively; do nothing to $(2,2')$ or $(3,3')$ until $d$ is isolated. If they are still not in $d$ at that point, then they are divided and qualify for Case \hyperlink{isolatecase1a}{1a}. This shows inductively that it is possible to isolate $d$ without increasing the number of non-isolated immersion circles. Moreover, if $g\co(\crit\alpha\setminus\iota)_{[a,b]}\to\N_{\geq2}$ is the fiber genus function, this process does not produce $g$-minimizing regions other than those bounded by isolated circles: If a self-crossing in $d$ is resolved, then the bigon move splits off a circle from $d$ which is either isolated or later re-integrated into $d$ in the course of isolating $d$ or by Case \hyperlink{isolatecase1a}{1a}. If the crossing was not a self-crossing in $d$, then the bigon move merges distinct circles, which could only decrease the number of regions ($g$-minimizing or not). Once this is accomplished, continue the algorithm from the beginning.
\begin{figure}\hypertarget{46b}\capstart%
	\labellist
	\small\hair 2pt
	\pinlabel 1 at 237 97
	\pinlabel 1' at 240 73
	\pinlabel 2 at 261 97
	\pinlabel 2' at 262 50
	\pinlabel 3 at 319 137
	\pinlabel 3' at 313 73
	\pinlabel 4 at 333 130
	\pinlabel 4' at 333 50
	\endlabellist
	\centering\capstart	
	\subfloat[\ ]{\label{sewing1}\includegraphics{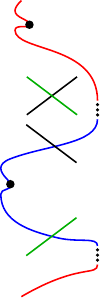}}\hspace{1cm}
	\subfloat[\ ]{\label{sewing2}\includegraphics{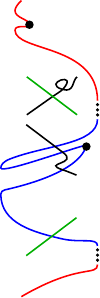}}\hspace{1cm}
	\subfloat[\ ]{\label{sewing3}\includegraphics{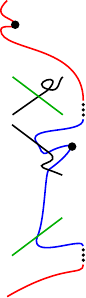}}\hspace{1cm}
	\subfloat[\ ]{\label{sewing4}\includegraphics{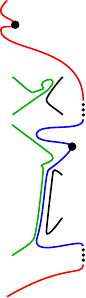}}\hspace{1cm}
	\subfloat[\ ]{\label{sewing5}\includegraphics{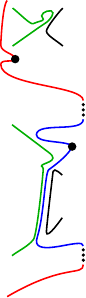}}
	\caption{Eliminating crossings from an initial genus-increasing pair. Counterparts are given like colors. The crossing not contained in $d$ may have been on the lower-genus side of $d$ instead; \protect\hyperlink{46b}{46e} is the resulting diagram.}\label{sewing}
\end{figure}

\emph{Case 2:}\ \hypertarget{isolatecase2} Consider the fiber genus function $g\co(\crit\alpha\setminus\iota)_{[a,b]}\to\N_{\geq2}$ and the following four moves.\begin{itemize}
 \item[(a)] A triangle move in which $\Delta$ lies on the lower-genus side of each of its three sides, 
 \item[(b)] A finger move in which the two primary immersion arcs approach each other's lower-genus sides,
 \item[(c)] The detour consisting of a flipping move immediately followed by its reverse,
 \item[(d)] The modification depicted in Figure~\ref{swalpush}.\end{itemize}
Tracing orientations as in Remark \ref{tracingorientations}, or simply examining appropriate base diagrams, shows that, in these moves, the movements of all primary and secondary arcs are genus-increasing. This implies that all new path components of $\crit\alpha\setminus\iota$ introduced by the move (a), that is, the three secondary bisected lunes created in the bottom right of Figure \ref{r3moves_2_fig}, have their interiors on the higher-genus sides of their boundaries. If $\Delta$ is one of the path components of $(\crit\alpha\setminus\iota)_{[a,b]}$ on which $g$ realizes its minimum value, then move (a) decreases the number of such components by 1. 
Referring to Figure \hyperlink{44b}{44b}, the new path components of $(\crit\alpha\setminus\iota)_{[a,b]}$ introduced by move (b) are as follows: a primary lune whose interior lies on the higher-genus side of its boundary and an embedded genus-increasing secondary pair of immersion circles, where each circle is bisected by the secondary arcs. None of these five regions minimize $g$. Move (c) introduces only an isolated genus-increasing single. Move (d) produces a small genus-increasing 1-gon in $\iota$. For these reasons, the only way these moves can increase the number of $g$-minimizing path components of $(\crit\alpha\setminus\iota)_{[a,b]}$ is to perform move (b) on a pair of arcs chosen from the boundary of a $g$-minimizing path component of $(\crit\alpha\setminus\iota)_{[a,b]}$.

If no visible pairs are genus-increasing and $\iota_{[a,b]}$ is not yet isolated, then there is a $g$-minimizing path component $A\subset(\crit\alpha\setminus\iota)_{[a,b]}$. Because it is $g$-minimizing, $A$ lies on the lower-genus side of its boundary. If $A$ is a triangle, perform move (a) to remove $A$.

If $A$ has more than three corners, perform moves (b) between nonconsecutive boundary arcs of $A$ until $A$ is cut into triangles and apply move (a) to each triangle. If it is not possible to do this because it would require a finger move between an immersion arc and its counterpart, in such a situation there is a loop in the critical image as in Figure \ref{2a0fig1}, and the detour shown there converts the pair containing $z$ and $z'$ into a single. Push the first swallowtail to $t=a$, convert it using Figure \ref{giscbd} and apply Case \hyperlink{isolatecase1b}{1b} to it (note the modifications in Figure \ref{sewing} only requires one swallowtail near $t=a$, and separated pairs can be reabsorbed by Case \hyperlink{isolatecase1a}{1a} once the single is isolated). The detour of Figure \ref{2a0fig2} splits $A$ into two $g$-minimizing regions and the push to $t=a$ possibly divides those regions further, but then the conversion of Figure \ref{giscbd} causes all of those regions to lie adjacent to a single which is then isolated, in sum decreasing the number of $g$-minimizing regions.
\begin{figure}\capstart%
	\centering
	\labellist
	\small\hair 2pt
	\pinlabel $z$ at 6 5
	\pinlabel $z'$ at 37 6
	\pinlabel $\nearrow$ at 170 37
	\pinlabel $\rightarrow$ at 3 33
	\endlabellist
	\subfloat[The configuration of immersion arcs on the left corresponds to the base diagram on the right, undergoing a trivial deformation.]{\makebox[11cm][c]{{\label{2a0fig1}\includegraphics{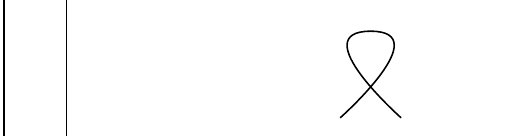}}}}\\
	\subfloat[The detour has base diagrams on the right consisting of a flip, a Reidemeister-II fold crossing, and the reverse.]{\makebox[11cm][c]{{\label{2a0fig2}\includegraphics{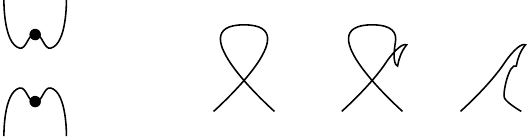}}}}
	\caption{When a finger move does not exist for Case \protect\hyperlink{isolatecase2a(0)}{2a (0 chords)}, a detour turns the offending pair into a single, almost immediately eligible for Case \protect\hyperlink{isolatecase1}{1}.}\label{2a0fig}
\end{figure}

If $A$ has two corners, perform moves (c) and (d) near the counterpart of one of its sides to turn $A$ into a triangle, then perform move (a) on that triangle.

If $A$ has one corner, apply the two-corner case with one extra (d) move, pushing the swallowtail back and forth across the relevant immersion arc to give two more corners to $A$. Then apply move (a).

If $A$ has no corners, then $A$ is bounded by isolated circles. Omit these circles from consideration when finding $g$-minimizing regions (passing an immersion arc over an isolated circle does not increase the number of crossings or $g$-minimizing regions with corners).

In this way, it is possible to eliminate $g$-minimizing regions until a genus-increasing visible pair appears (this is bound to happen because, eventually, the region containing $\crit\alpha_a$ becomes $g$-minimizing, at which point all visible pairs are genus-increasing).
\end{proof}

\begin{df}An isolated immersion circle $c$ is \emph{nested} if it or its counterpart bounds a disk in $\crit\alpha$ whose interior contains an immersion circle $d$ (for short, just say $c$ or $c'$ contains $d$). In this case, $d$ is also nested.\end{df}

\begin{lemma}\label{evertlemma}Suppose $\iota$ consists of isolated circles. Then, by a homotopy of $\alpha$, it is possible to arrange for all immersion circles to come from either pseudostabilizations or non-nested genus-decreasing immersion pairs.\end{lemma}
\begin{proof}The first step is to explain how to \emph{evert} a genus-decreasing circle.

\begin{figure}\hypertarget{48b}\capstart
	\centering
	\subfloat[\ ]{\label{evert1}\includegraphics{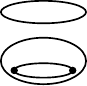}}\qquad
	\subfloat[\ ]{\label{evert2}\includegraphics{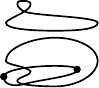}}\qquad
	\subfloat[\ ]{\label{evert3}\includegraphics{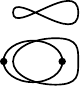}}\qquad
	\subfloat[\ ]{\label{evert4}\includegraphics{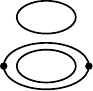}}\qquad
	\caption{Everting one circle of a bounding pair for Step \protect\hyperlink{unlinklemma5}{5} of the proof of Lemma \ref{unlinklemma}. This changes the bounding pair at left into a mixed pair (the top circle becomes genus-increasing).}\label{evert}
\end{figure} The moves in Figure \ref{evert} proceed from left to right, and mostly follow from previous arguments:
\begin{enumerate}
\item[\ref{evert1}] comes from inserting a detour given by a well-placed flip followed by its reverse, 
\item[\hyperlink{48b}{48b}] is an instance of Figure \ref{swalpush}, 
\item[\hyperlink{48b}{48c}] is an application of Lemma \ref{isotopecrit}, and 
\item[\hyperlink{48b}{48d}] is somewhat new. Figure \ref{contractlune} has the base diagrams corresponding to the relevant parts of Figures \hyperlink{48b}{48c}-\hyperlink{48b}{48d}. The move is merely a homotopy of $\alpha$ in which the ball on which the inverse flipping move is supported moves foreward in $t$. This movement exists because the existence of the required path of regular points exists, which in turn exist because of the disjointness condition that allows the $R_2$ deformation in Figure \ref{contractlune1}.
\end{enumerate}
\begin{figure}\hypertarget{49b}\capstart
	\centering
	\subfloat[\ ]{\label{contractlune1}\includegraphics{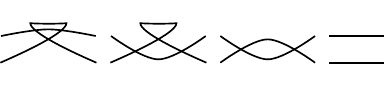}}\hspace{1.5cm}
	\subfloat[\ ]{\label{contractlune1a}\includegraphics{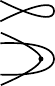}}\\
	\subfloat[\ ]{\label{contractlune2}\includegraphics{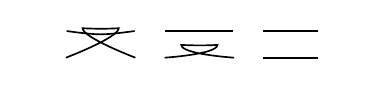}}\hspace{1.5cm}
	\subfloat[\ ]{\label{contractlune2a}\includegraphics{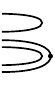}}
	\caption{Base diagrams for part of Figures \protect\hyperlink{48b}{48c} and \protect\hyperlink{48b}{48d}.}\label{contractlune}
\end{figure}
With this understood, the first modification is to use the Pinching Lemma \ref{r3moves}(\hyperlink{r3moves(1)}{1}) on any immersion single necessary and Lemma \ref{isotopecrit} to arrange for each single to occupy an interval in $t$ which is occupied by no other points of $\iota$. Now $\iota$ is a number of collections of isolated immersion pairs, with pseudostabilizations between each pair of collections.

The next step is to pinch all the immersion pairs to be disjoint from the cusp locus: Given a cusp arc with intersections, its intersections are partitioned into pairs that cancel up to isotopy of $\iota$ because the immersion circles are embedded. There is an \emph{innermost} pair in the sense that the cusp arc connecting them is otherwise disjoint from $\iota$ because the circles are disjoint. Use Lemmas \ref{isotopecrit} and \ref{reidcusp}(\hyperlink{reidcusp(1b)}{1b}) to cancel that pair of intersections and proceed inductively. With this completed, now follows an algorithm.\\

\noindent(1) Consider the first immersion pair. It is either genus-increasing or genus-decreasing, nested or non-nested; proceed to (2), (3) or (4) as appropriate. After applying one of the cases below, eliminate it from consideration and return to (1). Repeat until the conclusion of the Lemma is satisfied.\\

\noindent(2) If it is genus-increasing, convert it as in Figure \ref{gicbd}.\\

\noindent(3) If it is genus-decreasing and non-nested, use Lemma \ref{isotopecrit} to contract it so that it lies in a $t$-interval with no other immersion points.\\

\noindent(4) If the first pair $(c,c')$ is genus-decreasing and nested, consider the first pair $(d,d')$ nested inside $(c,c')$ and proceed to (4a), (4b) or (4c) as appropriate.\\

{\narrower\noindent(4a) If $d$ is inside $c$ and $d'$ is inside $c'$, then $(d,d')$ is genus-increasing (it is not mixed, as easily verified using base diagrams, and it is not genus-increasing because otherwise there would be disconnected fibers). Use the Pinching Lemma \ref{r3moves}(\hyperlink{r3moves(1)}{1}) on the two pairs to merge them into a single non-nested genus-decreasing pair and apply (3).\\ \par}

{\narrower\noindent(4b) If $d$ is inside $c$ and $d'$ is not inside $c$ and not inside $c'$, then it is straightforward to verify using base diagrams that $(d,d')$ is mixed, with $d$ genus-increasing and $d'$ genus-decreasing. Use the method from Figure \ref{mixed} to merge the circles (here, $(d,d')$ plays the role of $(a,a')$ in that figure, and note that the modifications there do not require $(c,c')$ to be a single). This creates an additional single and one mixed pair. Convert the single as in Figure \ref{giscbd} and evert the mixed pair, then convert the resulting singles as in FIgure \ref{giscbd} and the genus-increasing pair as in Figure \ref{gicbd}.\\ \par}

{\narrower\noindent(4c) If $c$ contains $d$ and $d'$, then $(d,d')$ is genus-iincreasing because otherwise, as verified using base diagrams, there would be disconnected fibers. Let $C$ be the disk bounded by $c$, let $D$ be the union of the two disks bounded by $(d,d')$, and let $F$ be the union of disks bounded by circles in $C\setminus D$. Use the Pinching Lemma \ref{r3moves}(\hyperlink{r3moves(1)}{1}) to pinch $(c,c')$ into smaller circles such that one of these circles (call it $e$) contains $D$ but no points of $F$, while its counterpart $e'$ contains no circles. Merge $d$ and $d'$ as in Figure \ref{2a0fig} and use the resulting single to evert $e'$. Convert the single and apply the method from Figure \ref{mixed} to the converted single and $(e,e')$, like was done in (4b).\\ \par}

\noindent Each iteration decreases the number of pairs that are nested or not genus-decreasing, so the algorithm eventually yields the required form for $\iota$.\end{proof}

\begin{lemma}\label{pseudostabilizations} By a homotopy of $\alpha$ is it possible to turn any pseudostabilization into deformation consisting of one stabilization and a collection of handleslides.\end{lemma}
\begin{proof}Given a pseudostabilization single $d$, use Corollary \ref{simplecancelcor} repeatedly until $T|_d$ has four extrema: Two coming from swallowtails $s_1$, $s_2\in d$ and the other two coming from a Reidemeister-II fold crossing. Next, use Lemma \ref{reidcusp} to remove all intersections between one component of $d\setminus\{s_1,s_2\}$ and the cusp locus (first push the swallowtails to lie in the same strip of fold points, then pinch along the cusps as in the paragraph just above (1) in the proof of Lemma \ref{evertlemma}. This converts $d$ into a stabilization single and pinches off a number of non-nested genus-decreasing pairs. Because their images lie in distinct regions of $S^2_{[0,1]}$, these pairs may be separated into distinct intervals in $t$ using Lemma \ref{isotopecrit}.\end{proof}

\begin{proof}[Proof of Theorem \ref{T}]\label{Tproof}The results of Section \ref{moves} that give the correspondence between model deformations and moves on surface diagrams show that a deformation consists of a sequence of model deformations exactly when its decorated critical surface is a concatenation of decorated critical surfaces whose decorations are characteristic of model deformations. The lemmas of Section \ref{proof} yield a deformation whose decorated critical surface is such a concatenation.\end{proof}


\begin{thebibliography}{9999}
\bibitem[AK]{AK}S. Akbulut and \c{C}. Karakurt, \MYhref{http://gokovagt.org/journal/2008/akbulutkarakurt.html}{Every 4-manifold is BLF}, \emph{J. G\"{o}kova Geom. Topol.} \textbf{2} (2008) 40-82.
\bibitem[ADK]{ADK}D. Auroux, S. Donaldson, and L. Katzarkov, \MYhref{http://msp.warwick.ac.uk/gt/2005/09/p024.xhtml}{Singular Lefschetz pencils}, \emph{Geom. Topol.} \textbf{9} (2005), 1043-1114. doi:10.2140/gt.2005.9.1043
\bibitem[B1]{B1}R. I. Baykur, \MYhref{http://msp.org/pjm/2009/240-2/p01.xhtml}{Topology of broken Lefschetz fibrations and near-symplectic 4-manifolds}, \emph{Pac. J. Math} \textbf{240} No. 2 (2009), 201-230. doi:10.2140/pjm.2009.240.201
\bibitem[B2]{B2}R. I. Baykur, \MYhref{http://imrn.oxfordjournals.org/content/2008/rnn101.abstract}{Existence of broken Lefschetz fibrations}, \emph{Int. Math. Res. Notices} \textbf{2008} (2008). doi:10.1093/imrn/rnn101
\bibitem[Be]{Be}S. Behrens, \MYhref{http://dx.doi.org/10.2140/pjm.2013.264.257}{On 4-manifolds, folds and cusps}, \emph{Pac. J. Math} \textbf{264}(2) (2013), 257-306. doi:10.2140/pjm.2013.264.257
\bibitem[BH]{BH}S. Behrens and K. Hayano, \MYhref{http://arxiv.org/abs/1210.5948}{Elimination of cusps in dimension 4 and its applications}, \emph{Proc. London Math Soc.} {\bf 113} (5), 674-724. doi:10.1112/plms/pdw042
\bibitem[DS]{DS}S. K. Donaldson and I. Smith, \MYhref{http://dx.doi.org/10.1016/S0040-9383(02)00024-1}{Lefschetz pencils and the canonical class for symplectic 4-manifolds}, \emph{Topology} \textbf{42} (2003), 743-785. doi:10.1016/S0040-9383(02)00024-1
\bibitem[FM]{FM}B. Farb and D. Margalit, \MYhref{http://press.princeton.edu/titles/9495.html}{A Primer on Mapping Class Groups}, Princeton Mathematical Series \textbf{49}, Princeton University Press, Princeton, NJ 2011.
\bibitem[GK1]{GK1}D. Gay and R. Kirby, \MYhref{http://arxiv.org/abs/1102.0750}{Indefinite Morse 2-functions; broken fibrations and generalizations}, Geom. Topol. \textbf{19} (2015), 2465–2534. DOI: 10.2140/gt.2015.19.2465
\bibitem[GK2]{GK2}D. Gay and R. Kirby, \MYhref{http://dx.doi.org/10.1073/pnas.1018465108}{Fiber connected, indefinite Morse 2-functions on connected $n$-manifolds}, \emph{Proc. Nat. Ac. Sci.} \textbf{108}, no. 20 (2011), 8122-8125.
\bibitem[GS]{GS}R. Gompf and A. Stipsicz, \MYhref{http://www.ams.org/bookstore-getitem/item=GSM-20}{4-manifolds and Kirby Calculus}, Graduate Studies in Math. \textbf{20}, Amer. Math. Soc., Providence, RI 1999.
\bibitem[H]{H}K. Hayano, \MYhref{http://arxiv.org/abs/1203.4299}{Modification rule of monodromies in $R_2$-move}, Alg. Geom. Topol. \textbf{14} (2014), 2181–2222. doi: 10.2140/agt.2014.14.2181
\bibitem[KMT]{KMT}R. Kirby, P. Melvin, and P. Teichner, \MYhref{http://msp.warwick.ac.uk/gtm/2012/18/p009.xhtml}{Cohomotopy sets of 4-manifolds}, \emph{Geom. Topol. Monographs} \textbf{18} (2012), 161-190. doi:10.2140/gtm.2012.18.161
\bibitem[L1]{L1}Y. Lekili, \MYhref{http://msp.warwick.ac.uk/gt/2009/13-01/p008.xhtml}{Wrinkled fibrations on near-symplectic= manifolds}, \emph{Geom. Topol.} \textbf{13} (2009), 277-318. doi:10.2140/gt.2009.13.277
\bibitem[L2]{L2}Y. Lekili, \MYhref{http://arxiv.org/abs/0903.1773}{Heegaard--Floer homology of broken fibrations over the circle}, Adv. Math. \textbf{244} (2013), 268–302. doi:10.1016/j.aim.2013.05.013
\bibitem[Lev]{Lev}H. Levine, Elimination of cusps, \emph{Topology} \textbf{3}, Suppl. 2 (1965), 263-296. doi:10.1016/0040-9383(65)90078-9
\bibitem[M]{M}J. Morgan, \MYhref{http://press.princeton.edu/titles/5866.html}{The Seiberg-Witten equations and applications to the topology of smooth four-manifolds}, Mathematical Notes \textbf{44} (1996), Princeton, NJ: Princeton University Press, pp. viii+128.
\bibitem[T1]{T1}C. Taubes, \MYhref{http://projecteuclid.org/euclid.jdg/1214459411}{Counting pseudo-holomorphic submanifolds in dimension 4}, \emph{J. Diff. Geom.} \textbf{44} no. 4 (1996), 818-893. \MYhref{http://www.ams.org/mathscinet-getitem?mr=1438194}{MR1438194}
\bibitem[T2]{T2}C. Taubes, \MYhref{http://dx.doi.org/10.2140/gt.1999.3.167}{Seiberg-Witten invariants and pseudo-holomorphic subvarieties for self-dual, harmonic 2-forms}, \emph{Geom. Topol.} \textbf{3} (1999), 167-210. doi:10.2140/gt.1999.3.167
\bibitem[U]{U}M. Usher, \MYhref{http://dx.doi.org/10.2140/gt.2004.8.565}{The Gromov invariant and the Donaldson-Smith standard surface count}, \emph{Geom. Topol.} \textbf{8} (2004), 565-610. doi:10.2140/gt.2004.8.565
\bibitem[Wa]{Wa}G. Wassermann, \MYhref{http://dx.doi.org/10.1007/BF02392016}{Stability of unfoldings in space and time}, \emph{Acta. Math.}, \textbf{135} (1975), 58-128. doi:10.1007/BF02392016
\bibitem[W1]{W1}J. Williams, \MYhref{http://dx.doi.org/10.2140/gt.2010.14.1015}{The $h$-principle for broken Lefschetz fibrations}, \emph{Geom. Topol.} \textbf{14} no. 2 (2010), 1015-1061. doi:10.2140/gt.2010.14.1015
\bibitem[W2]{W2}J. Williams, \MYhref{http://arxiv.org/abs/1411.1742}{Existence of 2-parameter crossings}, 2011 preprint.
\bibitem[W3]{W3}J. Williams, \MYhref{http://arxiv.org/abs/1310.2969}{Holomorphic polygons and smooth 4-manifold invariants}, 2013 preprint.
\end{thebibliography}
\end{document}